\documentclass{article}
\usepackage[utf8]{inputenc}

\title{Categories of singularities of invertible polynomials}
\author{Oleksandr Kravets}
\date{}

\usepackage{amsthm}
\usepackage{amsmath}
\usepackage{amsfonts}
\usepackage{amssymb}

\usepackage{breqn}
\usepackage{extarrows}

\usepackage{enumitem}

\usepackage{tikz}
\usepackage{tikz-cd}
\usepackage{tikz-3dplot}

\usepackage{hhline}

\usepackage{stackrel}

\usepackage{hyperref}
\usepackage{float} 
\usepackage{framed}
\usepackage{adjustbox}

\usepackage[style=alphabetic]{biblatex}
\usepackage{calc}

\usepackage[shortcuts]{extdash} 

\newtheorem{theorem}{Theorem}[section]
\newtheorem*{theorem*}{Theorem}
\newtheorem{prop}[theorem]{Proposition}
\newtheorem{lemma}[theorem]{Lemma}
\newtheorem{corollary}[theorem]{Corollary}
\newtheorem{conjecture}[theorem]{Conjecture}
\newtheorem*{conjecture*}{Conjecture}

\newtheorem{Theorem}{Theorem}

\newtheorem*{ConjectureO}{Orlov's Conjecture}

\theoremstyle{definition}
\newtheorem{definition}[theorem]{Definition}
\newtheorem*{definition*}{Definition}
\newtheorem{remark}[theorem]{Remark}
\newtheorem{recipe}[theorem]{Recipe}
\newtheorem*{remark*}{Remark}
\newtheorem{example}[theorem]{Example}
\newtheorem*{example*}{Example}

\numberwithin{equation}{section}

\DeclareMathOperator{\Hom}{Hom}
\DeclareMathOperator{\Ext}{Ext}
\def\HH{H}
\DeclareMathOperator{\Cone}{Cone}
\DeclareMathOperator{\Ker}{Ker}

\DeclareMathOperator{\krdim}{kr.dim}
\DeclareMathOperator{\Perf}{Perf}
\DeclareMathOperator{\id}{id}

\DeclareMathOperator{\ev}{ev}
\DeclareMathOperator{\coev}{coev}

\def\kk{{\mathbf{k}}}
\def\Cc{{\mathbf{C}}}
\def\ZZ{{\mathbb{Z}}}
\def\:{\colon}
\def\w{{\vec w}}
\def\Rw{{R/(w)}}

\def\op{\oplus}
\def\la{\left<}
\def\ra{\right>}
\def\<{\left<}
\def\>{\right>}
\def\phi{\varphi}
\def\eps{\varepsilon}

\def\TT{{\mathcal T}}
\def\AA{{\mathcal A}}
\def\BB{{\mathcal B}}
\def\CC{{\mathcal C}}
\def\DD{{\mathcal D}}

\newenvironment{mycube}{
\tdplotsetmaincoords{60}{20}%
\begin{tikzpicture}
		[scale=2, tdplot_main_coords,
			cube/.style={very thin,fill=red, opacity=.5},
            cube front edge/.style={very thin, opacity=.5},
            cube hidden/.style={very thin,dashed},
            grid dotted/.style={very thin,densely dotted},
            dots/.style={thick,loosely dotted},
            faceB/.style={thick,black,fill=blue, opacity = .5},
            faceG/.style={thick,black,fill=green, opacity = .5},
            faceY/.style={thick,black,fill=yellow, opacity = .5},
			grid/.style={very thin,gray},
			axis/.style={->,very thin},
            label arrow/.style={->, very thin, >=stealth', shorten >=2pt, shorten <=-2pt},
            arrow/.style={->,>=stealth',shorten >=2pt,semithick,black},
            arrow bold/.style={->, >=stealth', shorten >= 2pt, shorten <= 2pt, thick, black},
        ]
}{
\end{tikzpicture}
}

\newcommand*{\drawCoordinateAxes}[3]
{	
	\draw[axis,tdplot_main_coords] (#1,#2,#3) -- (#1+0.5,#2,#3) node[anchor=west]{$x$};
	\draw[axis,tdplot_main_coords] (#1,#2,#3) -- (#1,#2+0.5,#3) node[anchor=south]{$y$};
	\draw[axis,tdplot_main_coords] (#1,#2,#3) -- (#1,#2,#3+0.5) node[anchor=east]{$z$};
}

\newcommand{\boldCircle}{[black,thick,fill] circle [radius=0.5pt]}
\newcommand{\denoteBoldObject}[3]{\draw #1 \boldCircle node [anchor = #2] {{#3}};}
\newcommand{\denoteBoldObjectWithArrow}[4]{\draw [label arrow, darkgray, densely dashed] #1 node[anchor = #2, black]{{#3}} -- #4;
\draw #4 \boldCircle ;
}

\newcommand{\denoteCommonRedObjects}
{
    \denoteBoldObject{(2,2,0.5)}{west}{$\kk(-(p-2)\x-(q-2)\y)$}
    \denoteBoldObjectWithArrow{(2.25,2,1.25)}{west}
        {$\kk(-(p-2)\x-(r-2)\z)$}{(2,0.5,2)}
    \denoteBoldObject{(2,2,2)}{west}{$\kk(-(p-2)\x-(q-2)\y-(r-2)\z)$}
    \denoteBoldObject{(0.5,2,2)}{south west}{$\kk(-(q-2)\y-(r-2)\z)$}
}

\newcommand{\denoteManyRedObjects}
{
    \denoteBoldObject{(0.5,0.5,0.5)}{north east}{$\underline{\kk}$}
    \denoteBoldObject{(1,0.5,0.5)}{north}{$\kk(-\x)$}
    \denoteBoldObject{(2,0.5,0.5)}{west}{$\kk(-(p-2)\x)$}
    \denoteBoldObject{(2,1,0.5)}{west}{$\kk(-(p-2)\x-\y)$}

    \denoteCommonRedObjects

    \denoteBoldObject{(0.5,0.5,2)}{east}{$\kk(-(r-2)\z)$}
    \denoteBoldObject{(0.5,0.5,1)}{east}{$\kk(-\z)$}
    \denoteBoldObjectWithArrow{(0.25,0.5,0.6)}{east}
        {$\kk(-\y)$}{(0.5,1,0.5)}
}

\newcommand{\denoteMediumObjectsChain}
{
    \denoteCommonRedObjects
    \denoteBoldObject{(0.5,0.5,2)}{south east}{$\kk(-(r-2)\z)$}
    
    \denoteBoldObject{(0,0,2)}{east}{$\My(\x-(r-2)\z)$}
    \denoteBoldObjectWithArrow{(-0.5,0,1.5)}{east}{$\underline{\kk}$}{(0.5,0.5,0.5)}
    \denoteBoldObjectWithArrow{(-0.5,0,1)}{east}{$\lefteqn{\underline{\phantom{\Mz}}}\My$}{(0.5,0,0.5)}
    \denoteBoldObject{(0,0,0.5)}{east}{$\My(\x)$}
        
    \denoteBoldObject{(0,0,0)}{east}{$\Mz(\y+\x)$}

    \denoteBoldObject{(2,0.5,0)}{west}{$\Mz(-(p-2)\x)$}
    \denoteBoldObject{(1.5,0,0)}{north west}{$\Mz(-(p-3)\x+\y)$}
    \denoteBoldObject{(2,2,0)}{west}{$\Mz(-(p-2)\x-(q-2)\y)$}
    \denoteBoldObjectWithArrow{(1,-0.25,0)}{north}{$\underline{\Mz}$}{(0.5,0.5,0)}
    \denoteBoldObject{(0.5,0,0)}{north east}{$\Mz(\y)$}
    \denoteBoldObjectWithArrow{(2.25,1.25,0)}{west}{$\My(-(p-2)\x)$}{(2,0,0.5)}
}

\newcommand{\denoteMediumObjectsLoop}
{
    \denoteCommonRedObjects
    
    \denoteBoldObject{(0,0,2)}{east}{$\My(\x-(r-2)\z)$}
    \denoteBoldObjectWithArrow{(-0.25,0,1.5)}{east}{$\underline{\kk}$}{(0.5,0.5,0.5)}
    \denoteBoldObjectWithArrow{(0.25,-0.25,0)}{north east}{$\lefteqn{\underline{\phantom{\Mz}}}\My$}{(0.5,0,0.5)}
    \denoteBoldObject{(0,0,0.5)}{east}{$\My(\x)$}
        
    \denoteBoldObject{(2,0.5,0)}{west}{$\Mz(-(p-2)\x)$}
    \denoteBoldObject{(2,0,0)}{west}{$\Mz(-(p-2)\x+\y)$}
    \denoteBoldObject{(2,2,0)}{west}{$\Mz(-(p-2)\x-(q-2)\y)$}
    \denoteBoldObjectWithArrow{(1,-0.25,0)}{north west}{$\underline{\Mz}$}{(0.5,0.5,0)}
    \denoteBoldObjectWithArrow{(2.25,1.25,0)}{west}{$\My(-(p-2)\x)$}{(2,0,0.5)}

    \denoteBoldObject{(0,0,0)}{east}{$\Mxyz(\x+\y+\z)$}

    \denoteBoldObject{(0.5,0,0)}{north}{$\Mz(\y)$}

    \denoteBoldObject{(0,2,2)}{east}{$\Mx(-(q-2)\y-(r-2)\z)$}
    \denoteBoldObject{(0,0.5,2)}{south east}{$\Mx(-(r-2)\z)$}
    \denoteBoldObjectWithArrow{(-0.5,0,1)}{east}
        {$\underline{\Mx}$}{(0,0.5,0.5)}

}

\newcommand{\denoteManyRedObjectsSplitChain}
{
    \denoteBoldObject{(2,0.5,0.5)}{west}{$\kk(-(p-2)\x)$}
    \denoteBoldObject{(2,1,0.5)}{west}{$\kk(-(p-2)\x-\y)$}

    \denoteCommonRedObjects

    \denoteBoldObject{(0.5,0.5,2)}{south east}{$\kk(-(r-2)\z)$}
}

\newcommand{\denoteManyObjectsSplitChain}
{
    \denoteBoldObject{(2,0,0.5)}{north west}
        {$\My(-(p-2)\x)$}

    \denoteManyRedObjectsSplitChain

    \denoteBoldObject{(0,0,2)}{east}{$\My(\x-(r-2)\z)$}
    \denoteBoldObjectWithArrow{(-0.2,0.1,1.5)}{east}
        {$\My(-(r-2)\z)$}{(0.5,0,2)}
    \denoteBoldObjectWithArrow{(1.2,0.1,0.13)}{north west}
        {$\underline{\kk}$}{(0.5,0.5,0.5)}
    \denoteBoldObject{(0,0,1)}{east}{$\My(\x-\z)$}
    \denoteBoldObjectWithArrow{(0,0.1,0.2)}{north east}
        {$\My(-\z)$}{(0.5,0,1)}
    \denoteBoldObject{(0,0,0.5)}{east}{$\My(\x)$}
    \denoteBoldObject{(0.5,0,0.5)}{north}
        {$\lefteqn{\underline{\phantom{\Mz}}}\My$}
}

\newcommand{\denoteManyRedObjectsSplitLoop}
{
    \denoteBoldObject{(2,0.5,0.5)}{west}{$\kk(-(p-2)\x)$}
    \denoteBoldObject{(2,1,0.5)}{west}{$\kk(-(p-2)\x-\y)$}

    \denoteCommonRedObjects

}

\newcommand{\denoteManyObjectsSplitLoop}
{
    \denoteBoldObject{(2,0,0.5)}{north west}
        {$\My(-(p-2)\x)$}

    \denoteManyRedObjectsSplitLoop

    \denoteBoldObjectWithArrow{(0.5,0.2,0)}{north}
        {$\underline{\kk}$}{(0.5,0.5,0.5)}

    \denoteBoldObject{(0,0,2)}{east}{$\Mxy(\x+\y-(r-2)\z)$}
    \denoteBoldObjectWithArrow{(-0.6,0.3,1.2)}{east}
        {$\My(-(r-2)\z)$}{(0.5,0,2)}
    \denoteBoldObject{(0,0,1)}{east}{$\Mxy(\x+\y-\z)$}
    \denoteBoldObject{(0,0,0.5)}{east}{$\Mxy(\x+\y)$}
    \denoteBoldObject{(0.5,0,0.5)}{east}
        {$\lefteqn{\underline{\phantom{\Mz}}}\My$}
    \denoteBoldObject{(1,0,0.5)}{north}
        {$\My(-\x)$}

    \denoteBoldObject{(0,2,2)}{east}{$\Mx(-(q-2)\y-(r-2)\z)$}
    \denoteBoldObject{(0,0.5,2)}{south east}{$\Mx(-(r-2)\z)$}
    \denoteBoldObjectWithArrow{(-0.1,0.2,0)}{north}
        {$\underline{\Mx}$}{(0,0.5,0.5)}

}

\newcommand{\denoteManyObjectsChainNotStrong}{
    \denoteCommonRedObjects
    \denoteBoldObject{(0.5,0.5,2)}{south east}{$\kk(-(r-2)\z)$}
    
    \denoteBoldObject{(0,0,2)}{east}{$\My(\x-(r-2)\z)$}
    \denoteBoldObjectWithArrow{(-0.5,0,1)}{east}{$\underline{\kk}$}{(0.5,0.5,0.5)}
    \denoteBoldObject{(0,0,0.5)}{east}{$\My(\x)$}
    \denoteBoldObject{(0.5,0,0.5)}{north east}
        {$\lefteqn{\underline{\phantom{\Mz}}}\My$}
        
    \denoteBoldObject{(2,0.5,0)}{west}{$\Mz(-(p-2)\x)$}
    \denoteBoldObject{(2,0,0)}{west}{$\Mz(-(p-2)\x+\y)$}
    \denoteBoldObject{(2,2,0)}{west}{$\Mz(-(p-2)\x-(q-2)\y)$}
    \denoteBoldObjectWithArrow{(0,0,-0.25)}{east}{$\underline{\Mz}$}{(0.5,0.5,0)}
    \denoteBoldObjectWithArrow{(0.5,0,-0.25)}{north}{$\Mz(\y)$}{(0.5,0,0)}
}

\newcommand{\drawBlueFace}
{
    \draw[faceB] (0.5,0,0) -- (2,0,0) -- (2,2,0) -- (0.5,2,0) -- cycle;
}

\newcommand{\drawBlueFaceShifted}
{
    \draw[faceB] (0,0,0) -- (1.5,0,0) -- (2,0.5,0) -- (2,2,0) -- (0.5,2,0) -- (0.5,0.5,0) -- cycle;
}

\newcommand{\arrowsBlueToGreen}
{
    \foreach \x in {0.5,1,...,2}
        \draw[arrow] (\x,0,0) -- (\x,0,0.5);
    \foreach \x in {0.5,1,...,1.5}
        \draw[arrow] (\x,0,0) -- (\x+0.5,0,0.5);
}

\newcommand{\arrowsBlueToGreenShifted}
{
    \foreach \x in {0,0.5,...,1.5}
        \draw[arrow] (\x,0,0) -- (\x,0,0.5);
    \foreach \x in {0,0.5,...,1.5}
        \draw[arrow] (\x,0,0) -- (\x+0.5,0,0.5);
}

\newcommand{\arrowsInsideBlue}
{
    \foreach \x in {0.5,1,1.5}
        \foreach \y in {0,0.5}
        {
            \draw[arrow] (\x,\y,0) -- (\x+0.5,\y,0);
        }
}

\newcommand{\arrowsInsideBlueShifted}
{
    \foreach \x in {0.5,1,1.5}
        \foreach \y in {0,0.5}
        {
            \draw[arrow] (\x+\y-0.5,\y,0) -- (\x+\y,\y,0);
		}
}

\newcommand{\arrowsBlueToRed}
{
    \foreach \x in {0.5,1,...,2}
    {
        \draw[arrow] (\x,0,0) -- (\x,0.5,0.5);
        \draw[arrow] (\x,0.5,0) -- (\x,0.5,0.5);
    }

    \foreach \x in {0.5,1,1.5}
        \foreach \y in {0,0.5}
        {
            \draw[arrow] (\x,\y,0) -- (\x+0.5,0.5,0.5);
        }

    \foreach \y in {1,1.5,2}
    {
        \draw[arrow] (2,\y,0) -- (2,\y,0.5);
        \draw[arrow] (2,\y-0.5,0) -- (2,\y,0.5);
    }
}

\newcommand{\arrowsBlueToRedShifted}
{
    \foreach \x in {0.5,1,1.5} 
        \draw[arrow] (\x,0,0) -- (\x,0.5,0.5);
    \foreach \x in {0.5,1,...,2}
        \draw[arrow] (\x,0.5,0) -- (\x,0.5,0.5);

    \foreach \x in {0.5,1,1.5}
        \foreach \y in {0,0.5}
        {
            \draw[arrow] (\x,\y,0) -- (\x+0.5,0.5,0.5);
        }
    \draw[arrow] (0,0,0) -- (0.5,0.5,0.5);

    \foreach \y in {1,1.5,2}
    {
        \draw[arrow] (2,\y,0) -- (2,\y,0.5);
        \draw[arrow] (2,\y-0.5,0) -- (2,\y,0.5);
    }
}

\newcommand{\drawRedCube}
{
	\draw[cube hidden] (0.5,2,0.5) -- (2,2,0.5);
	\draw[cube hidden] (0.5,0.5,0.5) -- (0.5,2,0.5);
	\draw[cube hidden] (0.5,2,0.5) -- (0.5,2,2);

	\draw[cube] (0.5,0.5,0.5) -- (2,0.5,0.5) -- (2,2,0.5) -- (2,2,2) -- (0.5,2,2) -- (0.5,0.5,2) -- cycle;

    \draw[cube front edge] (0.5,0.5,2) -- (2,0.5,2) -- (2,2,2);
    \draw[cube front edge] (2,0.5,0.5) -- (2,0.5,2);
}

\newcommand{\drawRedVertixWithAllArrows}
{
    \draw[arrow] (0.5,0.5,0.5) -- (1,0.5,0.5);
    \draw[arrow] (0.5,0.5,0.5) -- (0.5,1,0.5);
    \draw[arrow] (0.5,0.5,0.5) -- (0.5,0.5,1);
    \draw[arrow] (0.5,0.5,0.5) -- (1,1,0.5);
    \draw[arrow] (0.5,0.5,0.5) -- (0.5,1,1);
    \draw[arrow] (0.5,0.5,0.5) -- (1,0.5,1);
    \draw[arrow] (0.5,0.5,0.5) -- (1,1,1);

    \draw[grid dotted] (0.5,1,1) -- (0.5,0.5,1) -- (1,0.5,1) -- (1,0.5,0.5) -- (1,1,0.5) -- (1,1,1) -- cycle;
    \draw[grid dotted] (0.5,1,1) -- (0.5,1,0.5) -- (1,1,0.5);
    \draw[grid dotted] (1,0.5,1) -- (1,1,1);

    \draw[arrow] (0.5,0.5,2) -- (1,0.5,2);
    \draw[arrow] (0.5,0.5,2) -- (1,1,2);
    \draw[arrow] (0.5,0.5,2) -- (0.5,1,2);
    \draw[dots] (1.15,1.15,2) -- (1.8,1.8,2);
    \draw[arrow] (1,0.5,2) -- (1,1,2);
    \draw[dots] (1.2,0.75,2) -- (1.8,0.75,2);
    \draw[arrow] (0.5,1,2) -- (1,1,2);
    \draw[dots] (0.75,1.2,2) -- (0.75,1.8,2);

    \draw[arrow] (2,0.5,0.5) -- (2,1,0.5);
    \draw[arrow] (2,0.5,0.5) -- (2,0.5,1);
    \draw[arrow] (2,0.5,0.5) -- (2,1,1);
    \draw[dots] (2,1.15,1.15) -- (2,1.8,1.8);
    \draw[arrow] (2,1,0.5) -- (2,1,1);
    \draw[dots] (2,1.2,0.75) -- (2,1.8,0.75);
    \draw[arrow] (2,0.5,1) -- (2,1,1);
    \draw[dots] (2,0.75,1.2) -- (2,0.75,1.8);

    \draw[arrow] (2,0.5,2) -- (2,1,2);
    \draw[dots] (2,1.2,2) -- (2,1.8,2);

    \draw[arrow] (0.5,2,2) -- (1,2,2);
    \draw[dots] (1.2,2,2) -- (1.8,2,2);

    \draw[arrow] (2,2,0.5) -- (2,2,1);
    \draw[dots] (2,2,1.2) -- (2,2,1.8);

}

\newcommand{\drawRedArrowsExtra}
{
\begin{scope}[shift={(0.5,0,0)}]
    \draw[arrow] (0.5,0.5,0.5) -- (1,0.5,0.5);
    \draw[arrow] (0.5,0.5,0.5) -- (0.5,1,0.5);
    \draw[arrow] (0.5,0.5,0.5) -- (0.5,0.5,1);
    \draw[arrow] (0.5,0.5,0.5) -- (1,1,0.5);
    \draw[arrow] (0.5,0.5,0.5) -- (0.5,1,1);
    \draw[arrow] (0.5,0.5,0.5) -- (1,0.5,1);
    \draw[arrow] (0.5,0.5,0.5) -- (1,1,1);

    \draw[grid dotted] (1,1,0.5) -- (1,1,1);
    \draw[grid dotted] (0.5,1,1) -- (1,1,1);
    \draw[grid dotted] (1,0.5,1) -- (1,1,1);
    \draw[grid dotted] (1,0.5,0.5) -- (1,0.5,1);
    \draw[grid dotted] (0.5,0.5,1) -- (1,0.5,1);
    \draw[grid dotted] (0.5,1,0.5) -- (1,1,0.5);
    \draw[grid dotted] (1,0.5,0.5) -- (1,1,0.5);
\end{scope}
\begin{scope}[shift={(0,0,0.5)}]
    \draw[arrow] (0.5,0.5,0.5) -- (1,0.5,0.5);
    \draw[arrow] (0.5,0.5,0.5) -- (0.5,1,0.5);
    \draw[arrow] (0.5,0.5,0.5) -- (0.5,0.5,1);
    \draw[arrow] (0.5,0.5,0.5) -- (1,1,0.5);
    \draw[arrow] (0.5,0.5,0.5) -- (0.5,1,1);
    \draw[arrow] (0.5,0.5,0.5) -- (1,0.5,1);
    \draw[arrow] (0.5,0.5,0.5) -- (1,1,1);

    \draw[grid dotted] (1,1,0.5) -- (1,1,1);
    \draw[grid dotted] (0.5,1,1) -- (1,1,1);
    \draw[grid dotted] (1,0.5,1) -- (1,1,1);
    \draw[grid dotted] (0.5,1,0.5) -- (0.5,1,1);
    \draw[grid dotted] (0.5,0.5,1) -- (0.5,1,1);
    \draw[grid dotted] (1,0.5,0.5) -- (1,0.5,1);
    \draw[grid dotted] (0.5,0.5,1) -- (1,0.5,1);
\end{scope}
    \draw[dots] (1.2,0.5,1.2) -- (1.8,0.5,1.8);
}

\newcommand{\arrowsGreenToRed}
{
    \foreach \x in {0,0.5}
    \draw[arrow] (\x,0,2) -- (\x+0.5,0.5,2);
    \foreach \x in {0.5,1,2}
    \draw[arrow] (\x,0,2) -- (\x,0.5,2);
    \draw[dots] (1.2,0.25,2) -- (1.8,0.25,2);

    \foreach \z in {0.5,1,1.5}
    {
        \draw[arrow] (0, 0, \z) -- (0.5, 0.5, \z); 
        \draw[arrow] (0.5, 0, \z) -- (0.5, 0.5, \z); 

        \draw[arrow] (0,0,\z) -- (0.5,0.5,\z+0.5); 
        \draw[arrow] (0.5,0,\z) -- (0.5,0.5,\z+0.5); 
    }

    \foreach \z in {0.5,1}
        \draw[arrow] (2,0,\z) -- (2,0.5,\z); 
    \draw[arrow] (2,0,0.5) -- (2,0.5,1); 
    \draw[dots] (2,0.25,1.2) -- (2,0.25,1.8);

}

\newcommand{\arrowsGreenToRedReduced}
{
    \foreach \x in {0.5}
    \draw[arrow] (\x,0,2) -- (\x+0.5,0.5,2);
    \foreach \x in {0.5,1,2}
    \draw[arrow] (\x,0,2) -- (\x,0.5,2);
    \draw[dots] (1.2,0.25,2) -- (1.8,0.25,2);

    \foreach \z in {0.5,1,1.5}
    {
        \draw[arrow] (0.5, 0, \z) -- (0.5, 0.5, \z); 

        \draw[arrow] (0.5,0,\z) -- (0.5,0.5,\z+0.5); 
    }

    \foreach \z in {0.5,1}
        \draw[arrow] (2,0,\z) -- (2,0.5,\z); 
    \draw[arrow] (2,0,0.5) -- (2,0.5,1); 
    \draw[dots] (2,0.25,1.2) -- (2,0.25,1.8);
}

\newcommand{\arrowsInsideGreen}{
    \foreach \z in {0.5,1,1.5}
        \foreach \x in {0,0.5,...,2}
            \draw[arrow] (\x,0,\z) -- (\x,0,\z+0.5);
}

\newcommand{\arrowsInsideGreenReduced}{
    \foreach \z in {0.5,1,1.5}
        \foreach \x in {0.5,1,1.5,2}
            \draw[arrow] (\x,0,\z) -- (\x,0,\z+0.5);
}

\newcommand{\drawGreenFace}
{
    \draw[faceG] (0,0,0.5) -- (2,0,0.5) -- (2,0,2) -- (0,0,2) -- cycle;
    \draw[faceG] (0.5,0,0.5) -- (0.5,0,2);
}

\newcommand{\drawGreenFaceReduced}
{
    \draw[faceG] (0.5,0,0.5) -- (2,0,0.5) -- (2,0,2) -- (0.5,0,2) -- cycle;
    \draw[faceG] (0.5,0,0.5) -- (0.5,0,2);
}

\newcommand{\drawYellowFace}
{
    \draw[faceY] (0,0.5,0) -- (0,2,0) -- (0,2,2) -- (0,0.5,2) -- cycle;
}

\newcommand{\drawYellowFaceReduced}
{
    \draw[faceY] (0,0.5,0.5) -- (0,2,0.5) -- (0,2,2) -- (0,0.5,2) -- cycle;
}

\newcommand{\arrowsInsideYellowReduced}
{
    \foreach \z in {0.5,1,1.5}
    {
        \foreach \y in {0.5,1,2}
            \draw[arrow] (0,\y,\z) -- (0,\y,\z+0.5);
    }
}

\newcommand{\arrowsInsideYellow}
{
    \foreach \z in {0,0.5,...,2}
        \draw[arrow] (0,0.5,\z) -- (0,1,\z);
    \foreach \y in {1,1.5}
        \draw[arrow]  (0,\y,2) -- (0,\y+0.5,2);
}

\newcommand{\arrowsGreenToYellow}
{
    \foreach \z in {0.5,1,...,2}
       \draw[arrow] (0,0,\z) -- (0,0.5,\z);
    \foreach \z in {0.5,1,1.5}
        \draw[arrow] (0,0,\z) -- (0,0.5,\z+0.5);
}

\newcommand{\arrowsYellowToRed}
{
    \draw[arrow] (0,0.5,0.5) -- (0.5,0.5,0.5);
    \draw[arrow] (0,0.5,0) -- (0.5,0.5,0.5);

    \foreach \y in {0.5,1,1.5}
    {
        \draw[arrow] (0,\y,2) -- (0.5,\y,2);
        \draw[arrow] (0,\y,2) -- (0.5,\y+0.5,2);
    }
    \draw[arrow] (0,2,2) -- (0.5,2,2);
    \draw[arrow] (0,2,0.5) -- (0.5,2,0.5);
}

\newcommand{\arrowsYellowToRedReduced}
{

    \foreach \y in {0.5,1,1.5}
    {
        \draw[arrow] (0,\y,2) -- (0.5,\y,2);
        \draw[arrow] (0,\y,2) -- (0.5,\y+0.5,2);
    }
    \foreach \z in {0.5,1,1.5}
    {
        \draw[arrow] (0,0.5,\z) -- (0.5,0.5,\z);
        \draw[arrow] (0,0.5,\z) -- (0.5,0.5,\z+0.5);
    }
    \draw[arrow] (0,2,2) -- (0.5,2,2);
    \draw[arrow] (0,2,0.5) -- (0.5,2,0.5);
}

\newcommand{\arrowsYellowToBlue}
{
    \draw[arrow] (0,0.5,0) -- (0.5,0.5,0);
    \draw[arrow] (0,0.5,0) -- (0.5,1,0);
}

\newcommand{\drawCornerArrows}
{
    \draw[arrow] (0,0,0) -- (0.5,0,0);
    \foreach \x in {0,0.5}
    {
        \draw[arrow] (0,0,0) -- (\x,0.5,0);
        \foreach \y in {0,0.5}
        {
            \draw[arrow] (0,0,0) -- (\x,\y,0.5);
        }
    }
}

\newcommand{\arrowsZaxisAndToRed}
{
    \foreach \z in {0.5,1,1.5}
    {
        \draw[arrow] (0,0,\z) -- (0,0,\z+0.5);
        \draw[arrow] (0,0,\z) -- (0.5,0.5,\z+0.5);
    }
    \foreach \z in {0.5,1,...,2}
        \draw[arrow] (0,0,\z) -- (0.5,0.5,\z);
}

\newcommand{\CubeThreeDimSplit}{{
\begin{mycube}
    \drawCoordinateAxes{-0.5}{1}{2}

    \drawRedCube
    \drawRedVertixWithAllArrows
    \drawRedArrowsExtra

    \denoteManyRedObjects
\end{mycube}
}}

\newcommand{\CubeThreeDimSplitChain}{{
\begin{mycube}
    \drawCoordinateAxes{-1.25}{1}{2}

    \drawRedCube
    \drawRedVertixWithAllArrows
    
    \arrowsGreenToRed
    \drawGreenFace
    \arrowsInsideGreen

    \denoteManyObjectsSplitChain
    
\end{mycube}
}}

\newcommand{\CubeThreeDimSplitLoop}{
\begin{mycube}

    \drawCoordinateAxes{4}{1.5}{0}
    
    \drawYellowFaceReduced
    \arrowsInsideYellowReduced
    \arrowsYellowToRedReduced

    \arrowsZaxisAndToRed

    \drawRedCube
    \drawRedVertixWithAllArrows
    
    \arrowsGreenToRedReduced
    \drawGreenFaceReduced
    \arrowsInsideGreenReduced

    \denoteManyObjectsSplitLoop
    
\end{mycube}
}

\newcommand{\CubeThreeDimChainNotStrong}{
\begin{mycube}
    \drawCoordinateAxes{-1.25}{1}{2}

    \drawBlueFace
    \arrowsBlueToRed
    \arrowsInsideBlue

    \drawRedCube
    \drawRedVertixWithAllArrows
    
    \arrowsGreenToRed
    \drawGreenFace
    \arrowsInsideGreen
    \arrowsBlueToGreen

    \denoteManyObjectsChainNotStrong
    
    \draw[arrow bold] (2, 0, 0) .. controls (1.5,-0.5,0) and (0,-0.5,0) .. (0, 0, 0.5);
\end{mycube}
}

\newcommand{\CubeThreeDimChain}{
\begin{mycube}
    \drawCoordinateAxes{4.5}{1}{0}

    \drawBlueFaceShifted
    \arrowsBlueToGreenShifted
    \arrowsBlueToRedShifted
    \arrowsInsideBlueShifted

    \drawRedCube
    \drawRedVertixWithAllArrows

    \arrowsGreenToRed
    \drawGreenFace
    \arrowsInsideGreen

    \denoteMediumObjectsChain
\end{mycube}
}

\newcommand{\CubeThreeDimLoop}{
\begin{mycube}
    \drawCoordinateAxes{4.5}{1}{0}

    \drawYellowFace
    \arrowsInsideYellow
    \arrowsYellowToRed
    \arrowsYellowToBlue
    \arrowsGreenToYellow

    \drawBlueFace
    \arrowsBlueToGreen
    \arrowsBlueToRed
    \arrowsInsideBlue

    \drawRedCube
    \drawRedVertixWithAllArrows

    \drawCornerArrows

	\arrowsGreenToRed
    \drawGreenFace
    \arrowsInsideGreen

    \denoteMediumObjectsLoop

    \draw[black,thick,fill] (0,0,0) circle [radius=0.5pt] node [anchor=east] {$\Mxyz(\x+\y+\z)$};
\end{mycube}
}

\mathchardef\mhyphen="2D
\newcommand*{\rmodgrgeneric}[3]{\mathrm{#1}^{#2}\mathrm{\mhyphen}{#3}}
\newcommand*{\rmodgr}[2]{\rmodgrgeneric{mod}{#1}{#2}}

\def\modLA{{\rmodgr LA}}

\newcommand*{\bull}[1]{{#1}^\bullet}
\def\Homb{\bull{\Hom}}

\def\Hb{\bull{\HH}}

\def\D{{\mathbf{D}}}
\def\Db{{\D^b}}
\def\Dsg{\D_{\mathrm{sg}}}
\newcommand*{\dbof}[1]{\Db({#1})}
\newcommand*{\dbofgr}[2]{\Db(\rmodgr{#1}{#2})}
\newcommand*{\dsgof}[2]{\Dsg^{#1}({#2})}
\def\dbmodLA{\dbof{\modLA}}

\def\dblrw{\dbof{\rmodgr L\Rw}}
\def\dsgLA{\dsgof LA}
\def\dsgA{\dsgof{}A}

\def\dsglrw{\dsgof L\Rw}
\def\dsgw{\dsgof{L_w}{R/(w)}}
\def\dsglrww{\dsgw}

\DeclareMathOperator{\Fuk}{Fuk}
\def\Fukto{\Fuk^\to}

\def\x{{\vec x}}
\def\y{{\vec y}}
\def\z{{\vec z}}
\def\w{{\vec w}}

\def\Aff#1{{\mathbb{A}_\kk^#1}}
\def\Affn{{\mathbb{A}_\kk^n}}

\def\SS{{\mathcal S}}

\def\RR{\mathbf{R}}
\def\LL{\mathbf{L}}

\def\ainfonly{{$A_{\infty}$}}

\def\dgonly{{\textbf{dg}}}
\def\ainf{\ainfonly\=/}

\def\dg{\dgonly\=/}

\begin{document}

\maketitle

\begin{abstract}
  We study the categories of singularities coming from Landau-Ginzburg models given by the invertible polynomials. Such categories appear on the B-side of the Berglund-H{\"u}bsch mirror symmetry. We provide an efficient method of computing morphism spaces in these categories and explicitly construct full strongly exceptional collections in the cases of small dimensions ($n\le 3$). Finally, we use this construction in order to prove Orlov's conjecture stating that such collections can be chosen to have block decompositions of size one more than the number of variables.
\end{abstract}

\setcounter{tocdepth}{2}
\tableofcontents

\section{Introduction}

\def\behu{Berglund-H{\"u}bsch}
\def\dgainf{\dgonly/\ainf}

The categories of graded singularities were introduced by Orlov in \cite{orlov03} and \cite{orlov05}. These categories naturally appear on the algebraic side of the Homological Mirror Symmetry correspondence (HMS) for equivariant Landau-Ginzburg models and are often referred to as categories of matrix factorizations. We are particularly interested in the case of singularities given by the so-called invertible polynomials (see Definition~\ref{def:invertibles}). There is a particularly explicit formulation of mirror symmetry for invertible polynomials, the so-called \behu{} duality as follows:

\begin{conjecture*}[{see the survey \cite{BeHu-survey}}]
  Let $w\in\Cc[x_1,\dots,x_n]$ be an invertible polynomial. Then there exists an equivalence of categories:
  $$ \dsgof{L_w}{\Cc[x_1,\dots,x_n]/(w)} \simeq \Db\Fukto(w^t), $$
  where on the left side we have a graded derived category of singularities of $w$ (see Definitions~\ref{def:dsg-gr} and~\ref{def:max-grading-group}), and on the right side we have a derived directed Fukaya category of a dual invertible polynomial $w^t$ considered as a superpotential on the ambient space $\mathbb{A}^n_\Cc$.
\end{conjecture*}

A general method of proving Homological Mirror Symmetry conjectures originating from Seidel consists of picking a finite set of generators in each of the algebraic and symplectic categories and showing that the \dgainf structures restricted to these generators are quasi-isomorphic (see \cite{seidel-origin}).
For example, when a category admits a full exceptional collection (see Definition~\ref{def:full/strong-collection}), we can choose the objects of this collection to be our generators. Moreover, when such collection is also \textit{strong}, the restricted \dgainf structure is formal and thus can be recovered from the $\HH^0$-level, that is from the triangulated structure of the corresponding categories (see Proposition~\ref{prop:strong-coll=>unique-dg-enhanc}). In such case, in order to prove the desired mirror symmetry equivalence, one just needs to match the objects of the exceptional collections between the algebraic and symplectic sides and to compare the morphisms on the triangulated/$\HH^0$-levels.

The above method is particularly applicable for the \behu{} duality since the symplectic category possesses a strong exceptional collection by definition. However, one still needs to find such collection on the algebraic side and do the necessary matching in order to prove the above conjecture. This was done by Ueda for simple elliptic singularities in \cite{ueda}, by Futaki-Ueda for Brieskorn-Pham singularities in \cite{futaki-ueda} and recently by Habermann-Smith for $n\le 2$ in \cite{hab-smith}. Also, Kajiura-Saito-Takahashi in \cite{KST} constructed full strongly exceptional collections on the algebraic side for ADE-singularities in the case $n=3$.

\subsection{Main results}

In this paper, we are working with the algebraic side of \behu{} duality. Namely, in Section~3, we explicitly construct full strongly exceptional collections in the corresponding categories for $n\le 3$ thus proving the following theorem.
\begin{Theorem}
\label{thm1}
  Let $w\in \Cc[x_1,\dots,x_n]$ be an invertible polynomial, where $n\le 3$. Then the maximally graded category of singularities $\dsgof{L_w}{\Cc[x_1,\dots,x_n]}$ admits an explicit full strongly exceptional collection.
\end{Theorem}

Our explicit collections also follow a very specific pattern which may be used for constructing the desired collections for higher $n$ (see \S\ref{sec3:higher-n}). One just needs to reuse our methods and to perform the number of computations for verifying the necessary conditions for the candidate collections.

Orlov also conjectured that the categories participating in \behu{} duality have a very particular form. Namely, these categories admit full strongly exceptional collections that can be split into a small number of blocks each consisting of pairwise orthogonal objects as in the following.
\begin{ConjectureO}
\label{conj:orlov}
  For an invertible polynomial $w\in \Cc[x_1,\dots,x_n]$ and the corresponding maximal grading group $L_w$, the graded category of singularities $\dsgof{L_w}{\Cc[x_1,\dots,x_n]/(w)}$ admits a full strongly exceptional collection consisting of at most $n+1$ blocks (see Definition~\ref{def:block}).
\end{ConjectureO}
\begin{remark*}
  The number of blocks is $n+1$ and depends only on the number of variables~$n$ but not on the particular invertible polynomial.
\end{remark*}

In Section~4, we apply mutations to the explicit collections from the proof of Theorem~1 and prove the above conjecture for $n\le 3$.

\begin{Theorem}
\label{thm2}
    Orlov's conjecture holds for $n\le 3$.
\end{Theorem}

\begin{remark*}{}
  In \cite{zametka}, we announced the results of this paper and described the explicit collections for the main cases of 3-chain and 3-loop polynomials. We apologize that it took so long to finish and submit this paper.
\end{remark*}

\subsection{Method of proof}

Our construction in the proof of Theorem~1 is based on the idea that the desired collections can be formed out of a very small number of initial objects by shifting them via the group action. This was firstly observed by Ueda in \cite{ueda} and reused by Futaki-Ueda in \cite{futaki-ueda} who studied the so-called Brieskorn-Pham polynomials $x_1^{p_1} + x_2^{p_2}+\ldots + x_n^{p_n}$ which form the simplest type of the invertible polynomials. Namely, they constructed explicit full strongly exceptional collections in these categories from the shifts of a \textit{single} object, the structure field, by arranging them into an $n$-dimensional cube (see Proposition~\ref{prop:futaki-ueda-split}). In this paper we generalize this construction to other types of invertible polynomials by adding new objects and their shifts to the $n$-dimensional lattice.

We use a well-known classification of invertible polynomials (see Example~\ref{ex:invertible-classification-n<=3}) and construct an exceptional collection for each type of $w$ individually. We do this in Proposition~\ref{prop:main:n=2} for the case of $n=2$ and in Proposition~\ref{prop:main:n=3} for $n=3$. The most important cases are the cases of so-called \textit{chain} and \textit{loop} polynomials, because all other invertible polynomials can be decomposed into them and the total exceptional collections can be described in terms of the collections for these cases.

Our proof in each case consists of two parts. The first part is checking that the collections are indeed exceptional (and strong). The second part is checking that the collections are full, that is that they generate the whole category.

In order to check that the collections are exceptional, we need to compute the morphisms between the given objects in the categories of singularities. An efficient way of doing so is to use the idea by Buchweitz that the categories of singularities, or the categories of matrix factorizations, can be described via certain stabilization properties. Namely, the morphisms in the categories of singularities can be obtained as the stabilizations of the morphisms in the corresponding derived categories of the coherent sheaves (see Proposition~\ref{prop:stabilization}). Buchweitz formulated this observation in Remark~1.3(a) of \cite{buchweitz}. However, the statement as written there omits the necessary conditions (see counter-example in Remark~\ref{remark:counter-example-for-Buchweitz}) and the full proof is missing as well. Section~2 is dedicated to filling this gap. Namely, we use Orlov's ideas to refine some general results about the categories of singularities, and then specialize them to the cases of Gorenstein and hypersurface rings.

In order to prove the generation statements, we follow an idea of Ueda from \cite{ueda}. Namely, in each case we show explicitly that the objects of our collections generate all the shifts of the structure field and then conclude the statement by applying the well-known generation criterion (see Proposition~\ref{prop:gen-criterion-invertibles-collections}). In general case (when $w$ is not Brieskorn-Pham), these statements are harder to prove since one needs to deal with many different objects simultaneously in a context of a more complex grading group. Our computations proving these statements take up significant part of Section~3.

\paragraph{Recent findings.}
While this paper was still in preparation, a few papers appeared which are related to this work. Habermann and Smith in \cite{hab-smith} constructed similar collections in the case of $n=2$. Takahashi and Aramaki in \cite{takahashi-new} considered the case of chain polynomials for arbitrary~$n$. They have built full exceptional collections in these cases which are not strong. Hirano and Ouchi in \cite{hirano-ouchi} also considered the chain case for arbitrary~$n$. They haven't constructed explicit collections but have built certain semi-orthogonal decompositions of the categories of singularities which agree with the pattern of our collections on Figures~\ref{fig:2chain} and~\ref{fig:n=3:3-chain-not-strong}.

\paragraph{Acknowledgements.} We thank Dmitri Orlov for suggesting the problem and for his undivided attention to this work, Mohammed Abouzaid for making useful suggestions and encouraging us to finish the paper, Ludmil Katzarkov, David Favero, Matthew Ballard and Alexander Efimov for helpful discussions, and Yankı Lekili for helpful discussions and sharing the reference for \cite{hab-smith}.

\section{Morphisms in the categories of singularities}
\label{sec:morphisms-dsg}

The goal of this section is to show how to efficiently compute morphisms in categories of hypersurface singularities.

In \S\ref{sec2:definition}, we recall the definition of the graded categories of singularities. In~\S\ref{sec2:explicit-morphisms}, we introduce the notion of \textit{descending sequences} in the bounded derived category of coherent sheaves and use this notion to explicitly express morphisms in the corresponding category of singularities. Our proofs are based on Orlov's ideas.
In \S\ref{sec:gr/ungr}, we recall the well-known correspondence between the graded and ungraded worlds necessary for the morphism computations.
In \S\ref{sec:stabilization-gorenstein-rings}, we specialize to the case of Gorenstein rings. We refine Orlov's results and essentially prove Buchweitz's remark in Propositions~\ref{prop:stabilization} and~\ref{prop:Buchweitz's-Remark}.
In \S\ref{sec:hypersurface-singularities}, we further specialize to the case of hypersurface singularities. We recall the quasi-periodicity property and use it to give an explicit description of the morphisms in categories of singularities in terms of Ext's of modules.
In \S\ref{sec:morphisms-applications}, we summarize the obtained results as an algorithmic recipe for computing the morphisms in the categories of hypersurface singularities (see Recipe~\ref{recipe}).

\subsection{Preliminaries}
\label{sec2:definition}

Let $A$ be an arbitrary Noetherian ring which is not necessarily commutative and which is equipped with an $L$-grading $A = \bigoplus_{l\in L} A_l$, where $L$ is an arbitrary abelian group. Let $\modLA$ denote the category of the right $L$-graded finitely generated $A$-modules. The bounded derived category $\dbmodLA$ then consists of the bounded complexes of modules. A \textit{perfect} complex in $\dbmodLA$ is a complex quasi-isomorphic to a bounded complex of projective modules. The perfect complexes in $\dbmodLA$ form a triangulated subcategory which we denote by $\Perf^L(A)$.

\begin{definition}[Orlov, see \cite{orlov05}]
\label{def:dsg-gr}
  A \textit{graded category of singularities} $\dsgLA$ is defined as the quotient of triangulated categories $\dbmodLA/\Perf^L(A)$.
\end{definition}

The objects of $\dsgLA$ are represented by the objects of $\dbmodLA$, that is by the bounded complexes. The morphisms in $\dsgLA$ from $M$ to $N$ are formed by the roofs $M \xrightarrow{a} T \xleftarrow{s} N$, where $a$ is an arbitrary morphism in $\dbmodLA$ and $s$ belongs to the localizing set $\SS = \{s\:X \to Y \mid \Cone(s) \in \Perf^L(A)\}$ of morphisms in $\dbmodLA$. These roofs are further subject to the usual equivalence relations coming from the definition of localization.

\begin{definition}
In the ungraded case, $L=\{0\}$, we will denote the category of singularities $\dsgLA$ simply by $\dsgA$ and call it the \textit{ungraded category of singularities}.
\end{definition}

\begin{remark}
We refer the reader to \cite{orlov03} and \cite{neeman-book} for the details on triangulated categories and their quotients.
\end{remark}

\subsection{Descending sequences and explicit characterization of the morphisms}
\label{sec:morphisms-details}
\label{sec2:explicit-morphisms}

Using ideas from \cite[\S1.3]{orlov05}, we provide an explicit characterization of morphisms in the categories of singularities in terms of the direct systems of morphisms in the corresponding bounded derived categories (see Theorem~\ref{theorem:dsg-morphisms:db->dsg}).

Namely, we firstly show that any bounded complex can be ``descended'' towards lower degrees via the isomorphisms in the corresponding singularity category. Then we show that in order to compute the morphisms between $\bull M$ and $\bull N$ in the singularity category $\dsgLA$, we should just replace $\bull N$ by its ``descending sequence'' and then take the direct limit of the corresponding sequence of morphisms in the bounded derived category $\dbmodLA$.

\begin{definition}
  Let $\bull M$ be a (non-zero) bounded complex situated in degrees $\le b$ with a non-zero component at $b$. By a \textit{descending morphism} of $\bull M$ we mean a morphism $p\: \bull M \to \bull N$ in the bounded derived category $\dbmodLA$, such that $\bull N$ is situated in degrees $\le b-1$ and $\Cone(p)$ is a perfect complex. We call $\bull N$ the \textit{descent} of $\bull M$.
  (By a descending morphism of a zero complex $\bull M = \bull 0$, we mean a zero morphism from this complex to itself $0\:\bull 0 \to \bull 0$.)
\end{definition}

\begin{remark}
  When we say that the complex $\bull M$ is \textit{situated in degrees} $\le b$ for some $b$, we mean that $M^i=0$ for $i>b$. This is equivalent to the condition that $\HH^i(\bull M)=0$ for $i>b$, if we allow replacing the complexes by the quasi-isomorphic ones.
\end{remark}

\begin{remark}
  The definition implies that any descending morphism $p$ from the category $\dbmodLA$ becomes an isomorphism in $\dsgLA$.
\end{remark}

\begin{remark}
  If $\bull M$ is situated in degrees $\le b$, then for any descending morphism $p$ of $\bull M$, an object $\Cone(p)$ is automatically situated in degrees $\le b-1$.
\end{remark}

In \cite{orlov05}, Orlov proved the following result.

\begin{lemma}
[{see \cite[Lemma~1.10]{orlov05}}]
\label{lem:everything-is-a-shift-of-a-module}
  Any object in $\dsgLA$ is isomorphic to a shift of a module.
\end{lemma}

Actually, the argument in the proof of the last lemma gives rise to the construction of a particular descending morphism from any complex to a shift of a module, as formulated in the following lemma.

\begin{lemma}
\label{lem:descend}
    Any bounded complex $\bull M$ can be descended. Moreover, its descent $\bull N$ can be chosen to be a module situated in an arbitrarily low degree. In other words, $\bull N = K[n]$, where $n\gg 0$ and $K$ is a module depending on $n$.
\end{lemma}
\begin{proof}
\def\bmin{c}
\def\bmax{b}
\def\pii{{\overline{\pi}}}
  If $\bull M = \bull 0$, the statement is obvious. Otherwise, let us pick $\bmin$ and $\bmax$ such that $\bull M$ is situated in degrees $[\bmin,\bmax]$ with a non-zero component at $\bmax$. Consider its projective resolution $\pi\:\bull P\to\bull M$ situated in degrees $\le\bmax$ and its smart truncation $\bull T = (K\to P^\bmin \to\dots\to P^\bmax)$ situated in degrees $[\bmin-1,\bmax]$, where $K = P^{\bmin-1}/\Ker d_P^{\bmin-1}$. Restriction of $\pi$ to degrees $[c,b]$ gives rise to a chain map $\pii\:\bull T\to\bull M$. Since $\pii$ is a quasi-isomorphism of complexes, it has an inverse $\pii^{-1}\:\bull M\to \bull T$ in the derived category $\dbmodLA$.

Now, consider also the stupid truncation $\sigma^{\ge\bmin} \bull P = (P^\bmin \to\dots\to P^\bmax)$ of $\bull P$ and the corresponding short exact sequence of complexes
$$ 0 \to \sigma^{\ge\bmin} \bull P \to \bull T \xrightarrow{a} K[-\bmin+1] \to 0. $$
The latter gives rise to the isomorphism $\Cone(a) \simeq \sigma^{\ge\bmin} \bull P[1]$ in the derived category, thus implying that $\Cone(a)$ is a perfect complex.

Define $p$ as the following composition of morphisms in $\dbmodLA$:
$$ p := \left( \bull M \xlongrightarrow{\pii^{-1}} \bull T \xlongrightarrow{a} K[-c+1] \right). $$
Since $\pii^{-1}$ is an isomorphism in $\dbmodLA$, the complex $\Cone(p)$ is isomorphic to $\Cone (a)$ in $\dbmodLA$, hence is perfect as well.

Since the complex $K[-c+1]$ is situated in degree $c-1$ with $c\le b$, the morphism $p\: \bull M \to K[-c+1]$ satisfies both conditions of being a descending morphism of $\bull M$. Moreover, we could choose $c$ to be arbitrarily low. This concludes the proof.
\end{proof}

\begin{remark}
  If we apply the construction of Lemma~\ref{lem:descend} to a module, then the descent will be given by this module's first \textit{syzygies} (see \cite{lam} for definition).
\end{remark}

\begin{definition}
  By a \textit{descending sequence} $(N_i, p_i)_{i\ge0}$ of a complex $\bull N$, we mean a sequence of descending morphisms starting at $\bull N$:
  $$ (\bull N = N_0) \xrightarrow{p_0} N_1 \xrightarrow{p_1} N_2 \xrightarrow{p_2} N_3 \xrightarrow{p_3} \dots $$
\end{definition}

\begin{lemma}
    Any complex $\bull N$ admits a descending sequence.
\end{lemma}
\begin{proof}
  Follows from the inductive application of Lemma~\ref{lem:descend}.
\end{proof}

By refining the result of \cite[Proposition 1.11]{orlov05}, we can characterize the morphisms in $\dsgLA$ in terms of the descending morphisms and sequences in $\dbmodLA$ via the following construction.

Let $(N_i, p_i)_{i\ge 0}$ be a descending sequence of $\bull N$ and $\bull M$ be another bounded complex. For this data, let us define the direct system $(\Omega_i, \phi_{ij})$ as follows. We define the spaces $\Omega_i$ for $i\ge 0$ as \\[-2ex]
    \begin{equation}
    \label{eq:omega-space}
        \Omega_i := \Hom_\dbmodLA(\bull M, N_i).
    \end{equation}
    Then we define the morphisms $\phi_{ij}\:\Omega_i\to \Omega_j$ for $0\le i \le j$ as the compositions with the given descending morphisms:
    \begin{equation}
        \label{eq:omega-space-phi}
        \phi_{ij}\:\Hom_\Db(\bull M,N_i)\to\Hom_\Db(\bull M,N_j),\; f\mapsto p_{j-1} \dots p_{i+1} p_i f,
    \end{equation}
    \\[-5ex]
    \begin{equation*}
        \label{def:omega-space-phi-2}
     \begin{tikzcd}
        \bull M \ar[d, dashed, ""{name=1, pos=0.65}] 
        \ar[rd, dashed, ""'{name=2}, ""{name={2r}, near end}]
        \ar[rrd, dashed, bend left = 5, ""'{name=3, pos=0.6}, ""{name={3r}, near end}]
        \ar[rrrd, bend left=10, dashed, ""'{name=4, pos=0.73}, ""{name=4r, pos=0.85}]
        \ar[rrrrd, dotted, bend left=15, ""'{name=5, pos=0.835}]
        \\
        N_0 \ar[r,"p_0"'] & N_1 \ar[r, "p_1"'] & N_2 \ar[r, "p_2"'] & N_3 \ar[r, "p_3"'] & \cdots
        \arrow[Rightarrow, from=1, to=2, "\phi_{01}"']
        \arrow[Rightarrow, from={2r}, to=3, "\phi_{12}"', near end]
        \arrow[Rightarrow, from={3r}, to=4, "\phi_{23}"', near end]
        \arrow[Rightarrow, from={4r}, to={5}, "\phi_{3\dots}"', near end]
    \end{tikzcd}
    \end{equation*}

\begin{theorem}\label{theorem:dsg-morphisms:db->dsg}
For any bounded complexes $\bull M$ and $\bull N$ and for any descending sequence $(N_i, p_i)_{i\ge0}$ of $\bull N$, we have the following isomorphism:\\[-1ex]
  $$\Hom_\dsgLA(\bull M,\bull N) \simeq \varinjlim_{i} \Hom_\dbmodLA(\bull M,N_i),$$\\[-2ex]
where the limit in the right hand side stands for the direct limit of the system $(\Omega_i, \phi_{ij})$ constructed above.
\end{theorem}
\begin{proof}
  Let us define $s_i\:\bull N \to N_i$ recursively for $i\ge 0$ by setting $s_0$ to be the identity morphism $s_0 = \id_{\bull N}$ and $s_{i+1} = p_i s_i$ for $i\ge 0$.
  All $s_i$ belong to the localizing set $\SS$ from the definition of $\dsgLA$ since they are the compositions of $p_i$ which belong to $\SS$ by definition.
  
  Let us define the maps $\psi_i\: \Hom_\dbmodLA(\bull M, N_i) \to \Hom_\dsgLA(\bull M, \bull N)$ by sending a morphism $a\: \bull M \to N_i$ in $\dbmodLA$ to the morphism in $\dsgLA$ represented by the roof $\bull M \xrightarrow{a} N_i \xleftarrow{s_i} \bull N$. Since the roofs $\bull M \xrightarrow{a} N_i \xleftarrow{s_i} \bull N$ and $\bull M \xrightarrow{p_i a} N_{i+1} \xleftarrow{p_i s_i} \bull N$ represent the same morphism in $\dsgLA$, the system of maps $(\psi_i)_{i\ge0}$ is compatible with $(\phi_{ij})_{0\le i\le j}$. This situation can be schematically depicted as follows:
  $$ \left( \begin{tikzcd}
    \bull M \ar[d, dashed, "\Db"']
    \ar[rd, dashed, near start] 
    \ar[rrd, bend left=10, dashed, near start, "\scalebox{0.75}{$\Db$}"']
    \ar[rrrd, bend left=15, dashed, "\scalebox{0.75}{$\Db$}"']
    \ar[rrrrd, bend left=15, dotted]\\
    N_0 \ar[r, "p_0"] & N_1 \ar[r, "p_1"] & N_2 \ar[r, "p_2"] & N_3 \ar[r, dotted] & \cdots \\
    \bull N \ar[u, equal] \ar[ru, "s_1"] \ar[rru, bend right=10, "s_2"] \ar[rrru, bend right=15, "s_3"] \ar[rrrru, bend right=15, dotted]
  \end{tikzcd}  \right) 
  \xlongrightarrow{(\psi_i)_{i\ge 0}}
  \begin{tikzcd}
    \bull M \ar[dd, dashed, "\Dsg"] \\ \phantom{N_0} \\ \bull N
  \end{tikzcd}
  $$
Taking the limit $\varinjlim_i (\psi_i)_{i\ge 0}$, we then get the map
$$ \Psi\:\varinjlim_{i} \Hom_\dbmodLA(\bull M,N_i) \to \Hom_\dsgLA(\bull M,\bull N). $$\\[-3ex]
We will now show that this map $\Psi$ is bijective and thus will obtain the proof.

Firstly, let us prove that $\Psi$ is surjective. Consider an arbitrary element in $\Hom_\dsgLA(\bull M,\bull N)$ represented by a roof $\bull M \xrightarrow{a} T \xleftarrow{t} \bull N$, where $C = \Cone(t)$ is perfect. By definition, for any projective module $P$, we have $\Hom_\Db(P,X) = 0$ as soon as the complex $X$ is situated in negative degrees. This implies that for the perfect complex $C$, we have $\Hom_\Db(C,X) = 0$ as soon as the complex $X$ is situated in low enough degrees depending on $C$. In particular, for large enough~$i$, we have both $\Hom_\Db(C,N_i) = 0$ and $\Hom_\Db(C,N_i[1]) = 0$, since $N_i$ form a descending sequence. Then applying functor $\Hom_\Db(-,N_i)$ to the distinguished triangle $C[-1] \to \bull N \xrightarrow{t} T \to C$ in $\dbmodLA$, we see that the composition with $t$ defines an isomorphism $\Hom_\Db(\bull N,N_i) \simeq \Hom_\Db(T,N_i)$. This means that $s_i = ut$ for some $u\: T\to N_i$. Thus our initial roof is equivalent to the roof $\bull M \xrightarrow{ua} N_i \xleftarrow{s_i=ut}\bull N$. The latter is contained in the image of $\psi_i$ and thus in the image of $\Psi$.

\begingroup
\def\j{{j}}
Now, let us prove that $\Psi$ is injective. Consider an element $a$ such that $\Psi(a)=0$. This element is represented by a morphism $a_\j\:\bull M \to N_\j$ in $\dbmodLA$ for some $j$. Then $\psi_\j(a_\j) = \Psi(a) = 0$. The latter means that the roof $\bull M \xrightarrow{a_\j} N_\j \xleftarrow{s_\j} \bull N$ is equivalent to a zero roof $\bull M \xrightarrow{0} N_\j \xleftarrow{s_\j} \bull N$. This in turn means that there exists a morphism $t\:N_\j \to T$ in $\dbmodLA$ such that $ta_\j = 0$ and $\Cone(t)$ is perfect. Applying the same argument for $t$ as in the proof of surjectivity, while replacing $\bull N$ by $N_\j$, we get an isomorphism $\Hom_\Db(N_\j,N_i) \simeq \Hom_\Db(T,N_i)$ for large enough $i$ defined by composing with $t$. In particular, for the composition $s_{i\j} := p_{j-1}\dots p_i \: N_i \to N_\j$, there exists $u\: T \to N_i$ such that $s_{i\j} = ut$. Then $s_{i\j} a_\j = ut a_\j = u0 = 0$ which means that $a_\j$ is equivalent to zero in the direct limit, hence $a=0$.
\endgroup
\end{proof}

\subsection{Graded/ungraded correspondence}
\label{sec:gr/ungr}

\def\Lgr{{L\textrm{-gr}}}
\def\HomLgr{\Hom^\Lgr}
\def\HomLgrA{\Hom_{\modLA}}

For an $L$-graded ring~$A$ and two $L$-graded $A$-modules~$M$ and~$N$, we can consider two kinds of the morphism spaces from $M$ to $N$. The first kind respects the grading and is denoted by $\HomLgrA(M,N)$. We will call its elements the \textit{graded morphisms}. The second kind ignores the grading and is denoted by $\Hom_A(M,N)$. We will call its elements the \textit{ungraded morphisms}. These two kinds of spaces are related via the following well-known correspondence.

\begin{lemma}[{e.g., see \cite{lam}}]
    If $A$ is Noetherian and $M$ is finitely generated, then the ungraded morphisms from $M$ to $N$ decompose into the graded morphisms from $M$ to the various shifts of $N$ with respect to the $L$-grading:
    \begin{equation*}
        \Hom_A(M,N) = \bigoplus_{l\in L} \HomLgrA(M,N(l)),
    \end{equation*}
    where the grading on the module $N(l)$ is defined by setting $N(l)_m := N_{l+m}$ for all $l,m\in L$.
\end{lemma}

The above correspondence naturally transfers to the setting of bounded derived categories where it compares the morphisms in the graded bounded derived category $\dbmodLA$ with the morphisms in the ungraded one $\dbof{A}$ as follows.

\begin{lemma}
\label{lemma:gr/ungr-dbmod}
For any two bounded complexes $\bull M$ and $\bull N$ of finitely generated $L$-graded $A$-modules, the morphisms between them in the graded and ungraded bounded derived categories are related as follows:
\begin{equation*}
  \Hom_{\dbof A}(\bull M,\bull N) = \bigoplus_{l\in L}\Hom_\dbmodLA(\bull M, \bull N(l)).
\end{equation*}
\end{lemma}
\begin{proof}
  Consider an $L$-graded projective resolution $\bull P \to \bull M$ consisting of the finitely generated modules. Lemma~\ref{lemma:gr/ungr-dbmod} implies in a straightforward way that\\[-1ex]
\begin{equation*}
  \Homb_A(\bull P,\bull N) = \bigoplus_{l\in L}\Homb_\modLA
  (\bull P,\bull N(l)).
\end{equation*}
Taking $\HH^0$ from both sides gives us the desired result.
\end{proof}

Since the perfect complexes are stable under the grading-shift functors~$(l)$ in $\dbmodLA$, these functors are well-defined on the quotient categories $\dsgLA$ as well. Moreover, the explicit characterization of the morphisms obtained in~\S\ref{sec2:explicit-morphisms} allows us to automatically extend the graded/ungraded correspondence to the categories of singularities as follows.

\begin{prop}
\label{prop:gr/ungr-dsg}
  For any two bounded complexes $\bull M$ and $\bull N$ of finitely generated $L$-graded $A$-modules, the morphisms between them in the graded and ungraded categories of singularities are related as follows:
  \begin{equation*}
    \Hom_{\dsgA}(\bull M, \bull N) = \bigoplus_{l\in L} \Hom_{\dsgLA}(\bull M, \bull N(l)).
  \end{equation*}
\end{prop}
\begin{proof}
  Consider some descending sequence for $\bull N$ in the graded bounded derived category $\dbmodLA$. This sequence will be descending if considered in the ungraded category $\dbof{A}$ as well. Let us now apply Theorem~\ref{theorem:dsg-morphisms:db->dsg} for both ungraded and graded settings. Using Lemma~\ref{lemma:gr/ungr-dbmod} and passing to the limit in the corresponding direct systems, we get the desired correspondence.
\end{proof}

\begin{remark}
  Lemma~\ref{lemma:gr/ungr-dbmod} and Proposition~\ref{prop:gr/ungr-dsg} show that there are natural $L$-gradings on the spaces of morphisms in the ungraded bounded derived category $\dbof A$ and in the ungraded category of singularities $\dsgA$ as soon as the objects themselves are endowed with an $L$-grading. The resulted $L$-gradings are canonical in a sense that they do not depend on the choice of projective resolution in the proof of Lemma~\ref{lemma:gr/ungr-dbmod} and on the choice of descending sequence in the proof of Proposition~\ref{prop:gr/ungr-dsg}. These gradings are also functorial.
\end{remark}

\subsection{Stabilization for Gorenstein rings}
\label{sec:stabilization-gorenstein-rings}

Here, we refine the results in \cite{orlov05} and formulate the stabilization property for Gorenstein rings, that is the rings that have finite injective dimension as modules over themselves. In Proposition~\ref{prop:stabilization}, we show that over such rings, the morphisms from $M$ to $N[i]$ in the categories of singularities $\dsgLA$ stabilize to the corresponding morphisms in the derived category $\dbmodLA$ for large enough $i$. Thus we essentially prove Buchweitz's Remark for Gorenstein rings (see Proposition~\ref{prop:Buchweitz's-Remark}).

\begin{definition}
  The \textit{injective dimension} of an $A$-module $M$ is defined to be the minimal integer $n$, such that $\Ext^i_A(X,M) = 0$ for all $i > n$ and all modules~$X$. (It is equal to~$\infty$, if such~$n$ does not exist.)
\end{definition}

\begin{definition}
  A ring~$A$ is \textit{Gorenstein} if it has finite dimension as the module over itself, that is if there exists such integer~$n$ that $\Ext^i_A(X,A) = 0$ for all $i>n$ and all modules~$X$.
\end{definition}

The following lemma is the refinement of \cite[Proposition~1.11]{orlov05}.

\begin{lemma}
\label{lemma:for-stabilization}
  Let a module $M$ satisfies the condition for some integer $n$ that $\Ext^i_A(M,A) = 0$ for all $i > n$. Then for any complex $\bull N$ situated in nonpositive degrees, we have the following isomorphisms for all $i > n$:
  $$ \Hom_\dsgLA(M,\bull N[i]) \simeq \Hom_\dbmodLA(M, \bull N[i]). $$
\end{lemma}
\begin{proof}
    Let us fix $i > n$ and apply Theorem~\ref{theorem:dsg-morphisms:db->dsg} for the complexes $\bull M = M$ and $\bull N[i]$, and some descending sequence $(N_j,p_j)_{j\ge 0}$ of $\bull N[i]$. Since $\bull N[i]$ is situated in degrees $\le -i$ and $-i < -n$, we conclude that all $N_j$ and $\Cone(p_j)$ are situated in degrees $< -n$ for $j\ge 0$. Let us prove that all the maps $\phi_{j,j+1}$ are isomorphisms. This will imply that the limit map $\Psi$ as well as all the maps $(\psi_j)_{j\ge0}$ from the proof of the theorem are isomorphisms as well. In particular, $\psi_0$ will give us the desired isomorphism.
    
    The condition on $M$ implies that $\Hom_\Db(M,X) = 0$ as soon as $X$ is a perfect complex situated in degrees strictly less than $-n$. In particular, this applies for $X = \Cone(p_j)$ and $X=\Cone(p_j)[-1]$ for all $j\ge 0$. Applying the functor $\Hom_\Db(M,-)$ to the distinguished triangle $\Cone(p_j)[-1] \to N_j \xrightarrow{p_j} N_{j+1} \to \Cone(p_j)$, we get a desired isomorphism $\Hom_\Db(M,N_j) \simeq \Hom_\Db(M,N_{j+1})$ for each $j \ge 0$. Hence the proof.
\end{proof}

\begin{remark}
  The result of Proposition~1.11 in \cite{orlov05} can be viewed as the boundary case of $i = n$ for the above lemma.
  In this case, the morphism $p_j$ from the proof of the lemma induces an isomorphism in the direct system starting from $j=1$. The morphism $p_0$, however, induces only a surjection, thus producing the description of $\Hom_\Dsg(M,N[i])$ as the quotient of $\Hom_\Db(M,N[i])$ as in \cite{orlov05}.
\end{remark}

\begin{prop}[Stabilization]
\label{prop:stabilization}
  Let a ring $A$ have finite injective dimension~$n$ as a module over itself. Then for any $i>n$, we have the following isomorphism for all $A$-modules $M$ and $N$:
  $$ \Hom_\dsgLA(M,N[i]) \simeq \Hom_\dbmodLA(M,N[i]). $$
\end{prop}
\begin{proof}
  Follows immediately from Lemma~\ref{lemma:for-stabilization}.
\end{proof}

If we replace the modules $M$ and $N$ by the complexes, we get the statement of Buchweitz's Remark as follows.

\begin{prop}[Buchweitz's Remark for Gorenstein rings]
\label{prop:Buchweitz's-Remark}
If $A$ is Gorenstein, then for any two complexes $\bull M$ and $\bull N$ we have the following isomorphisms for large enough $i$:
  $$ \Hom_\dsgLA(\bull M,\bull N[i]) \simeq \Hom_\dbmodLA(\bull M,\bull N[i]). $$
\end{prop}
\begin{proof}
  By shifting the complexes if necessary, we can assume that $\bull M$ is situated in nonnegative degrees and $\bull N$ is situated in nonpositive degrees. Then we can repeat the argument of Lemma~\ref{lemma:for-stabilization} by replacing $M$ by $\bull M$. Indeed, such complex will also satisfy the condition that $\Hom_\Dsg(\bull M,A[i])=0$ for $i>n$ as soon as $A$ has finite injective dimension~$n$. Then to get the desired isomorphism we just need to take large enough $i$ such that the higher bound of $\bull N[i]$ is separated from the lower bound of $\bull M$ by at least $n$.
\end{proof}

\begin{remark}
\label{remark:counter-example-for-Buchweitz}
  Proposition~\ref{prop:Buchweitz's-Remark} doesn't hold if we remove the Gorenstein condition. For example, consider a field~$\kk$, the ring $A=\kk[x,y]/(xy,x^2)$ and the module~$\kk$. It is easy to see that $\Homb_{\dbof{A}}(\kk,\kk[i])$ are all finite-dimensional vector spaces over~$\kk$. However, when we descend $\kk$ by Lemma~\ref{lem:descend}, we get $\kk[1]\oplus M[1]$ as the descent (the first syzygies of $\kk$), where $M$ is the ring $\kk[y]$ considered as an $A$-module. Similarly, the descent of $M$ is $\kk[1]$. So, the terms of the descending sequence of $\kk$ will consist of the increasing number of summands of the forms $\kk[n]$ and $M[n]$. Hence when we construct the direct system as in~\eqref{eq:omega-space} and~\eqref{eq:omega-space-phi} for $\bull M=\kk$ and $\bull N=\kk[i]$ (for any~$i$), we will get the increasing sequence of vector spaces. Its limit forms an infinite-dimensional vector space $\Hom_{\dsgA}(\kk,\kk[i])$ by Theorem~\ref{theorem:dsg-morphisms:db->dsg}. The latter cannot be equal to a finite-dimensional space $\Hom_{\dbof{A}}(\kk,\kk[i])$.
\end{remark}

\subsection{Hypersurface singularities}
\label{sec:hypersurface-singularities}

Let $R$ be a commutative regular ring of finite Krull dimension $n$ equipped with an $L$-grading $R = \bigoplus_{l\in L} R_l$, and $w\in R$ be a homogeneous element. We denote by $\w \in L$ the degree of $w$, that is $w \in R_{\w}$. We further assume that $w$ is not a zero-divisor in $R$, that is that the multiplication by $w$ defines an injective map on~$R$. Under these assumptions, the quotient ring $A = \Rw$ is naturally $L$-graded and Noetherian and so we can apply the results of the previous sections.

Standard considerations of commutative algebra tell us that the injective dimension of $R/(w)$ under the above assumptions is at most $n-1$. This is summarized in the next lemma.

\begin{lemma}
\label{lemma:inj.dim(R/w)}
  Let $R$ and $w$ be as described above. Then for any $\Rw$-module $M$, we have $\Ext^i_{\Rw}(M,\Rw) = 0$, where $i \ge n$.
\end{lemma}
\begin{proof}
  Indeed,
  by Lemma~2(i) in \cite[\S18]{matsumura},
  we have $\Ext^i_{\Rw}(M,\Rw) \simeq \Ext^{i+1}_R(M,R)$, where in the right-hand side we consider $M$ as an $R$-module with a trivial action of $w$. Since the ring $R$ is regular, its homological dimension equals its Krull dimension which is $n$. In particular, $\Ext^{i+1}_R$ vanishes for $i \ge n$. Hence the proof.
\end{proof}

Proposition~\ref{prop:stabilization} and Lemma~\ref{lemma:inj.dim(R/w)} imply the following.

\begin{corollary}
\label{cor:stab-for-hypersurf}
  If $R$ and $w$ are as above and $n = \krdim R$, then
  $$ \Hom_\dsglrw(M,N[i]) \simeq \Hom_\dblrw(M,N[i]) $$
  for any two $\Rw$-modules $M$ and $N$ and $i \ge n$.
\end{corollary}


The category $\dsglrw$ of hypersurface singularities satisfies the following quasi-periodicity property.

\begin{lemma}[Quasi-periodicity]\label{lemma:quasi-period}
  Let $R$ and $w$ be as above. Then in the category $\dsglrw$, the double shift functor and graded $\w$-shift functors are equivalent:
  \begin{equation}\label{eq:quasi-period}
      [2] \simeq (\w).
  \end{equation} 
\end{lemma}
\begin{proof}
  The foundational results by Orlov state that the category of singularities is equivalent to the corresponding category of matrix factorizations (see \cite[Theorem~3.9]{orlov03} and \cite[Theorem~3.5]{orlov11}).
  However, for the matrix factorizations, such functors are equivalent by definition. Hence the proof.
\end{proof}

Combining Corollary~\ref{cor:stab-for-hypersurf} with the quasi-periodicity property we get the following.

\begin{theorem}\label{theorem:homs-dsg-hyper:ultimate}
  Let $R$, $w$ and $n = \krdim R$ be as above. Then for any two $\Rw$-modules $M$ and $N$ and any $i\in\ZZ$, we have
  $$ \Hom_\dsglrw(M,N[i]) \simeq \Hom_\dblrw(M,N(-j\w)[i+2j]) $$
  for any integer $j \ge \frac{n-i}2$.
\end{theorem}
\begin{proof}
  Indeed, by Lemma~\ref{lemma:quasi-period}, the objects $N[i]$ and $N(-j\w)[i+2j]$ are isomorphic in $\dsglrw$, so we just apply Corollary~\ref{cor:stab-for-hypersurf} for the morphisms from $M$ to $N(-j\w)[i+2j]$.
\end{proof}

\begin{remark}
  Since Theorem~\ref{theorem:homs-dsg-hyper:ultimate} does not impose any conditions on~$i$, we can use it to compute the morphisms in $\dsglrw$ from $M$ to \textit{every} shift of~$N$. We will do this in the applications.
\end{remark}

\subsubsection{Useful lemma}

The following lemma gives us examples of modules which become zero in the corresponding category of singularities. We will use such modules a lot in Section~3 when proving that our collections are full and generating new modules from the given ones.

\begin{lemma}\label{lemma-perfect}
    Consider a finitely generated $R$-module~$K$ (which is not necessarily graded). If the multiplication by $w$ defines an injective map $K \xrightarrow{w\cdot} K$, then $K/wK$ is a perfect $\Rw$-module.
\end{lemma}
\begin{proof}
    By assumption, $w$ acts injectively on the ring $R$ itself, hence it also acts injectively on any projective $R$-module $P$. It is also clear that $P/wP$ is a projective $R/(w)$-module. Thus the statement is true when the module $K$ is projective.

    Since $R$ is regular, its global homological dimension agrees with its Krull dimension $n$. In particular, any module $K$ has finite projective dimension $k$ over $R$, where $k \le n$. Let us prove the statement by induction on $k$. If $k=0$, then $K$ is a projective $R$-module and we already proved the statement in this case. Now assume that $k > 0$. Cover $K$ by a projective module $P$ and consider the kernel $K'$ of the covering map:
    $ 0 \to K' \to P \to K \to 0. $
    We see that $K'$ has projective dimension $k-1$. Moreover, since $K'$ is a submodule of a projective module, $w$ acts injectively on it as well. The induction hypothesis for $K'$ then says that $K'/wK'$ is perfect over $\Rw$, thus has a finite projective resolution.
    Applying the snake lemma to the diagram
    $$ \begin{tikzcd}
      0 \ar[r] & K' \ar[r] \ar[d, "w", hook] & P \ar[r] \ar[d, "w", hook] & K \ar[r] \ar[d, "w", hook] & 0 \\
      0 \ar[r] & K' \ar[r] & P \ar[r] & K \ar[r] & 0,
    \end{tikzcd} $$
    we get a short exact sequence $0 \to K'/wK' \to P/wP \to K/wK \to 0$ of $\Rw$-modules. By gluing it with the projective resolution for $K'/wK'$, we get a finite projective resolution for $K/wK$, thus proving that $K/wK$ is perfect as well.
\end{proof}

\subsection{Applications}
\label{sec:morphisms-applications}

The above results can be applied for making efficient computations in the graded categories of \textit{hypersurface} singularities~$\dsgof{L}{R/(w)}$, that is the categories of matrix factorizations, where $R$, $w$ and $L$ satisfy the same assumptions as in~\S\ref{sec:hypersurface-singularities}.

\subsubsection{Computing morphisms}

We can compute the morphism spaces between the objects of~$\dsgof{L}{R/(w)}$ as in the following step-by-step recipe.

\begin{recipe}
\label{recipe}
  Given two $L$-graded $R/(w)$-modules $M$ and $N$, we can compute the morphisms $\Hom_{\dsgof{L}{\Rw}}(M,N(l)[i])$ simultaneously for all $l\in L$ and $i\in\ZZ$ by performing the following 4~steps:
\begin{enumerate}[itemsep=0em, label={\textit{Step \arabic*.}}, labelwidth=\widthof{\textit{Step 10.}}, leftmargin=!, topsep=1ex]
  \item Build a graded projective resolution $\bull P$ of $M$.
  \item Compute $\Hom_{\dbof\Rw}(M,N[i]) = \HH^i(\Homb_\Rw(\bull P,N))$ for $i\in\ZZ$, while keeping track of the inherent $L$-grading, as in Lemma~\ref{lemma:gr/ungr-dbmod}.
  \item Use Theorem~\ref{theorem:homs-dsg-hyper:ultimate} and the previous computations to get the \textit{ungraded} morphisms $\Hom_{\dsgof{}\Rw}(M,N[i])$ for all $i\in \ZZ$, while keeping track of the $L$-grading.
  \item Use Proposition~\ref{prop:gr/ungr-dsg} to decompose the ungraded morphisms into the \textit{graded} ones, thus getting $\Hom_{\dsgof L\Rw}(M, N(l)[i])$ for all $l \in L$ and $i\in \ZZ$ at once.
\end{enumerate}
\end{recipe}

\begin{remark}
  Due to the estimates in Theorem~\ref{theorem:homs-dsg-hyper:ultimate}, it is enough to make computations only for $i = n$ and $i = n+1$, where $n = \krdim R$. Then we can use quasi-periodicity to recover the morphisms for all other $i$. This works because the minimal projective resolution of any module becomes 2-quasi-periodic after $i=n$.
\end{remark}

\begin{remark}
  We formulated the recipe for the case when~$M$ and~$N$ are modules so that to use the precise estimates as in the previous remark. If~$M$ and~$N$ are arbitrary objects in $\dsgof{L}{R/(w)}$, then we firstly need to represent them as the shifts of some modules by Lemma~\ref{lem:everything-is-a-shift-of-a-module} and then apply the recipe for those modules.
\end{remark}

\begin{example}
\label{example:morphisms-k[x]/(x^p)}
  Let $\kk$ be an arbitrary field. Consider the regular ring $R = \kk[x]$ and $w = x^p$, where $p \ge 2$. We equip $R$ with an arbitrary $L$-grading making $x$ homogeneous (we also assume that all scalars have zero degree). We denote $\x = \deg x \in L$ and $\w = \deg w = p \x$ to be the degrees of $x$ and $w$ respectively. This induces an $L$-grading on the quotient ring $R/(w) = \kk[x]/(x^p)$. Consider the scalar field $\kk$ as a graded $R/(w)$-module situated in $L$-degree zero and its grading shifts $\kk(l)$ for all $l\in L$ which are situated in $L$-degrees $-l$. Let us use Recipe~\ref{recipe} in order to compute the morphisms $\Hom_{\dsgof{L}{\kk[x]/(x^p)}}(\kk,\kk(l)[i])$ for different $i\in\ZZ$ and $l\in L$ all at once.
  
\vspace{0.5ex}\noindent\textit{Step 1.}
  For the sake of brevity, let us denote $A=\kk[x]/(x^p)$. Then the following is a projective resolution of $\kk$ in the category of graded $A$-modules:
  $$ \bull P = \left( \dots \to A(-2\w) \xrightarrow{x^{p-1}} A(-\w-\x) \xrightarrow{x} A(-\w) \xrightarrow{x^{p-1}} A(-\x) \xrightarrow{x} A \right) \xrightarrow{x\mapsto 0} \kk.$$

\noindent\textit{Step 2.}
  Since $\Hom_A(A(l),\kk) = \Hom_A(A,\kk(-l))=\kk(-l)$ for all $l\in L$, the corresponding $\ZZ$-$L$-graded $\Homb$-complex is
  $$ \Homb_A(\bull P,\kk) = \Bigl( \dots \gets \kk(2\w) \xleftarrow{0} \kk(\w+\x) \xleftarrow{0} \kk(\w) \xleftarrow{0} \kk(\x) \xleftarrow{0} \kk \Bigr). $$
  This complex has a natural 2-$\w$-periodic structure inherited from $\bull P$. After passing to homologies, we get the following:
  \begin{align*}
    \makebox[0pt][l]{$\Hom_{\dbof A}(\kk,\kk[i])$}\phantom{\Hom_{\dbof A}(\kk,\kk[2j+1])} &= 0, \text{ for } i<0, \\
    \makebox[0pt][l]{$\Hom_{\dbof A}(\kk,\kk[2j])$}\phantom{\Hom_{\dbof A}(\kk,\kk[2j+1])} &= \kk(j\w), \text{ for } j \ge 0, \\
    \Hom_{\dbof A}(\kk,\kk[2j+1]) &= \kk(j\w+\x), \text{ for } j \ge 0.
  \end{align*}

\vspace{0.5ex}\noindent\textit{Step 3.}
  Since $R = \kk[x]$ is a regular ring of Krull dimension $1$, we can apply Theorem~\ref{theorem:homs-dsg-hyper:ultimate} to immediately get all the needed morphisms in the ungraded category of singularities $\dsgof{}{A} = \dsgof{}{\kk[x]/(x^p)}$:
  \begin{align*}
    \makebox[0pt][l]{$\Hom_{\dsgof{}{A}}(\kk,\kk[2j])$}\phantom{\Hom_{\dsgof{}{A}}(\kk,\kk[2j+1])} &= \kk(j\w),\\
    \Hom_{\dsgof{}{A}}(\kk,\kk[2j+1]) &= \kk(j\w+\x).
  \end{align*}

\vspace{0.5ex}\noindent\textit{Step 4.}
  We can notice that $\kk(j\w)$ and $\kk(j\w+\x)$ are 1-dimensional $L$-graded vector spaces situated in $L$-degrees $-j\w$ and $-j\w-\x$ respectively. Thus, applying Proposition~\ref{prop:gr/ungr-dsg}, we get the following morphisms between the different shifts of $\kk$ in the graded category of singularities $\dsgLA=\dsgof{L}{\kk[x]/(x^p)}$:
\begin{align*}
  &\Hom_\dsgLA(\kk,\kk(l)[2j]) &&=
  \begin{cases}
    \kk & \text{if } l = -j\w\\
    0 & \text{otherwise}
  \end{cases} \\
  &\Hom_\dsgLA(\kk,\kk(l)[2j+1]) &&=
  \begin{cases}
    \kk & \text{if } l = -j\w-\x \\
    0 & \text{otherwise}
  \end{cases}
\end{align*}
\end{example}

\begin{remark}
\label{remark:n=1:k-exceptional}
In particular, we see that $\kk$ is an exceptional object in $\dsgLA$ as soon as $\x$ has an infinite order in $L$. Otherwise, if $j\x = 0$ in $L$ for some $j\ne 0$, then $\Hom_\dsgLA(\kk,\kk[2j]) = \kk$, since $-j\w=-jp\x=0$.
For example, in the case of the ``maximal'' grading group $L = \ZZ$ and $\x = 1$, the module $\kk$ and all of its shifts will be exceptional objects in $\dsgLA$. Also, we see that two different grading shifts $\kk(l_1)$ and $\kk(l_2)$ are orthogonal if and only if $l_2 - l_1$ does not belong to the set $\{j\w +\varepsilon\x\mid j\in \ZZ, \varepsilon = -1,0,1\}$. 
\end{remark}

\begin{example}
  Consider $R = \kk[x]$ and $w = x$. In this case, $A = \kk[x]/(x) \simeq \kk$, so the module $\kk$ is projective, hence it is a zero object in $\dsgLA$. We can also see it by computing the morphisms $\Hom_{\dsgLA}(\kk,\kk)$ as in Example~\ref{example:morphisms-k[x]/(x^p)}. Indeed, since $\Hom_{\dbof A}(\kk,\kk[i]) = \Ext^i(\kk,\kk)=0$ for $i>0$, the stabilization and quasi-periodicity phenomena tell us that $\Hom_{\dsgof{}A}(\kk,\kk[i]) = 0$ for any integer $i$. In particular, $\Hom_{\dsgLA}(\kk,\kk) = 0$ regardless of the grading in question, which confirms that $\kk$ is a zero object. By generation criterion of Proposition~\ref{prop:gen-criterion-invertibles-general}, this immediately implies that the whole category $\dsgLA$ is zero as well, as was expected from the fact that the corresponding subscheme $\{x=0\}$ in $\mathbb{A}^1_\kk$ is smooth (see \cite{orlov03}).
\end{example}

\subsubsection{Quasi-periodic localization}

Here, we take another look on the definition of a category of singularities as the localized bounded derived category. We notice that this localization can be performed by mere inversion of a certain natural transformation connecting the functors $[2]$ and $(\w)$ in the bounded derived category $\dbof{\rmodgr{L}{\Rw}}$. This transformation then becomes a quasi-periodicity equivalence~\eqref{eq:quasi-period} in the corresponding category of singularities $\dsgof{L}{R/(w)}$.

Indeed, by unfolding the proof of the quasi-periodicity property of Lemma~\ref{lemma:quasi-period}, we get the following statement.

\begin{lemma}
\label{lemma:canonical-s}
  For any $L$-graded $\Rw$-module~$N$, there exists a canonical descending morphism $s\:N\to N(-\w)[2]$ in $\dbof{\rmodgr{L}{\Rw}}$.
\end{lemma}

\begin{proof}
  Consider $M$ as an $R$-module with a trivial action of $w$. Cover it by a projective $R$-module $P$ and consider the corresponding short exact sequence $\bigl(0 \to K \to P \to N \to 0\bigr)$ of $R$-modules. Applying the snake lemma to the commutative diagram
  $$ \begin{tikzcd}
    0 \ar[r] & K(-\w) \ar[r] \ar[d, "w", hook] & P(-\w) \ar[r] \ar[d, "w", hook] & N(-\w) \ar[r] \ar[d, "0"] & 0 \\
    0 \ar[r] & K \ar[r] & P \ar[r] & N \ar[r] & 0
  \end{tikzcd} $$
  we get an exact sequence
  $ ( 0 \to N(-\w) \to K/wK \to P/wP \to N \to 0 ) $ of $\Rw$-modules.
  This sequence defines a morphism $s\:N\to N(-\w)[2]$ in the derived category such that $\Cone(s)$ is isomorphic to the complex $(K/wK \to P/wP)$ situated in degrees $-2$ and $-1$. The latter complex is perfect, since both its components are perfect by Lemma~\ref{lemma-perfect}. Hence the proof.
\end{proof}

\begingroup
\def\myA{{R/(w)}}
The following is the direct corollary.

\begin{lemma}
\label{lemma:canonical-s-sequence}
  Every graded module $N$ admits the canonical descending sequence as follows:\\[-1ex]
    $$ N \xrightarrow{s} N(-\w)[2] \xrightarrow{s(-\w)[2]} N(-2\w)[4] \xrightarrow{s(-2\w)[4]} N(-3\w)[6] \to \dots $$
\end{lemma}

Let us now consider $\bull\Ext_\myA(M,N) := \Hom_{\dbof{\Rw}}(M,N[\bullet])$ as a $\ZZ$-graded complex. By Lemma~\ref{lemma:gr/ungr-dbmod}, the latter is also equipped with an $L$-grading, thus becoming a $\ZZ$-$L$-bigraded space. Similarly, the space $\bull\Ext_\myA(N,N)$ of the \textit{self}-extensions of $N$ has a natural structure of a $\ZZ$-$L$-graded \textit{ring}, where the multiplication is given by the composition of the corresponding morphisms. Moreover, $\bull\Ext_\myA(M,N)$ is a right module over $\bull\Ext_\myA(N,N)$. In this language, the descending morphism $s$ of Lemma~\ref{lemma:canonical-s} can be treated as a homogeneous element of $\bull\Ext_\myA(N,N)$ of bidegree $(2,-\w)$. Then Theorem~\ref{theorem:dsg-morphisms:db->dsg} for the descending sequence of Lemma~\ref{lemma:canonical-s-sequence} implies the following.

\begin{prop}
  Let $M$ and $N$ be two graded $\myA$-modules and the descending morphism $s\:N\to N(-\w)[2]$ be as above. Then we have the following isomorphism of two $\ZZ$-$L$-bigraded spaces:
  $$ \Hom_{\dsgof{}{\myA}}(M,N[\bullet]) \simeq \bull\Ext_\myA(M,N)[s^{-1}], $$
  where the right-hand side is the graded localization of the module $\bull\Ext_\myA(M,N)$ obtained by inverting the element $s$ of the ring $\bull\Ext_\myA(N,N)$.
\end{prop}

In other words, the quotient functor $\dbof{\myA}\to\dsgof{}{\myA}$ from the definition of $\dsgof{}{\myA}$ coincides with the localization functor $[s^{-1}]$.

\begin{remark}
\label{remark:db-non-zero-morphism-in-dsg-criterion}
  The above descending sequence provides a simple criterion for when a given morphism from the bounded derived category is zero in the corresponding category of singularities. Indeed, by the above argument, a morphism $a\:M\to N$ in $\dbof{R/(w)}$ is zero in $\dsgof{}{R/(w)}$ if and only if $s^i a = 0$ in $\dbof{R/(w)}$ for some $i>0$. Moreover, by Theorem~\ref{theorem:homs-dsg-hyper:ultimate}, it is enough to check this equality only for $i$ equal to the smallest integer greater or equal to the half Krull dimension of $R$.
\end{remark}

\endgroup

\section{Exceptional collections for invertible polynomials}
\label{sec3}

\def\lwbar{\overline{L}_w}
\def\lvbar{\overline{L}_v}
\def\xx{{\bar x}}
\def\yy{{\bar y}}
\def\zz{{\bar z}}

\newcommand{\mat}[1]{\left(\begin{smallmatrix}#1\end{smallmatrix}\right)}

In this section, we consider the categories of singularities $\dsgof{L_w}{\kk[x_1,\dots,x_n]/(w)}$ defined by invertible polynomials $w\in \kk[x_1,\dots,x_n]$ with $n\le 3$. We explicitly construct full strongly exceptional collections in these categories on a case by case basis using the classification of invertible polynomials from Example~\ref{ex:invertible-classification-n<=3}. We summarize these results in Propositions~\ref{prop:main:n=2} and~\ref{prop:main:n=3} for $n=2$ and $n=3$ respectively. Together with Example~\ref{ex:main:n=1} for $n=1$, this constitutes the proof of Theorem~\ref{thm1}. The corresponding collections are depicted on Figures~\ref{fig:2split}--\ref{fig:2loop} for $n=2$ and on Figures~\ref{fig:n=3:split-split}--\ref{fig:n=3:loop-split},~\ref{fig:n=3:3-chain}--\ref{fig:n=3:3-loop} for $n=3$.

In \S\ref{sec3:preliminaries}, we recall the notion of invertible polynomial and a well-known generation criterion for isolated singularities. We also recall the results of Futaki and Ueda for Brieskorn-Pham singularities.

In \S\ref{sec:n=2} and \S\ref{sec:n=3}, we explicitly construct exceptional collections for the cases of $n=2$ and $n=3$ respectively. Both sections are structured similarly. In each case, we firstly describe the necessary properties of the maximal grading groups. Then we present the objects which we will be using to form the collections. Then we compute the morphism spaces between these objects by using Recipe~\ref{recipe} of Section~2. After that we describe the explicit collections. Then we check that the collections are indeed exceptional and strong based on our previous computations. After that, we show that the given collections generate all shifts of the base field~$\kk$ considered as an object of $\dsgof{L_w}{\kk[x_1,\dots,x_n]/(w)}$. Finally, we use generation criterion and properties of the grading group to conclude the proof in each case.

In \S\ref{sec3:higher-n}, we discuss how the collections should look like in general case based on the pattern for $n\le 3$.

\subsection{Preliminaries}
\label{sec3:preliminaries}

\subsubsection{Invertible polynomials}
\label{sec:classification-of-invertibles}
\label{sec:invertible}

\begin{definition}
\label{def:invertibles}
  An \textit{invertible polynomial} is the polynomial of type\\[-2ex]
  $$ w = \sum_{i=1}^n x_1^{a_{i1}} \dots x_n^{a_{in}} \in \kk[x_1,\dots,x_n], $$\\[-1.5ex]
  where $A = (a_{ij})_{i,j=1}^n$ is an invertible $n\times n$-matrix with integer entries. We also assume that $w$ has an isolated singularity at the origin.
\end{definition}

From now on, we assume that $\kk$ is an algebraically closed field of characteristic $0$. Under this assumption, the invertible polynomials can be classified as being the split sums of the following two atomic types (see \cite{classify-quasihom}):
\begin{itemize}
    \item \textit{\textbf{$n$-chain}}: $w = x_1^{p_1}+x_1 x_2^{p_2} + \dots + x_{n-1} x_n^{p_n}$, where $n \ge 1$ and $p_i \ge 2$;
    \item \textit{\textbf{$n$-loop}}: $w = x_n x_1^{p_1}+x_1 x_2^{p_2} + \dots + x_{n-1} x_n^{p_n}$, where $n \ge 2$ and $p_i \ge 2$.
\end{itemize}

\begin{example}\label{ex:invertible-classification-n<=3}
In the case of $n\le 3$, which will be our main case of study, the above classification gives us 9 types of invertible polynomials as follows:\\[1ex]
\begingroup
    \def\makeletter#1{{\def\temp{#1}{\if\temp1a\else\if\temp2b\else\if\temp3c\else\if\temp4d\else e\fi\fi\fi\fi}}}
    \newcommand{\wcase}[2]{{(#1\def\temp{#1}\if\temp1\else\makeletter#2\fi)}}
    \newcommand{\wdcase}[3]{(#1\makeletter#2-\makeletter#3)}
\noindent
$\left.\begin{minipage}{0.55\textwidth}
\noindent(1)\quad $w = x^p$ --- 1-chain
\end{minipage}\right.\Bigr\}$\qquad $n=1$,\; $p\ge 2$
\par\vspace{0.5ex}\noindent
$\left.\begin{minipage}{0.55\textwidth}
\noindent\wcase21\quad $w = x^p + y^q$ --- 1-chain + 1-chain\\[0.5ex]
\noindent\wcase22\quad $w = x^p + xy^q$ --- 2-chain\\[0.5ex]
\noindent\wcase23\quad $w = yx^p + xy^q$ --- 2-loop
\end{minipage}\right\}$\qquad $n=2$,\; $p,q\ge 2$
\par\vspace{1.5ex}\noindent
$\left.\begin{minipage}{0.72\textwidth}
\noindent\wcase31\quad $w = x^p + y^q + z^r$ --- 1-chain + 1-chain + 1-chain\\[0.5ex]
\noindent\wcase32\quad $w = x^p + xy^q + z^r$ --- 2-chain + 1-chain\\[0.5ex]
\noindent\wcase33\quad $w = yx^p + xy^q + z^r$ --- 2-loop + 1-chain\\[0.5ex]
\noindent\wcase34\quad $w = x^p + xy^q + yz^r$ --- 3-chain\\[0.5ex]
\noindent\wcase35\quad $w = zx^p + xy^q + yz^r$ --- 3-loop
\end{minipage}\right\}$\quad$\vcenter{\noindent$n=3$,\\ $p,q,r\ge 2$}$
\endgroup
\end{example}

\begin{definition}[Maximal grading group]
\label{def:max-grading-group}
Let $w \in \kk[x_1,\dots,x_n]$ be an invertible polynomial.
\begin{enumerate}[itemsep=0em, label={(\alph*)}, labelwidth=\widthof{ (a)}, leftmargin=!, topsep=1ex]
    \item We define $L_w$ as the maximal abelian grading group on $\kk[x_1,\dots,x_n]$ which makes all~$x_i$ and $w$ homogeneous, and which is generated by the degrees of~$x_i$.
    \item We define the quotient group $\lwbar = L_w/\<\w\>$, where $\w$ denotes the degree of~$w$ in $L_w$.
\end{enumerate}
\end{definition}

We can describe the group $L_w$ explicitly as the quotient of the free abelian group generated by the symbols $(\x_i)_{1\le i\le n}$, representing the degrees of the variables $x_i$, over the relations forcing all the monomials of $w$ to have the \textit{same} degree. Similarly, the group $\lwbar$ can be described by the generators $(\xx_i)_{1\le i\le n}$ and the relations stating that all the monomials of $w$ have \textit{zero} degree.

\begin{example*}
  If $w = x^2 + xy^4 + z^5$, then $L_w = \ZZ\<\x,\y,\z\>/\<2\x=\x+4\y=5\z\>$, $\w = 2\x = \x+4\y=5\z$, and $\lwbar = L_w/\<\w\>= \ZZ\<\xx,\yy,\zz\>/\<2\xx=\xx+4\yy=5\zz=0\>$.
\end{example*}

\begin{remark*}
Since the group $L_w$ is generated by $n$ variables and $n-1$ independent relations, it has always rank $1$, that is $L_w\simeq \ZZ \oplus T$, where $T$ is a finite abelian group. Similarly, the group $\lwbar$ is always finite, since it is generated by $n$ variables and $n$ independent relations.
\end{remark*}

\subsubsection{Generation criterion}

It is well-known that if a polynomial $w\in\kk[x_1,\dots,x_n]$ has an isolated singularity at a point, then the corresponding category of singularities is generated by the structure sheaf of this point up to taking the direct summands of the objects. In particular, this holds for invertible polynomials and can be adapted for the graded setting as follows.

\begin{prop}[{see \cite[Proposition~2.3.1]{PV-gen-crit}}]
\label{prop:gen-criterion-invertibles-general}
  Let $R=\kk[x_1,\dots,x_n]$ and $w\in R$ be an invertible polynomial. Let $\AA$ denote the minimal triangulated subcategory of $\dsglrww$ containing the grading shifts $\kk(l)$ for all $l \in L_w$. Then every object of $\dsglrww$ is a direct summand of an object of $\AA$.
\end{prop}

\begin{remark}
    In the original formulation in \cite{PV-gen-crit}, the base field is $\Cc$. However, the argument works for the case of an algebraically closed field $\kk$ of characteristic $0$. It also works for arbitrary grading group $L$ making $x_i$ homogeneous, and for arbitrary $L$-homogeneous polynomial $w$ defining an isolated singularity at the origin, where the latter condition reduces to the property that the localized rings $(R/(w))_{x_i}$ are regular for all $1\le i\le n$. Moreover, we can consider non-hypersurface singularities as well. In this case, we should just replace the ring $R/(w)$ by $R/I$, where $I$ is a homogeneous ideal contained in the maximal ideal of the origin $(x_1,\dots,x_n)\subseteq R$ such that $(R/I)_{x_i}$ are regular for all $1\le i\le n$.
\end{remark}

We will use the following variation of the above statement for proving the fullness of our exceptional collections.

\begin{prop}
\label{prop:gen-criterion-invertibles-collections}
  Let $R$ and $w$ be as before. If $\dsglrww$ admits an $\Ext$-finite exceptional collection $(E_1,\dots,E_n)$ generating $\kk(l)$ for all $l\in L_w$, then this collection is full:\\[-2ex]
  $$ \dsglrww = \<E_1,\dots, E_n\>. $$
\end{prop}
\begin{proof}
  Follows from Proposition~\ref{prop:gen-criterion-invertibles-general} and Lemma~\ref{lemma:exc-collection:summand-thick=>full}.
\end{proof}

\begin{remark}
\label{remark:generation-up-to-quasi-periodicity}
  This statement shows that in order to prove that a given collection is full, we just need to generate a shift $\kk(l)$ out of this collection for each $l\in L_w$. In fact, due to quasi-periodicity property for the categories of hypersurface singularities, it is enough to generate such shifts for all $l$ only up to the multiples of $\w$, since as soon as $\kk(l)$ belongs to a triangulated subcategory, all $\kk(l+i\w)\simeq\kk(l)[2i]$, where $i\in\ZZ$, belong to it automatically as well. This means that in total we need to generate only a finite number of shifts corresponding to the elements of the finite group $\lwbar = L_w/\<\w\>$.
\end{remark}

\subsubsection{Brieskorn-Pham singularities}

Futaki and Ueda in \cite{futaki-ueda} and \cite{ueda} studied the categories of singularities $\dsgof{L_w}{\kk[x_1,\ldots,x_n]/(w)}$ for Brieskorn-Pham polynomials~$w$. Such polynomials have the form $x_1^{p_1}+\ldots+x_n^{p_n}$ which is the split sum of $n$ monomials of 1-chain type, and where $p_i\ge 2$ for all $i$. They observed that these categories admit full exceptional collections of very simple form as follows.

\begin{prop}[see \cite{futaki-ueda} and \cite{ueda}]
\label{prop:futaki-ueda-split}
  Let the ring be $R=\kk[x_1,\dots,x_n]$, the potential be $w=x_1^{p_1}+\ldots+x_n^{p_n}$, and the grading group be the maximal one $L_w$. Then the corresponding graded category of singularities $\dsgof{L_w}{R/(w)}$ admits a full exceptional collection consisting of $(p_1-1)\cdot\ldots\cdot(p_n-1)$ objects of the form $\kk(-(j_1 \x_1 + \ldots + j_n \x_n))$ for $0\le j_i \le p_i-2$.
\end{prop}

In fact, one can also see that these collections are ``almost'' strong. To obtain a strong collection, one needs to replace the object $\kk(-(j_1 \x_1 + \ldots + j_n \x_n))$ by its shift $\kk(-(j_1 \x_1 + \ldots + j_n \x_n))[j_1+\ldots+j_n]$. These collections can be also arranged into nice $n$-dimensional cubes (more precisely, rectangular boxes) in which the morphisms are going only in nonnegative directions (see Figure~\ref{fig:2split} for $n=2$ and Figure~\ref{fig:n=3:split-split} for $n=3$).

\begin{example}[well-known]
\label{ex:main:n=1}
  In the case of $n=1$, the ring is $R=\kk[x]$, the potential is $w=x^p$, $p\ge 2$, and the maximal grading group is $L_w = \ZZ\<\x\>\simeq\ZZ$. Then the strong exceptional collection in the category of singularities $\dsgof{L_w}{\kk[x]/(x^p)}$ consists of $p-1$ objects of the form $\kk(-i\x)[i]$, where $0\le i\le p-2$, that are ordered as follows:
  \begin{equation}
  \label{eq:main:n=1:collection}
    \kk \to \kk(-\x)[1] \to \kk(-2\x)[2] \to \dots \to \kk(-(p-2)\x)[p-2].
  \end{equation}
\end{example}

    Futaki and Ueda proved Proposition~\ref{prop:futaki-ueda-split} by identifying the category of singularities $\dsgof{L_w}{\kk[x_1,\dots,x_n]/(w)}$ with a certain triangulated subcategory of the bounded derived category $\dbofgr{L_w}{(\kk[x_1,\dots,x_n]/(w))}$. In this way, they have checked that the collections are indeed exceptional by computing the corresponding morphisms in $\dbofgr{L_w}{(\kk[x_1,\dots,x_n]/(w))}$. We give an independent proof for this statement for $n\le3$ by explicitly computing the morphisms in~$\dsgof{L_w}{\kk[x_1,\dots,x_n]/(w)}$ via Recipe~\ref{recipe}.
    For example, for $n=1$, we can see that the collection \eqref{eq:main:n=1:collection} is indeed strongly exceptional from our explicit computations of the morphisms in Example~\ref{example:morphisms-k[x]/(x^p)} and Remark~\ref{remark:n=1:k-exceptional}.


\subsection{\texorpdfstring{$n=2$}{n=2}}
\label{sec:n=2}

\def\lwbar{\overline{L}_w}
\def\xx{\bar x}
\def\yy{\bar y}

For $n = 2$ and $R=\kk[x,y]$, as we saw in Example~\ref{ex:invertible-classification-n<=3}, there are three different types of invertible polynomials. Let us call them \textit{2-split} for $w = x^p+y^q$, \textit{2-chain} for $w = x^p+xy^q$ and \textit{2-loop} for $w = yx^p + xy^q$, where $p,q\ge 2$.

The explicit exceptional collections are described in Proposition~\ref{prop:main:n=2} and on Figures~\ref{fig:2split}--\ref{fig:2loop}.

\subsubsection{The maximal grading groups}

The maximal grading groups $L_w$ on $\kk[x,y]/(w)$ and the corresponding quotients $\lwbar = L_w / \la\w\ra$ are described in Table~\ref{tbl:grading-groups:n=2}.

\begingroup
\newcommand{\mystack}[1]{{
\begin{array}{c}#1\end{array}
}}
\newcommand{\mystackl}[1]{{
\begin{array}{l}#1\end{array}
}}

\def\vspbot{{\rule[-1ex]{0pt}{0ex}}}
\def\vsptop{{\rule{0pt}{2.5ex}}}
\def\vspboth{{\vsptop\vspbot}}


\begin{table}[h]
\centering
$\begin{array}{|c||c|c|l|}
\hline
\vspboth
  \text{\textbf{case}} & w & L_w & \hfill\lwbar=L_w/\<\w\>\hfill \\
\hhline{|=#=|=|=|}
  \text{\textit{2-split}} & x^p+y^q
  & \mystack{\vsptop\ZZ\la\x,\y\ra/\la p\x = q\y\ra \\ \simeq \ZZ \oplus \ZZ/\text{lcd}(p,q)\ZZ}
  & \mystackl{\ZZ\la\xx,\yy\ra/\la p\xx=q\yy=0\ra \\ \simeq \ZZ/p\ZZ \oplus \ZZ/q\ZZ,\\ \text{basis: }(\xx,\yy)} \\
\hline
  \text{\textit{2-chain}} & x^p+xy^q
  & \mystack{\vsptop\ZZ\la\x,\y\ra/\la p\x = \x+q\y\ra \\ \simeq \ZZ \oplus \ZZ/\text{lcd}(p-1,q)\ZZ}
  & \mystackl{{\rule[-2ex]{0pt}{5.5ex}}\ZZ\la\xx,\yy\ra \left/\la\footnotesize{\mystack{ p\xx=\\\xx+q\yy=0}}\ra\right. \\ \simeq \ZZ/pq\ZZ,  \text{ generator: } \yy} \\
\hline
  \text{\textit{2-loop}} & yx^p+xy^q
  & \footnotesize{\mystack{\vsptop\ZZ\la\x,\y\ra/\la \y+p\x = \x+q\y\ra \\ \vsptop\simeq \ZZ \oplus \ZZ/\text{lcd}(p-1,q-1)\ZZ}}
  & \mystackl{{\rule[-2ex]{0pt}{5.5ex}}\ZZ\la\xx,\yy\ra\left/\la \footnotesize{\mystack{\yy+p\xx=\\ \xx+q\yy=0}}\ra\right. \\ \simeq \ZZ/(pq-1)\ZZ, \\\text{generators: $\xx$, $\yy$}} \\
\hline
\end{array}$
\caption{Grading groups on $\kk[x,y]/(w)$}
\label{tbl:grading-groups:n=2}
\end{table}
\endgroup

We will use the following fact about the grading groups $\lwbar$.
\begin{lemma}
\label{lem:lwbar-elements:n=2}
  In each of the three cases for $w$ described above, the sequence $(-i\xx-j\yy \mid 0\le i\le p-1,\, 0\le j\le q-1)$ contains all of the elements of the group $\lwbar$.
\end{lemma}
\begin{proof}
  The proof for all 3 types of $w$ is the same, so let us prove the statement only for the loop case $w=yx^p+xy^q$. In this case, $-i\xx-j\yy=-i(-q\yy)-j\yy=(iq-j)\yy$. For $0\le i\le p-1$ and $0\le j\le q-1$, the coefficient $iq-j$ takes $pq$ consecutive integer values from the interval $[-(q-1),(p-1)q]$. Thus the elements $(iq-j)\yy$ span the whole $\lwbar$, since $\yy$ is a generator of $\lwbar$ of order $pq-1$.
\end{proof}
\begin{remark}
  Note that in the loop case two of the mentioned elements are equal, namely, $-(p-1)\xx = -(q-1)\yy$, while in the split and chain cases, these are all different elements of $\lwbar$.
\end{remark}

\subsubsection{The objects}

Let us describe the objects which we will use for the exceptional collections in $\dsgof{L_w}{\kk[x,y]/(w)}$. In each of the three cases we will use the shifts of $\kk$, similarly to the case of Brieskorn-Pham polynomials in Proposition~\ref{prop:futaki-ueda-split}. In the chain and loop cases, we will additionally use the module $M_y = \kk[y]$ defined by setting the action of $x$ to $0$. (It is a well-defined $\kk[x,y]/(w)$-module, since 
$w$ is a multiple of $x$ in both chain and loop cases.) In the loop case, we will also use an analogous module $M_x = \kk[x]$ with a trivial action of $y$ and a module $M_{xy}=\kk[x,y]/(xy)$ with the actions of $x$ and $y$ inherited from $\kk[x,y]$.

\begin{remark}
  Geometrically, the module $\kk$ corresponds to the structure sheaf of the origin. The modules $M_x$ and $M_y$ correspond to the structure sheaves of the $x$- and $y$-axes. The module $M_{xy}$ corresponds to the structure sheaf of the union of $x$- and $y$-axes.
\end{remark}

\subsubsection{The morphisms}
\label{sec:app:explicit-morphisms:n=2}

We can compute the morphisms between the objects $\kk$, $M_y = \kk[y]$, $M_x=\kk[x]$ and $M_{xy}=\kk[x,y]/(xy)$ by applying Recipe~\ref{recipe}. For that we need to build some projective resolutions for these modules. Since all our modules have the form $M = \kk[x,y]/I$, where $I$ is an ideal in $\kk[x,y]$ generated by a regular sequence and $w\in I$, we can use Koszul resolutions as described in Section~2.3 of \cite{dyckerhoff} in order to build the projective resolutions as follows.

For the module $\kk = \kk[x,y]/(x,y)$, the potential $w$ decomposes as the sum $w = xw_x +yw_y$, where $w_x = x^{p-1}$ in the split and chain cases, $w_x = yx^{p-1}$ in the loop case, $w_y = y^{q-1}$ in the split case and $w_y = xy^{q-1}$ in the chain and loop cases. Then the projective resolution of $\kk$ is as follows:
\vspace{-1ex}
\begin{multline*}
    \bull P_\kk = \Biggl(
    \dots \xrightarrow{\mat{w_x & -y \\ w_y & x}}
    {\scriptstyle A(-\w-\x)\, \oplus\, A(-\w-\y) }
    \xrightarrow{\mat{x & y \\ -w_y & w_x}}
    {\scriptstyle A(-\w)\, \oplus \, A(-\x-\y)}
    \\
    \xrightarrow{\mat{w_x & -y \\ w_y & x}} A(-\x)\oplus A(-\y) \xrightarrow{\mat{x & y}} A \Biggr) \xrightarrow{1} 
    \kk, \qquad
\end{multline*}
where we use the notation $A = \kk[x,y]/(w)$ here and henceforward, for the sake of brevity.

For the module $M_y = \kk[x,y]/(x)$, we write $w = xw_x$ where $w_x = x^{p-1}+y^q$ in the chain case and $w_x = yx^{p-1}+y^q$ in the loop case. This gives the following resolution:
$$ \bull P_{M_y} = \left( \dots \to A(-2\w) \xrightarrow{w_x} A(-\w-\x) \xrightarrow{x} A(-\w) \xrightarrow{w_x} A(-\x) \xrightarrow{x} A \right) \xrightarrow{1} M_y. $$
The resolution for the module $M_x$ in the loop case can be obtained from that for $M_y$ by switching the variables $x\leftrightarrow y$ and $p\leftrightarrow q$.

For the resolution of $M_{xy}=\kk[x,y]/(xy)$ in the loop case, we use $w_{xy} = x^{p-1} + y^{q-1}$, so that $w = xy w_{xy}$:
$$ \bull P_{M_{xy}} = \left( \dots \to A(-\w-\x-\y) \xrightarrow{xy} A(-\w) \xrightarrow{w_{xy}} A(-\x-\y) \xrightarrow{xy} A \right) \xrightarrow{1} M_{xy}. $$

The rest of the computations from Recipe~\ref{recipe} are straightforward. The resulting morphisms in $\dsgof{}{\kk[x,y]/(w)}$ between the given objects and their shifts are summarized in Table~\ref{table:all-morphisms:n=2} along with the $L_w$-grading information.

\newcommand{\twoshialign}[3]{\begin{array}{#1}#2\textstyle,\\[0.5ex]#3\end{array}}
\newcommand{\twoshifts}[2]{\twoshialign{c}{#1}{#2}}

\begin{table}[h]
\hfill
\scalebox{0.75}{
$
\begin{array}{|c||c|c|c|c|}
\hline
  \twoshialign{l}{\Hom(\downarrow,\rightarrow)}
  {\Hom(\downarrow,\rightarrow[1])}
  & \kk & M_y & M_x & M_{xy} \\
\hhline{|=#=|=|=|=|}
  \kk &
    \twoshifts{\kk\oplus\kk(\x+\y-\w)}{\kk(\x)\oplus\kk(\y)} &
    \twoshifts{\kk(\x+\y-\w)}{\kk(\y)} &
    \twoshifts{\kk(\x+\y-\w)}{\kk(\x)} &
    \twoshifts{\kk(\x+\y-\w)}{\kk} \\
\hline
  M_y &
    \twoshifts{\kk}{\kk(\x)} &
    \twoshifts{\kk[y]/(y^q)}{0} &
    \twoshifts{0}{\kk(\x)} &
    \twoshifts{\scriptstyle\kk[y]/(y^{q-1}) (-\y)}{0} \\
\hline
  M_x &
    \twoshifts{\kk}{\kk(\y)} &
    \twoshifts{0}{\kk(\y)} &
    \twoshifts{\kk[x]/(x^p)}{0} &
    \twoshifts{\scriptstyle\kk[x]/(x^{p-1}) (-\x)}{0} \\
\hline
  M_{xy} &
    \twoshifts{\kk}{\kk(\x+\y)} &
    \twoshifts{\kk[y]/(y^{q-1})}{0} &
    \twoshifts{\kk[x]/(x^{p-1})}{0} &
    \twoshifts{\stackrel{\scriptstyle\kk[x]/(x^{p-1})\oplus{}}{ \scriptstyle \kk[y]/(y^{q-1}) (-\y)}}{0} \\
\hline
\end{array}
$}\hfill
\caption{Morphisms between the objects in $\dsgof{}{\kk[x,y]/(w)}$ for different $w$}
\label{table:all-morphisms:n=2}
\end{table}

\begin{remark}
\label{rem:app-morphisms:table-remark1}
  In Table~\ref{table:all-morphisms:n=2}, we have specified only the morphisms $\Hom(M,N)$ and $\Hom(M,N[1])$ for different $M$ and $N$. The morphisms between all other shifts of $M$ and $N$ can be obtained via quasi-periodicity:
$$ \Hom(M,N[2j]) = \Hom(M,N(j\w)) = \Hom(M,N)(j\w), $$
$$ \Hom(M,N[2j+1]) = \Hom(M,N[1])(j\w). $$
\end{remark}

\begin{remark}
\label{rem:app-morphisms:table-remark2}
  In Table~\ref{table:all-morphisms:n=2}, an expression like $\kk[x]/(x^{p-1})$ is just a short-hand notation for the $L_w$-graded vector space $\bigoplus_{0\le i < p-1}\kk(-i\x)$, where the monomial~$x^i$ corresponds to the summand $\kk(-i\x)$ situated in degree $i\x$.
\end{remark}

\begin{remark}
\label{rem:app-morphisms:table-remark3}
  The reader may note an asymmetricity in the expression for $\Hom(M_{xy},M_{xy})$. More symmetrical but longer way to write it down would be as $\kk\oplus\kk[x]/(x^{p-2})(-\x)\oplus\kk[y]/(y^{q-2})(-\y)\oplus\kk(\x+\y-\w)$, where $\x+\y-\w = -(p-1)\x = -(q-1)\y$.
\end{remark}

\subsubsection{Exceptional collections}

\begingroup
\def\myA{{\kk[x,y]/(w)}}
\def\mydsg{{\dsgof{L_w}{\kk[x,y]/(w)}}}

Now we are ready to form exceptional collections in $\dsgof{L_w}{\kk[x,y]/(w)}$ out of our modules and their shifts for each case of $w$.


\begin{prop}\label{prop:main:n=2}
  Let the ring be $R=\kk[x,y]$, the potential be $w=x^p+y^q$, $w=x^p+xy^q$ or $w=yx^p+xy^q$ for $p,q\ge 2$, the grading group be $L_w$ and the graded $\myA$-modules $\kk$, $M_x$, $M_y$ and $M_{xy}$ be as described above. Then the corresponding graded category of singularities $\mydsg$ admits a full strongly exceptional collection as follows:
  \begin{enumerate}[label=(\alph*)]
      \item \textbf{Split case {\normalfont(see Prop.~\ref{prop:futaki-ueda-split} for $n=2$)}:} If $w = x^p+y^q$, the collection consists of $(p-1)(q-1)$ shifts of $\kk$ of the form $\kk(-i\x-j\y)[i+j]$ for $0\le i \le p-2$, $0\le j \le q-2$ arranged as on Figure~\ref{fig:2split}.
      \item \textbf{Chain case:} If $w = x^p+xy^q$, the collection consists of $pq-q+1$ objects, namely
      \begin{itemize}[itemsep=0ex, topsep=0ex, parsep=0ex]
          \item $(p-1)(q-1)$ shifts of $\kk$ as in (a),
          \item and $p$ shifts of $M_y$ of the form $M_y(-i\x)[i]$ for $-1 \le i \le p-2$,
      \end{itemize} arranged as on Figure~\ref{fig:2chain}.
      \item \textbf{Loop case:} If $w = yx^p+xy^q$, the collection consists of $pq$ objects, namely
      \begin{itemize}[itemsep=0ex, topsep=0ex, parsep=0ex]
          \item $(p-1)(q-1)$ shifts of $\kk$ as in (a),
          \item $p-1$ shifts of $M_y$ of the form as in (b) for $0\le i \le p-2$,
          \item $q-1$ shifts of $M_x$ of the form $M_x(-j\y)[j]$ for $0\le j\le q-2$,
          \item and one extra-object $M_{xy}(\x+\y)[-1]$,
      \end{itemize} arranged as on Figure~\ref{fig:2loop}.
  \end{enumerate}
  \textbf{Legend:} For the sake of readability, we did not apply shifts $[\bullet]$ to the objects on Figures~\ref{fig:2split}--\ref{fig:2loop}. Instead, we used the notation $X \xrightarrow{[k]} Y$ to denote a one-dimensional space of morphisms from $X$ to $Y[k]$, or more generally, from $X[i]$ to $Y[i+k]$ for $i\in\ZZ$.
\end{prop}
\endgroup

\begingroup

\def\My{M_y}
\def\Mx{M_x}
\def\Mxy{M_{xy}}

\def\myextraspace{0ex}

\begin{figure}[p]
\centering
\begin{tikzcd}[column sep = small]
\kk(-(q-2)\y) \ar[r, "{[1]}"] & \kk(-\x-(q-2)\y) \ar[r, "{[1]}"] & ... \ar[r, "{[1]}"] & \kk(-(p-2)\x - (q-2)\y) \\
... \ar[u, "{[1]}"] \ar[ru, "{[2]}"] & ... \ar[u, "{[1]}"] & ... \ar[ru, "{[2]}"] & ... \ar[u, "{[1]}"] \\ 
\kk(-\y) \ar[u, "{[1]}"] \ar[ru, "{[2]}"] \ar[r, "{[1]}"] & \kk(-\x - \y) \ar[u, "{[1]}"] \ar[ru, "{[2]}"] \ar[r, "{[1]}"] & ... \ar[ru, "{[2]}"] \ar[r, "{[1]}" pos=0.7] & \kk(-(p-2)\x - \y) \ar[u, "{[1]}"] \\
\kk \ar[u, "{[1]}"] \ar[ru, "{[2]}"] \ar[r, "{[1]}"] & \kk(-\x) \ar[u, "{[1]}"] \ar[ru, "{[2]}"] \ar[r, "{[1]}"] & ... \ar[ru, "{[2]}"] \ar[r, "{[1]}"] & \kk(-(p-2)\x ) \ar[u, "{[1]}"]
\end{tikzcd}
\vspace{\myextraspace}
\caption{\textbf{(2-split)} Exceptional collection in $\dsgof{L_w}{\kk[x,y]/(w)}$ for $w = x^p + y^q$}
\label{fig:2split}
\end{figure}

\begin{figure}[p]
\centering
\begin{tikzcd}[column sep = small]
& \kk(-(q-2)\y) \ar[r, "{[1]}"] & \kk(-\x-(q-2)\y) \ar[r, "{[1]}"] & ... \ar[r, "{[1]}"] & \kk(-(p-2)\x - (q-2)\y) \\
& ... \ar[u, "{[1]}"] \ar[ru, "{[2]}"] & ... \ar[u, "{[1]}"] & ... \ar[ru, "{[2]}"] & ... \ar[u, "{[1]}"] \\ 
& \kk(-\y) \ar[u, "{[1]}"] \ar[ru, "{[2]}"] \ar[r, "{[1]}"] & \kk(-\x - \y) \ar[u, "{[1]}"] \ar[ru, "{[2]}"] \ar[r, "{[1]}"] & ... \ar[ru, "{[2]}"] \ar[r, "{[1]}" pos=0.7] & \kk(-(p-2)\x - \y) \ar[u, "{[1]}"] \\
& \kk \ar[u, "{[1]}"] \ar[ru, "{[2]}"] \ar[r, "{[1]}"] & \kk(-\x) \ar[u, "{[1]}"] \ar[ru, "{[2]}"] \ar[r, "{[1]}"] & ... \ar[ru, "{[2]}"] \ar[r, "{[1]}"] & \kk(-(p-2)\x ) \ar[u, "{[1]}"]\\
\My(\x) \ar[ru, "{[1]}"] & \My \ar[ru, "{[1]}"] \ar[u, "{[0]}"] & \My(-\x) \ar[ru, "{[1]}"] \ar[u, "{[0]}"] & ... \ar[ru, "{[1]}"] & \My(-(p-2) \x) \ar[u, "{[0]}"]
\end{tikzcd}
\vspace{\myextraspace}
\caption{\textbf{(2-chain)} Exceptional collection in $\dsgof{L_w}{\kk[x,y]/(w)}$ for $w = x^p + xy^q$}
\label{fig:2chain}
\end{figure}

\begin{figure}[p]
\centering
\begin{adjustbox}{scale=0.9}
\begin{tikzcd}[column sep = small]
\Mx(-(q-2)\y) \ar[r, "{[0]}"] & \kk(-(q-2)\y) \ar[r, "{[1]}"] & \kk(-\x-(q-2)\y) \ar[r, "{[1]}"] & ... \ar[r, "{[1]}"] & \kk(-(p-2)\x - (q-2)\y) \\
... \ar[ru, "{[1]}"] & ... \ar[u, "{[1]}"] \ar[ru, "{[2]}"] & ... \ar[u, "{[1]}"] & ... \ar[ru, "{[2]}"] & ... \ar[u, "{[1]}"] \\
\Mx(-\y) \ar[ru, "{[1]}"] \ar[r, "{[0]}"] & \kk(-\y) \ar[u, "{[1]}"] \ar[ru, "{[2]}"] \ar[r, "{[1]}"] & \kk(-\x - \y) \ar[u, "{[1]}"] \ar[ru, "{[2]}"] \ar[r, "{[1]}"] & ... \ar[ru, "{[2]}"] \ar[r, "{[1]}" pos=0.7] & \kk(-(p-2)\x - \y) \ar[u, "{[1]}"] \\
\Mx \ar[ru, "{[1]}"] \ar[r, "{[0]}"] & \kk \ar[u, "{[1]}"] \ar[ru, "{[2]}"] \ar[r, "{[1]}"] & \kk(-\x) \ar[u, "{[1]}"] \ar[ru, "{[2]}"] \ar[r, "{[1]}"] & ... \ar[ru, "{[2]}"] \ar[r, "{[1]}"] & \kk(-(p-2)\x ) \ar[u, "{[1]}"] \\
\Mxy(\x+\y) \ar[ru, "{[1]}"] & \My \ar[ru, "{[1]}"] \ar[u, "{[0]}"] & \My(-\x) \ar[ru, "{[1]}"] \ar[u, "{[0]}"] & ... \ar[ru, "{[1]}"] & \My(-(p-2) \x) \ar[u, "{[0]}"]
\end{tikzcd}
\end{adjustbox}
\vspace{\myextraspace}
\caption{\textbf{(2-loop)} Exceptional collection in $\dsgof{L_w}{\kk[x,y]/(w)}$ for $w = yx^p + xy^q$}
\label{fig:2loop}
\end{figure}
\endgroup

\begin{remark}
  By definition, an exceptional collection is an ordered sequence of objects. When we say that the exceptional collection is arranged as in Figures~\ref{fig:2split}--\ref{fig:2loop}, we mean that the objects can be rearranged into a 1-dimensional sequence in which the morphisms go only in one direction. For example, in the given situation, one may use lexicographic ordering, since our arrows go only in the rightward and upward directions.
\end{remark}

\subsubsection{Verifying exceptionality}

In order to prove Proposition~\ref{prop:main:n=2}, we firstly need to check that our collections are indeed strong and exceptional and have morphisms exactly as described on the corresponding figures. We do this in the following lemma by using information about the morphisms from Table~\ref{table:all-morphisms:n=2} and about the group $\lwbar$ from Table~\ref{tbl:grading-groups:n=2}.

\begingroup
\def\myA{{\kk[x,y]/(w)}}
\def\mydsg{{\dsgof{L_w}{\kk[x,y]/(w)}}}
\def\mydsgungr{{\dsgof{}{\kk[x,y]/(w)}}}
\begin{lemma}
\label{lemma:main:n=2:checking-collections}
  The collections described in Proposition~\ref{prop:main:n=2} are strongly exceptional and the morphisms are as depicted on Figures~\ref{fig:2split}--\ref{fig:2loop}.
\end{lemma}
\begin{proof}
  The proof consists of considering many similar cases. Let us, for example, consider the chain type polynomial $w=x^p+xy^q$, and morphism spaces from the given shifts of $\kk$ to the given shifts of $M_y$. In this case, we should prove that all such morphisms are zero as depicted on the Figure~\ref{fig:2chain}.

  From Table~\ref{table:all-morphisms:n=2}, we see that the morphisms from $\kk$ to the shifts of $M_y$ in the ungraded category $\mydsgungr$ are as follows:
\begin{equation*}
  \begin{aligned}
    &\Hom_\mydsgungr(\kk,M_y[2k]) &&= \kk(\x+\y-\w+k\w),\\
    &\Hom_\mydsgungr(\kk,M_y[2k+1]) &&= \kk(\y+k\w),
\end{aligned}
\end{equation*}
for all $k\in\ZZ$. Thus the morphisms in the corresponding graded category $\mydsg$ are as follows (see Step~4 of Recipe~\ref{recipe}):
\begin{equation*}
  \begin{aligned}
    &\Hom_\mydsg(\kk,M_y(l)[2k]) &&= \begin{cases}
    \kk,&\text{ if } l+\x+\y-\w+k\w=0,\\
    0 &\text{ otherwise},
    \end{cases}
    \\
    &\Hom_\mydsg(\kk,M_y(l)[2k+1]) &&= \begin{cases}
    \kk,&\text{ if } l+\y+k\w=0,\\
    0 &\text{ otherwise}.
    \end{cases}
\end{aligned}
\end{equation*}

We now see that the non-trivial morphisms from $\kk$ to $M_y(l)[\bullet]$ may exist only when $l+\x+\y\in\<\w\>$ or $l+\y\in\<\w\>$, that is when \\[1ex]
 $(*)$ \hfill $\bar l+\xx+\yy = 0$ or $\bar l+\yy = 0$ in $\lwbar$, \hfill {} \\[1ex]
where $\bar l$ is the class of $l$ in $\lwbar$. In our collection, however, we are only using the shifts $\kk(-i_1\x-j\y)$ for $0\le i_1 \le p-2$, $0\le j\le q-2$, and $M_y(-i_2\x)$ for $-1 \le i_2 \le p-2$. So in order to show that there are no morphisms from the shifts of $\kk$ to the shifts of $M_y$ in our collection,
we just need to check that the condition $(*)$ on $\bar l$ never holds for $l = -i_2\x -(-i_1\x-j\y) = (i_1-i_2)\x+j\y$.

To show the latter, we use information about the group $\lwbar$ from Table~\ref{tbl:grading-groups:n=2}.
Since $\xx = -q\yy$, we have $\bar l = (i_1-i_2)\xx+j\yy = (-q(i_1-i_2)+j)\yy$. Similarly, $\bar l + \xx+\yy = (-q(i_1-i_2+1)+j+1)\yy$ and $\bar l+\yy = (-q(i_1-i_2)+j+1)\yy$. Since $\yy$ has order $pq$ in $\lwbar$, the elements $\bar l+\xx+\yy$ and $\bar l+\yy$ may be equal to zero only when $pq|(-q(i_1-i_2+1)+j+1)$ or $pq|(-q(i_1-i_2)+j+1)$ respectively. In particular, this would imply that $q|(j+1)$ in either case. However, this cannot happen, since $1\le (j+1)\le q-1$ by assumption. Thus the non-zero morphisms cannot occur between the given sets of shifts of $\kk$ and $M_y$. This concludes the proof in the this case.

To complete the proof, we need to consider all types of $w$, all possible pairs of objects among $\kk$, $M_y$, $M_x$ and $M_{xy}$, and the morphisms from the shifts of the first object to the shifts of the second one. This gives us in total 1 case when $w$ has split type, 4~cases when $w$ has chain type (since we have 2~kinds of objects in the collection, $\kk$ and $M_y$), and 16~cases when $w$ has loop type (since $16=4^2$ for 4 kinds of objects). The argument however stays the same in each of the cases, so we omit it. The only difference is that if there are some arrows present between the given groups of objects on Figures~\ref{fig:2split}--\ref{fig:2loop}, then instead of showing that the condition analogous to $(*)$ never holds, we just show that it holds exactly in the cases depicted on the corresponding figures.

For example, in the case of the morphisms from the shifts of $\kk$ to themselves (for all types of~$w$), we see from Table~\ref{table:all-morphisms:n=2} that the morphisms from $\kk$ to $\kk(l)[\bullet]$ are non-trivial only when
 \\[1ex]
 $(**)$ \phantom{} \hfill one of $\bar l$, $\bar l+\xx$, $\bar l +\yy$ or $\bar l+\xx+\yy$ is zero in $\lwbar$. \hfill {} \\[1ex]
Then, comparing the shifts $\kk(-i_1\x-j_1\y)$ and $\kk(-i_2\x-j_2\y)$ for the given bounds on $i_1$, $i_2$, $j_1$, $j_2$, setting $l=(-i_2\x-j_2\y)-(-i_1\x-j_1\y)$ and using information from the last column of Table~\ref{tbl:grading-groups:n=2}, we see that $(**)$ may hold only when $i_2=i_1+\eps_x$ and $j_2=j_1+\eps_y$, where $(\eps_x, \eps_y)\in\{0, 1\}^2$. The four cases for $(\eps_x,\eps_y)$ will give us respectively the horizontal, vertical and diagonal arrows on Figures~\ref{fig:2split}--\ref{fig:2loop} (with the prescribed shifts $[\bullet]$), as well as the condition that $\kk$ is exceptional (for $\eps_x=\eps_y=0$).
\end{proof}

\subsubsection{Generating shifts of \texorpdfstring{$\kk$}{k}}

The second part of the proof of Proposition~\ref{prop:main:n=2} consists of generating a bunch of shifts of the object $\kk$ from the objects of the given collection. To do this we use an idea described in \cite[Section 4]{ueda} and \cite[Lemma~4.2]{futaki-ueda}. Namely, we consider simple short exact sequences in the category of $\kk[x,y]/(w)$-modules and use them to generate new objects from the given ones. Since we need to repeat similar steps multiple times, we also formulate them in terms of the so-called generation rules. We then apply these rules to the objects of the given collection until we generate all the desired shifts of $\kk$. Also, our Lemma~\ref{lemma-perfect} comes in handy here since it often allows us to simplify generation rules by removing zero objects. This whole procedure is detalized in the following lemma.

\begin{lemma}
\label{lemma:main:n=2:checking-fullness}
  The collections described in Proposition~\ref{prop:main:n=2} generate the modules $\kk(-i\x-j\y)$ for all $0\le i\le p-1$ and $0\le j\le q-1$.
\end{lemma}
\begin{proof}
\textit{Step 1 {(short exact sequences)}.} Consider the graded $\myA$-modules $M_{x,i} = \kk[x]/(x^i)$ and $M_{y,j}=\kk[y]/(y^j)$, where $1 \le i\le p$ and $1 \le j\le q$, with the obvious actions of $x$ and $y$. In particular, we have $M_{x,1} = M_{y,1} = \kk$. Similarly to \cite{ueda} and \cite{futaki-ueda}, we consider the following short exact sequences of modules:
\vspace{-2ex}
\begin{gather}
\label{eq:exc-coll:n=2-ses-M_xi}
    0\to \kk(-i\x) \xrightarrow{x^i} M_{x,i+1} \xrightarrow{1} M_{x,i} \to 0,
    \\
\label{eq:exc-coll:n=2-ses-M_yj}
    0\to \kk(-j\y) \xrightarrow{y^j}M_{y,j+1} \xrightarrow{1} M_{y,j} \to 0,
\end{gather}
where $1\le i \le p-1$ and $1\le j \le q-1$.

In the split and chain cases, the module $M_{x,p} = \kk[x]/(x^p)$ is perfect by Lemma~\ref{lemma-perfect} (put $K = \kk[x]$ considered as a $\kk[x,y]$-module with a trivial action of $y$), hence it is a zero object in $\mydsg$. In the loop case, however, the object $M_x$ and its shift give us the following short exact sequence for $M_{x,p}$:
\vspace{-1ex}
\begin{equation}
\label{eq:exc-coll:n=2-ses-M_xp}
  0\to M_x(-p\x) \xrightarrow{x^p} M_x \xrightarrow{1} M_{x,p} \to 0. 
\end{equation}
Similarly, $M_{y,q}$ is a zero object in $\mydsg$ in the split case, and is embedded in the following short exact sequence in the chain and loop cases:
\vspace{-1ex}
\begin{equation}
\label{eq:exc-coll:n=2-ses-M_yq}
  0\to M_y(-q\y) \xrightarrow{y^q} M_y \xrightarrow{1} M_{y,q} \to 0.
\end{equation}

In the loop case, we will also use the following short exact sequences:
\vspace{-1ex}
\begin{gather}
\label{eq:exc-coll:n=2-ses-M_x-x}
    0\to M_x(-\x) \xrightarrow{x} M_x \xrightarrow{1} \kk \to 0, \\
\label{eq:exc-coll:n=2-ses-M_y-y}
    0\to M_y(-\y) \xrightarrow{y} M_y \xrightarrow{1} \kk \to 0, \\
\label{eq:exc-coll:n=2-ses-M_xy_x}
    0\to M_x(-\x) \xrightarrow{x} M_{xy} \xrightarrow{1} M_y \to 0, \\
\label{eq:exc-coll:n=2-ses-M_xy_y}
    0\to M_y(-\y) \xrightarrow{y} M_{xy} \xrightarrow{1} M_x \to 0,
\end{gather}

Each of the above short exact sequences gives rise to a distinguished triangle in $\mydsg$ which can be used for generating each object of this triangle from the other two.

\vspace{0.5ex}
\noindent\textit{Step 2 (generation rules).} Now, let us use the short exact sequences from Step~1 in order to build various \textit{generation rules}. For example, let us use sequences \eqref{eq:exc-coll:n=2-ses-M_xi} for all $1\le i\le p-1$ to show that the objects $\kk(-i\x)$ for $0\le i \le p-2$ together with $M_{x,p}$ generate $\kk(-(p-1)\x)$. Indeed, by induction on $i$ and \eqref{eq:exc-coll:n=2-ses-M_xi}, we can generate $M_{x,p-1}$ from the above shifts of $\kk$. Then by plugging in $i=p-1$ into \eqref{eq:exc-coll:n=2-ses-M_xi}, we see that $\kk(-(p-1)\x)$ is generated by $M_{x,p-1}$ and $M_{x,p}$ which concludes this argument. We can also shift the grading in the argument to show that the objects $\kk(-i\x+l)$ for $0\le i\le p-2$ together with the object $M_{x,p}(l)$ generate the object $\kk(-(p-1)\x+l)$ for any $l\in L_w$. We can express this situation as follows:
\begin{equation}
\label{eq:exc-coll:n=2-gen-rule-x}
\kk(-(p-1)\x+l) \in \Bigl\langle \kk(l), \dots, \kk(-(p-2)\x+l), M_{x,p}(l) \Bigr\rangle.
\end{equation}
Similarly, the objects $\kk(-j\y+l)$ for $0\le j\le q-2$ and $M_{y,q}(l)$ generate $\kk(-(q-1)\x+l)$ for any $l\in L_w$:
\begin{equation}
\label{eq:exc-coll:n=2-gen-rule-y}
\kk(-(q-1)\y+l) \in \Bigl\langle \kk(l), \dots, \kk(-(q-2)\y+l), M_{y,q}(l) \Bigr\rangle.
\end{equation}
These will be our first generation rules.

We have noted previously that depending on the type of $w$, either $M_{x,p}$ is a perfect module, hence a zero object in $\mydsg$, or it can be generated from $M_x$ and $M_x(-p\x)$ by \eqref{eq:exc-coll:n=2-ses-M_xp}. Also, in the latter case $w$ is loop, and we have $\y+p\x=\w$, so $M_x(-p\x)= M_x(-w+\y) \simeq M_x(\y)[-2]$. Hence $M_{x,p}$ can be generated from $M_x$ and $M_x(\y)$.
Similarly, $M_{y,q}$ is either zero (when $w$ is split) or can be generated from $M_y$ and $M_y(\x)$ (when $w$ is chain or loop and $\x+q\y=\w$). Summarizing these arguments, we see that in the ``zero'' cases for $M_{x,p}$ and $M_{y,q}$, the rules \eqref{eq:exc-coll:n=2-gen-rule-x} and \eqref{eq:exc-coll:n=2-gen-rule-y} simplify to the following:
\vspace{-1ex}
\begin{gather}
\label{eq:exc-coll:n=2-gen-rule-k-simple-x}
  \kk(-(p-1)\x+l) \in \Bigl\langle \kk(l), \dots, \kk(-(p-2)\x+l) \Bigr\rangle,\\
\label{eq:exc-coll:n=2-gen-rule-k-simple-y}
  \kk(-(q-1)\y+l) \in \Bigl\langle \kk(l), \dots, \kk(-(q-2)\y+l) \Bigr\rangle,
\end{gather}
while in the ``non-zero'' cases, we get respectively:
\vspace{-1ex}
\begin{gather}
\label{eq:exc-coll:n=2-gen-rule:k(-ix)+Mx+Mx(y)=>k+}
  \kk(-(p-1)\x+l) \in \Bigl\langle \kk(l), \dots, \kk(-(p-2)\x+l), M_x(l), M_x(\y+l) \Bigr\rangle,\\
\label{eq:exc-coll:n=2-gen-rule:k(-jy)+My+My(x)=>k+}
  \kk(-(q-1)\y+l) \in \Bigl\langle \kk(l), \dots, \kk(-(q-2)\y+l), M_y(l), M_y(\x+l) \Bigr\rangle,
\end{gather}
for any $l\in L_w$.

In the loop case, we will need one more generation rule. We firstly note that by \eqref{eq:exc-coll:n=2-ses-M_x-x}, $M_x(-\x)$ is generated by $M_x$ and $\kk$. Shifting this statement by $-i\x$ for $0\le i\le p-2$ and applying induction, we see that $M_x(-(p-1)\x)$ is generated by $M_x$ and $\{\kk(-i\x)\mid 0\le i\le p-2\}$. Then we notice that $M_x(-(p-1)\x)=M_x(-(q-1)\y)=M_x(-\w+\x+\y)\simeq M_x(\x+\y)[-2]$, since $\w=\y+p\x=\x+q\y$. Summarizing, we get the following:
\vspace{-1ex}
\begin{equation}
\label{eq:exc-coll:n=2-gen-rule-loop:Mx+k(-ix)=>Mx(x+y)}
 M_x(-(q-1)\y),\; M_x(\x+\y) \in \Bigl\langle \kk, \dots, \kk(-(p-2)\x), M_x \Bigr\rangle.
\end{equation}

\noindent\textit{Step 3 (step-by-step generation).} Using the generation rules obtained in Step~2, we can generate the desired shifts of $\kk$ for each type of $w$ from the given collection as follows.

In the split case, let us quickly recap the argument from \cite{futaki-ueda}. We already have shifts $\kk(-i\x-j\y)$ for $0\le i\le p-2$ and $0\le j\le q-2$ in our collection. Thus we can apply \eqref{eq:exc-coll:n=2-gen-rule-k-simple-x} for $l=-j\y$ to generate $\kk(-(p-1)\x-j\y)$ from the latter shifts for all $0\le j\le q-2$. Then, we apply \eqref{eq:exc-coll:n=2-gen-rule-k-simple-y} for $l=-i\x$ to generate $\kk(-i\x-(q-1)\y)$ for all $0\le i \le p-1$. This gives us the shifts $\kk(-i\x-j\y)$ for all $0\le i\le p-1$ and $0\le j\le q-1$ as was needed.

Similarly, in the chain case, we firstly use \eqref{eq:exc-coll:n=2-gen-rule:k(-jy)+My+My(x)=>k+} for $l=-i\x$ and the given shifts of $\kk$ and $M_y$ to generate $\kk(-i\x-(q-1)\y)$ for all $0\le i\le p-2$. Then, we use \eqref{eq:exc-coll:n=2-gen-rule-k-simple-x} for $l=-j\y$, $0\le j\le q-1$, to generate the rest of the shifts of $\kk$.

In the loop case, we can proceed similarly to the split and chain cases. However, for that we also need to preliminarily generate few more helper objects. Namely, we firstly use \eqref{eq:exc-coll:n=2-gen-rule-loop:Mx+k(-ix)=>Mx(x+y)} to generate $M_x(\x+\y)$ and $M_x(-(q-1)\y)$. Shifting \eqref{eq:exc-coll:n=2-ses-M_xy_y} by $\x+\y$, we see that $M_y(\x)$ is generated by $M_{xy}(\x+\y)$ and $M_x(\x+\y)$. Thus $M_y(\x)$ is also generated by our collection. Repeating this argument while substituting $x\leftrightarrow y$, $p\leftrightarrow q$, we get that $M_x(\y)$ and $M_y(-(p-1)\x)$ are also generated by our collection. Now, since we have generated $M_y(\x)$, we can repeat the first step of the argument for the chain case in order to generate $\kk(-i\x-(q-1)\y)$ for all $0\le i\le p-2$ by using \eqref{eq:exc-coll:n=2-gen-rule:k(-jy)+My+My(x)=>k+}. After that, we can similarly use \eqref{eq:exc-coll:n=2-gen-rule:k(-ix)+Mx+Mx(y)=>k+} and all $M_x(-j\y)$ for $-1\le j\le q-1$ in order to generate $\kk(-(p-1)\x-j\y)$ for all $0\le j\le q-1$.
\end{proof}
\endgroup

\subsubsection{Proof of Proposition~\ref{prop:main:n=2} and concluding remarks}

After making all necessary checks and computations, we can conclude the proof of our proposition as follows.

\begin{proof}[Proof of Proposition~\ref{prop:main:n=2}]
In Lemma~\ref{lemma:main:n=2:checking-collections}, we checked that the collections from Figures~\ref{fig:2split}--\ref{fig:2loop} are indeed exceptional and strong. In Lemma~\ref{lemma:main:n=2:checking-fullness}, we showed that these collections generate the shifts $\kk(-i\x-j\y)$ for all $0\le i\le p-1$ and $0\le j\le q-1$. Finally, by Lemma~\ref{lem:lwbar-elements:n=2}, the generation criterion of Proposition~\ref{prop:gen-criterion-invertibles-collections} and Remark~\ref{remark:generation-up-to-quasi-periodicity}, we see that the collections are full. This concludes the proof of this proposition.
\end{proof}

\begin{remark}
\label{remark:quiver-compositions:n=2}
  If one wants to finish describing the category of singularity in the spirit of Proposition~\ref{prop:strong-coll=>unique-dg-enhanc}, one should also study the compositions of the morphisms. It is actually quite easy to describe the compositions on our figures since all the spaces are 1-dimensional. For that we just need to provide generators for these vector spaces and to write down the corresponding relations. We already know that the composition of two vertical arrows is trivial because there are no non-trivial morphisms between the vertices that do not belong to the same unit square. Similarly, the composition of two horizontal arrows is also trivial. So, our relations reduce to comparing the composition of one vertical and one horizontal arrow in each square with the corresponding diagonal arrow.
  
  Let us describe the generators explicitly.
  The short exact sequence~\eqref{eq:exc-coll:n=2-ses-M_xi} for $i=1$ defines an element of $\Ext^1_{\kk[x,y]/(w)}(\kk,\kk(-\x))=\Hom_{\Db}(\kk,\kk(-\x)[1])$ in the bounded derived category of $L_w$-graded $\kk[x,y]/(w)$-modules. Let us denote it by $e_x$. The morphism~$e_x$ is actually non-zero in $\dsgof{L_w}{\kk[x,y]/(w)}$. (This can be checked, for example, by Remark~\ref{remark:db-non-zero-morphism-in-dsg-criterion}.) Thus we can choose~$e_x$ and its shifts as the generators for the horizontal morphism spaces on Figures~\ref{fig:2split}--\ref{fig:2loop}. Similarly, the sequence~\eqref{eq:exc-coll:n=2-ses-M_yj} for $j=1$ defines an element~$e_y$ of $\Hom_{\Db}(\kk,\kk(-\y)[1])$ which is a non-zero morphism in $\Dsg$. Thus we can choose it as the generator for the vertical morphisms spaces. Finally, we can check that the compositions $(e_y(-\x)[1])\circ e_x$ and $(e_x(-\y)[1]) \circ e_y $ are non-zero morphisms from $\kk$ to $\kk(-\x-\y)$ in $\Dsg$, hence either of them can be chosen as the generator for the diagonal morphism space. Moreover, the sum of these two morphisms is a~zero morphism in $\Dsg$. Thus, the only non-trivial relation that we have for the arrows between the shifts of $\kk$ is: $ e_y e_x + e_x e_y = 0$. (Here, we omitted the shifts for the sake of clarity.)
  
  Similarly, we can consider the morphism spaces from $M_x$ and $M_y$ to $\kk$. In this case, the generators are given by the standard module projections $M_x=\kk[x] \xrightarrow{x\mapsto 0} \kk$ and $M_y=\kk[y] \xrightarrow{y\mapsto 0} \kk$ which are also non-trivial morphisms in the singularity category. Let us denote these morphisms by $f_x$ and $f_y$ respectively. Then one can check that the composition $e_x \circ f_y$ gives a non-zero morphism in $\Dsg$, hence provides a generator for the diagonal arrows from $M_y$ to $\kk$. Similarly, $e_y\circ f_x$ provides a generator for the diagonal arrows from $M_x$ to $\kk$.
  
  To complete the picture, let us provide a generator for the diagonal arrow from $M_{xy}(\x+\y)$ to $\kk[1]$ in the loop case. We choose it to be a morphism defined by the extension\\[-1ex]
  $$ 0\to \kk \xrightarrow{xy} \left(\kk[x,y]/(x^2y,xy^2)\right) (\x+\y) \xrightarrow{1} M_{xy}(\x+\y) \to 0. $$
  This morphism is non-trivial in $\Dsg$, so it defines the needed generator.
\end{remark}

\begin{remark}
\label{remark:geometric-2}
  If one uses the geometric language as described in [BFK13], the graded modules will be represented by the equivariant sheaves over $\Aff2$. In particular, $\kk$ will correspond to the structure sheaf of the origin, $M_x$ to the structure sheaf of the $x$-axis, $M_y$ to the structure sheaf of the $y$-axis, and $M_{xy}$ to the structure sheaf of the union of the $x$- and $y$-axes. Then the modules $M_x$, $M_y$ and $M_{xy}$ will be well-defined over $\kk[x,y]/(w)$ exactly when their supports ($x$- and $y$-axes) will lie on the hypersurface $\{w=0\}$ in $\Aff2$.
\end{remark}

\subsection{\texorpdfstring{$n=3$}{n=3}}
\label{sec:n=3}

\newcommand{\splitsub}[1]{Split${}^{\text{\textbf{#1}}}$}
\def\threea{{3-split${}^{\text{\textbf{a}}}$}}
\def\threeb{{3-split${}^{\text{\textbf{b}}}$}}
\def\threec{{3-split${}^{\text{\textbf{c}}}$}}
\def\threed{{3-chain}}
\def\threee{{3-loop}}

Here we consider the case of $n=3$ and $R=\kk[x,y,z]$. In this case, the classification of Example~\ref{ex:invertible-classification-n<=3} gives us 5~different types of invertible polynomials $w\in\kk[x,y,z]$. Let us call the corresponding subcases as follows:\\
(a) \textit{\threea} for $w=x^p+y^q+z^r$, \qquad\; (b) \textit{\threeb} for $w=x^p+xy^q+z^r$,\\ (c) \textit{\threec} for $w=yx^p+xy^q+z^r$,\qquad (d)~\textit{3-chain} for $w=x^p+xy^q+yz^r$,\\ and (e) \textit{3-loop} for $w=zx^p+xy^q+yz^r$, where $p,q,r\ge2$ in each case.

\begin{remark*}
The first three split cases, are obtained from the three 2-dimensional cases considered in \S\ref{sec:n=2} by adding the monomial $z^r$.
\end{remark*}

In this case, we repeat all the steps done for the case of $n=2$. We will also reuse a lot of arguments for $n=2$ as well. The explicit exceptional collections are described in Proposition~\ref{prop:main:n=3} and on Figures~\ref{fig:n=3:split-split}--\ref{fig:n=3:3-loop}.

\subsubsection{The maximal grading groups}

The maximal grading groups $L_w$ on $\kk[x,y,z]$ making all of $x$, $y$, $z$ and $w$ homogeneous, and the corresponding quotient groups $\lwbar = L_w/\<\w\>$, are described in Table~\ref{table:max-grading-groups:n=3}. Similarly to the case of $n=2$, we will use the following fact about the groups~$\lwbar$.

\begingroup
\newcommand{\mystack}[1]{{
\begin{array}{c}#1\end{array}
}}
\newcommand{\mystackl}[1]{{
\begin{array}{l}#1\end{array}
}}

\def\vspbot{{\rule[-1ex]{0pt}{0ex}}}
\def\vsptop{{\rule{0pt}{2.5ex}}}
\def\vspboth{{\vsptop\vspbot}}

\def\zz{{\bar z}}


\begin{table}[h]
\hfill\scalebox{0.85}{
$\begin{array}{|c||c|c|l|}
\hline
\vspboth
  \text{\textbf{case}} & w & L_w & \hfill\lwbar=L_w/\<\w\>\hfill \\
\hhline{|=#=|=|=|}
  \text{\textit{(a) \threea}} & x^p+y^q+z^r
  & \vsptop\ZZ\la \x,\y,\z \ra\!\!\left/\!\la\!\!\footnotesize\mystack{ p\x = q\y \\ = r\z } \!\!\ra\right.
  & \mystackl{{\rule[-2ex]{0pt}{5.5ex}}\ZZ\la\xx,\yy,\zz\ra\!\!\left/\!\la\!\!\!\footnotesize{\mystack{ p\xx=q\yy=\\r\zz=0}}\!\!\ra\right. \\ \simeq \ZZ/p\ZZ \oplus \ZZ/q\ZZ \oplus \ZZ/r\ZZ,  \\
   \text{basis: $(\xx,\yy,\zz)$}
}\\
\hline
  \text{\textit{(b) \threeb}} & x^p+xy^q+z^r
  & \vsptop\ZZ\la \x,\y,\z \ra\!\!\left/\!\la\!\! \footnotesize\mystack{p\x = r\z\\ = \x+q\y}\!\!\ra\right.
  & \mystackl{{\rule[-2ex]{0pt}{6.5ex}}\ZZ\la\xx,\yy,\zz\ra \!\!\!\left/\!\!\la\!\!\footnotesize{\mystack{ p\xx=\\\xx+q\yy=\\r\zz=0}}\!\!\!\ra\right. \\ \simeq \ZZ/pq\ZZ\oplus \ZZ/r\ZZ,\\ \text{basis: $(\yy,\zz)$}
  }\\
\hline
  \text{\textit{(c) \threec}} & yx^p+xy^q+z^r
  & \vsptop\ZZ\la \x,\y,\z \ra\!\!\left/\!\!\la\!\! \footnotesize{\mystack{\y+p\x = \\ \x+q\y =\\ r\z }}\!\!\ra\right.
  & \mystackl{{\rule[-2ex]{0pt}{6.5ex}}\ZZ\la\xx,\yy,\zz\ra\!\!\left/\!\!\la\!\!\footnotesize{\mystack{\yy+p\xx=\\ \xx+q\yy= \\ r\zz = 0}}\!\!\ra\right. \\ \simeq \ZZ/(pq-1)\ZZ\oplus \ZZ/r\ZZ,\\\text{basis: $(\xx,\zz)$ or $(\yy,\zz)$}} \\
\hline
  \text{\textit{(d) 3-chain}} & x^p+xy^q+yz^r
  & \vsptop\ZZ\la \x,\y,\z \ra\!\!\left/\!\!\la\!\! \footnotesize{\mystack{p\x = \\ \x+q\y = \\ \y+r\z }}\!\!\ra\right.
  & \mystackl{{\rule[-2ex]{0pt}{6.5ex}}\ZZ\la\xx,\yy,\zz\ra\!\!\left/\!\!\la\!\!\footnotesize{\mystack{p\xx= \\ \xx+q\yy= \\ \yy+r\zz = 0}}\!\!\ra\right. \\ \simeq \ZZ/pqr\ZZ,\text{ generator: $\zz$}} \\
\hline
  \text{\textit{(e) 3-loop}} & zx^p+xy^q+yz^r
  & \vsptop\ZZ\la \x,\y,\z \ra\!\!\left/\!\!\la\!\! \footnotesize{\mystack{\z+p\x = \\ \x+q\y = \\ \y+r\z }}\!\!\ra\right.
  & \mystackl{{\rule[-2ex]{0pt}{6.5ex}}\ZZ\la\xx,\yy,\zz\ra\!\!\left/\!\!\la\!\!\footnotesize{\mystack{\zz+p\xx= \\ \xx+q\yy= \\ \yy+r\zz = 0}}\!\!\ra\right. \\ \simeq \ZZ/(pqr+1)\ZZ,\\\text{generators: $\xx$, $\yy$, $\zz$}} \\
\hline
\end{array}$
}\hfill
\caption{Maximal grading groups on $\kk[x,y,z]/(w)$}
\label{table:max-grading-groups:n=3}
\end{table}
\endgroup

\def\zz{\bar z}
\begin{lemma}
\label{lem:lwbar-elements:n=3}
  In the three 3-split cases and the 3-chain case, the sequence $(-i\xx-j\yy-k\zz \mid 0\le i\le p-1,\, 0\le j\le q-1,\, 0\le k\le r-1)$ contains all of the elements of the group $\lwbar$. In the 3-loop case, this sequence contains all of the elements of $\lwbar$ except for $\xx+\yy+\zz$.
\end{lemma}
\begin{proof}
  The same as in Lemma~\ref{lem:lwbar-elements:n=2}.
\end{proof}

\def\My{M_y}
\def\Mx{M_x}
\def\Mz{M_z}
\def\Mxy{M_{xy}}
\def\Mxyz{M_{xyz}}

\subsubsection{The objects}
\label{sec:n=3:objects}

Similarly to the case of $n=2$, we will use the shifts of the following modules for the exceptional collections in $\dsgof{L_w}{\kk[x,y,z]/(w)}$. First of all, we use~$\kk$ considered as the module over $\kk[x,y,z]/(w)$ with the trivial actions of $x$, $y$ and~$z$. Then, we consider the $\kk[x,y,z]$-modules $M_x = \kk[x]$, $M_y = \kk[y]$, $M_z = \kk[z]$, $\Mxy=\kk[x,y]/(xy)$ and $\Mxyz = \kk[x,y,z]/(xy,yz,xz)$ with the obvious actions of $x$, $y$ and $z$. Depending on whether $w$ acts on them trivially or not, these modules may or may not be well-defined over the quotient $\kk[x,y,z]/(w)$. This situation is summarized in the following list:
\begin{enumerate}[label=(\alph*), itemsep=0pt, parsep=0pt, topsep=1ex]
    \item $w=x^p+y^q+z^r$ --- only $\kk$ is well-defined,
    \item $w=x^p+xy^q+z^r$ --- $\kk$ and $M_y$ are well-defined,
    \item $w=yx^p+xy^q+z^r$ --- $\kk$, $M_x$, $M_y$ and $\Mxy$ are well-defined,
    \item $w=x^p+xy^q+yz^r$ --- $\kk$, $M_y$ and $M_z$ are well-defined,
    \item $w=zx^p+xy^q+yz^r$ --- $\kk$, $M_x$, $M_y$, $M_z$, $\Mxy$ and $\Mxyz$ are well-defined.
\end{enumerate}

\begin{remark}
   In the proof of fullness for the 3-loop case, we will also use the modules $M_{xz}=\kk[x,z]/(xz)$ and $M_{yz} = \kk[y,z]/(yz)$.
\end{remark}

\subsubsection{The morphisms}
\label{sec:app:explicit-morphisms:n=3}

We follow the same method as in the case of $n=2$. Firstly, we provide the projective resolutions for our modules and then follow Recipe~\ref{recipe}.

All of the modules listed in \S\ref{sec:n=3:objects}, except for $\Mxyz$, are the quotients of $\kk[x,y,z]$ over a regular ideal, so we can apply the procedure from \cite{dyckerhoff} for them the same way as in the case of $n=2$ and use Koszul resolutions as follows.

\newcommand*{\splitintwo}[2]{\raisebox{0.3ex}{$\stackbin[#2]{#1}{}$}}

\newcommand{\arrc}[1]{\begin{array}{c}#1\end{array}}
\newcommand{\arrcscript}[1]{{\scriptsize\arrc{#1}}}

For $\kk$, we notice that $\kk = \kk[x,y,z]/(x,y,z)$ where $(x,y,z)$ is a regular sequence. Decomposing $w=w_x x + w_y y + w_z z$, where $w_x = x^{p-1}$ or $zx^{p-1}$, $w_y = y^{q-1}$, $xy^{q-1}$ or $zy^{q-1}$, and $w_z = z^{r-1}$ or $yz^{r-1}$, depending on the type of $w$, we get the following projective resolution:
\begin{gather*}
P^\bullet_\kk = \left(\raisebox{-1ex}{$
    \dots
    \xrightarrow{\mat{x & y & z & 0 \\ -w_y & w_x & 0 & z \\ -w_z & 0 & w_x & -y \\ 0 & -w_z & w_y & x}} \!
  \arrcscript{A(-2\w) \, \oplus \\ A(-\w-\x-\y) \, \oplus \\ A(-\w-\x-\z) \,\oplus \\ A(-\w-\y-\z)} \!
    \xrightarrow{\mat{w_x & -y & -z & 0 \\ w_y & x & 0 & -z \\ w_z & 0 & x & y \\ 0 & w_z & -w_y & w_x}} \!\!
    \arrcscript{A(-\w-\x) \, \oplus \\ A(-\w-\y) \, \oplus \\ A(-\w-\z) \, \oplus \\ A(-\x-\y-\z)}
  $}
  \right. \\
    \xrightarrow{\mat{x & y & z & 0 \\ -w_y & w_x & 0 & z \\ -w_z & 0 & w_x & -y \\ 0 & -w_z & w_y & x}} \!
    \arrcscript{A(-\w) \,\oplus\\ A(-\x-\y) \,\oplus\\ A(-\x-\z) \,\oplus \\ A(-\y-\z)} \!
    \xrightarrow{\mat{w_x & -y & -z & 0 \\ w_y & x & 0 & -z \\ w_z & 0 & x & y}} \!
   \arrcscript{A(-\x) \, \oplus \\ A(-\y) \, \oplus \\ A(-\z)} \!
    \xrightarrow{\mat{x & y & z}}
  A \Biggr)
    \xrightarrow{1}
  \kk,
\end{gather*}
which is quasi-periodic on the left.

For $M_z$, we have $M_z = \kk[z] = \kk[x,y,z]/(x,y)$, where $(x,y)$ is the regular sequence in $\kk[x,y,z]$. We decompose $w = w_x x + w_y y$ with $w_x = c_1 x^{p-1}$ and $w_y = c_2 y^{q-1} + z^r$, where $c_1 = 1$ or $z$,
and $c_2 = 1$, $x$ or $z$, depending on the case. The projective resolution is then as follows:
\begin{multline*}
P^\bullet_{M_z}=\Biggl(\dots
    \xrightarrow{\mat{w_x & -y \\ w_y & x}} \,
  {\text{\footnotesize $ A(-\w-\x) \oplus A(-\w-\y)$}} \,
    \xrightarrow{\mat{x & y \\ -w_y & w_x}} \,
  {\text{\footnotesize $ A(-\w) \oplus A(-\x-\y)$}} \\
    \xrightarrow{\mat{w_x & -y \\ w_y & x}} \,
  { A(-\x) \oplus A(-\y)}
    \xrightarrow{\mat{x & y}}
  A \Biggr)
  \xrightarrow{1} 
  M_z.\qquad\qquad
\end{multline*}
The projective resolutions for $M_y = \kk[y] = \kk[x,y,z]/(x,z)$, $M_x = \kk[x] = \kk[x,y,z]/(y,z)$ and $M_{xy} = \kk[x,y]/(xy) = \kk[x,y,z]/(xy,z)$ can be built similarly by using decompositions $w = w_x x + w_z z$, $w = w_y y + w_z z$ and $w = w_{xy} xy + w_z z$ respectively with the appropriate coefficients depending on the type of~$w$.

The module $M_{xyz}$ cannot be presented as $\kk[x,y,z]/I$ where $I$ is generated by a regular sequence of elements. However, we notice that there is a short exact sequence $0 \to M_x(-\x) \xrightarrow{x} M_{xyz} \to M_{yz} \to 0$. Thus as soon as we have projective resolutions $\bull P \to M_x$ and $\bull Q \to M_{yz}$, we can use a standard procedure of homological algebra to build a projective resolution for $M_{xyz}$. Namely, we just take the direct sum $\bull P \oplus \bull Q$ and deform the differential. Applying this method to the resolutions of $M_x$ and $M_{yz}$ obtained above, we get the following:
\begin{gather*}
\bull P_{M_{xyz}} = \left( \raisebox{-1ex}{$
    \dots
    \longrightarrow
  \arrcscript{A(-\w-\x-\y) \oplus \\ A(-\w-\x-\z) \oplus \\ A(-\w-\x) \oplus \\ A(-\w-\y-\z)} \!
    \xrightarrow{\mat{y & z & -1 & 0 \\ -w_z & w_y & 0 & y^q \\ 0 & 0 & x & yz \\ 0 & 0 & -w'_{yz} & w'_x}}
  \arrcscript{A(-\w-\x) \oplus \\ A(-\x-\y-\z) \\ \oplus \, A(-\w) \, \oplus \\ A(-\x-\y-\z)} $} \right.
  \\
  \; \xrightarrow{\mat{w_y & -z & y^{q-1} & 0 \\ w_z & y & x^{p-1} & -y \\ 0 & 0 & w'_x & -yz \\ 0 & 0 & w'_{yz} & x}}
  \splitintwo{A(-\x-\y) \oplus A(-\x-\z)}{\oplus A(-\x) \oplus A(-\y-\z)}
    \xrightarrow{\mat{y & z & -1 & 0 \\ 0 & 0 & x & yz}}
  {\scriptstyle A(-\x) \oplus A} \Biggr)
    \xrightarrow{\mat{x & 1}}
  M_{xyz},
\end{gather*}
where $w = zx^p + xy^q + yz^r = w_y y + w_z z = w'_x x + w'_{yz} yz$, with $w_y = xy^{q-1}$, $w_z = x^p + yz^{r-1}$, $w'_x = zx^{p-1}+y^q$ and $w'_{yz} = z^{r-1}$, and the resolution becomes quasi-periodic on the left.

The rest of the steps of Recipe~\ref{recipe} are straightforward and can be done as in Example~\ref{example:morphisms-k[x]/(x^p)}. The morphisms between these objects in the categories of singularities $\dsgof{L_w}{\kk[x,y,z]/(w)}$ are summarized in Table~\ref{table:all-morphisms:n=3}, 
with Remarks~\ref{rem:app-morphisms:table-remark1}, \ref{rem:app-morphisms:table-remark2} and \ref{rem:app-morphisms:table-remark3} applying here as well.

\subsubsection{Exceptional collections}

Using the objects from \S\ref{sec:n=3:objects} and information about the morphisms between them from Table~\ref{table:all-morphisms:n=3}, we can build the exceptional collections in the graded categories of singularities $\dsgof{L_w}{\kk[x,y,z]/(w)}$ as in the following proposition.

\begingroup
\def\myA{{\kk[x,y,z]/(w)}}
\def\mydsg{{\dsgof{L_w}{\kk[x,y,z]/(w)}}}
\begin{prop}\label{prop:main:n=3}
  Let the ring be $R=\kk[x,y,z]$, the potential $w$ be one of the 5 types listed above, the grading group be $L_w$ and the graded $\myA$-modules $\kk$, $M_x$, $M_y$, $M_z$, $M_{xy}$ and $\Mxyz$ be as described in \S\ref{sec:n=3:objects}. Then the corresponding graded category of singularities $\mydsg$ admits a full strongly exceptional collection as follows:
  \begin{enumerate}[label=(\alph*)]
      \item \textbf{\splitsub a case {\normalfont(see Prop.~\ref{prop:futaki-ueda-split} for $n=3$)}:} If $w = x^p+y^q+z^r$, the collection consists of 
      \begin{itemize}[itemsep=0ex, topsep=0ex, parsep=0ex]
          \item $(p-1)(q-1)(r-1)$ shifts of $\kk$ of the form $\kk(-i\x-j\y-k\z)[i+j+k]$ for $0\le i \le p-2$, $0\le j \le q-2$, $0\le k\le r-2$,
      \end{itemize} arranged as on Figure~\ref{fig:n=3:split-split}.
      \item \textbf{\splitsub b case:} If $w = x^p+xy^q+z^r$, the collection consists of $(pq-q+1)(r-1)$ objects, namely
      \begin{itemize}[itemsep=0ex, topsep=0ex, parsep=0ex]
          \item $(p-1)(q-1)(r-1)$ shifts of $\kk$ as in (a),
          \item and $p(r-1)$ shifts of $M_y$ of the form $M_y(-i\x-k\z)[i+k]$ for $-1 \le i \le p-2$ and $0\le k\le r-2$,
      \end{itemize} arranged as on Figure~\ref{fig:n=3:chain-split}.
      \item \textbf{\splitsub c case:} If $w = yx^p+xy^q+z^r$, the collection consists of $pq(r-1)$ objects, namely
      \begin{itemize}[itemsep=0ex, topsep=0ex, parsep=0ex]
          \item $(p-1)(q-1)(r-1)$ shifts of $\kk$ as in (a),
          \item $(p-1)(r-1)$ shifts of $M_y$ of the form as in (b) for $0\le i \le p-2$ and $0\le k\le r-2$,
          \item $(q-1)(r-1)$ shifts of $M_x$ of the form $M_x(-j\y-k\z)[j+k]$ for $0\le j\le q-2$ and $0\le k\le r-2$,
          \item and $r-1$ shifts of $\Mxy$ of the form $M_{xy}(\x+\y-k\z)[-1+k]$ for $0\le k\le r-2$,
      \end{itemize} arranged as on Figure~\ref{fig:n=3:loop-split}.
      \item \textbf{Chain case:} If $w = x^p+xy^q+yz^r$, the collection consists of $pqr-qr+r-1$ objects, namely
      \begin{itemize}[itemsep=0ex, topsep=0ex, parsep=0ex]
          \item $(p-1)(q-1)(r-1)$ shifts of $\kk$ as in (a),
          \item $p(r-1)$ shifts of $M_y$ as in (b),
          \item and $(p-1)q$ shifts of $M_z$ of the form $M_z(-i\x-j\y)[i+j]$ for $0\le i\le p-2$, $0\le j\le q-2$, and for $-1\le i\le p-3$, $j=-1$,
      \end{itemize} arranged as on Figure~\ref{fig:n=3:3-chain}.
      \item \textbf{Loop case:} If $w = zx^p+xy^q+yz^r$, the collection consists of $pqr$ objects, namely
      \begin{itemize}[itemsep=0ex, topsep=0ex, parsep=0ex]
          \item $(p-1)(q-1)(r-1)$ shifts of $\kk$ as in (a),
          \item $p(r-1)$ shifts of $M_y$ as in (b),
          \item $(q-1)r$ shifts of $M_x$ of the form as in (c) for $0\le j\le q-2$ and $-1\le k\le r-2$,
          \item $(p-1)q$ shifts of $M_z$ of the form as in (d) for $0\le i \le p-2$ and $-1\le j\le q-2$,
          \item and one extra-object $\Mxyz(\x+\y+\z)[-2]$,
      \end{itemize} arranged as on Figure~\ref{fig:n=3:3-loop}.
  \end{enumerate}
  \textbf{Legend:} (a) On Figures~\ref{fig:n=3:split-split}--\ref{fig:n=3:3-loop}, the objects are placed in the 3-dimensional grid as follows. The object~$\kk$ is always placed at the origin. It is situated at the bottom-front-left corner of the red cube. All other objects are placed at the points with the following coordinates:
  \begin{itemize}[itemsep=0ex, topsep=0.5ex, parsep=0ex]
      \item $\kk(-i\x-j\y-k\z)[i+j+k]$\quad at\quad $(i,j,k)$ -- the red cube,
      \item $M_z(-i\x-j\y)[i+j]$\quad at\quad $(i,j,-1)$ -- the bottom/blue face,
      \item $M_y(-i\x-k\z)[i+k]$\quad at\quad $(i,-1,k)$ -- the front/green face,
      \item $M_x(-j\y-k\z)[j+k]$\quad at\quad $(-1,j,k)$ -- the left/yellow face,
      \item $M_{xy}(\x+\y-k\z)[-1+k]$\quad at\quad $(-1,-1,k)$ -- the front-left edge,
      \item and $M_{xyz}(\x+\y+\z)[-2]$\quad at\quad $(-1,-1,-1)$ -- the bottom-front-left point.
  \end{itemize}
  \vspace{0.5ex}
  (b) On these figures, we did not write the shifts $[\bullet]$ so that to not clutter up the space. However, it should be assumed that the objects are shifted as in the above listing.\\[0.5ex]
  (c) For the objects lying in the same unit cube, there are arrows between them in directions $(\eps_x, \eps_y, \eps_z)$ for $\eps_x, \eps_y, \eps_z \in \{0,1\}$, except for the cases indicated under each figure. Each arrow denotes a one-dimensional morphism space between the corresponding shifted objects except for one arrow from $M_{xyz}(\x+\y+\z)[-2]$ to $\kk$ in the loop case which denotes a two-dimensional morphism space.
\end{prop}
\endgroup

\begin{remark}
\label{remark:thom-sebastiani:2->3}
The cases (a), (b) and (c) are split in a sense that $w$ can be decomposed as $w = v + z^r$, where $v$ is an invertible polynomial in $\kk[x,y]$. In these cases, our 3-dimensional collections can be thought as the ``Thom-Sebastiani products'' of the 2-dimensional collections of Proposition~\ref{prop:main:n=2} for $R=\kk[x,y]$, $w=v$, with the 1-dimensional collection of Example~\ref{ex:main:n=1} for $R=\kk[z]$, $w=z^r$. Namely, we see that all the horizontal slices of Figures~\ref{fig:n=3:split-split}--\ref{fig:n=3:loop-split} coincide with the 2-dimensional diagrams of Figures~\ref{fig:2split}--\ref{fig:2loop} shifted by $(-k\z)$, where the slices are indexed by $0\le k \le r-2$ starting from the bottom one, and where we treat modules over $\kk[x,y]/(v)$ as the modules over $\kk[x,y,z]/(w)$ with the trivial action of~$z$. Due to this fact, most of the arguments for the split cases reduce to the arguments made for the corresponding invertible polynomials in the case of $n=2$.
\end{remark}

\begingroup

\def\scri{\scriptstyle}
\def\bigcomma{\textstyle,}
\def\bico{\text{\normalsize,}}
\def\sep{\bico\\[1em]}
\newcommand{\arrcsmall}[1]{\footnotesize\arrc{#1}}

\newcommand{\twoshi}[2]{\arrcsmall{#1\sep#2}}

\newcommand{\dectwo}[1]{\text{\footnotesize$#1$}}
\newcommand{\simpledec}[2]{\twoshifts{\dectwo{#1}}{\dectwo{#2}}}

\newcommand{\caseformat}[1]{\text{\textit{\footnotesize #1}}}
\def\kcases{\caseformat{(all cases)}}
\def\Mycases{\caseformat{(cases (b,c,d,e))}}
\def\Mxcases{\caseformat{(cases (c,e))}}
\def\Mxycases{\caseformat{(case (c) only)}}
\def\Mzcases{\caseformat{(cases (d,e))}}
\def\Mxyzcases{\caseformat{(case (e) only)}}

\def\Myzcases{\caseformat{(no cases)}}

\begin{table}[p]
\centering
\scalebox{0.85}{
$
\begin{array}{|c||c|c|c|}
\hline
  \twoshialign{l}{\Hom(\downarrow,\rightarrow)}
  {\Hom(\downarrow,\rightarrow[1])}
  & \arrc{\kk\\\kcases} & \arrc{M_y\\\Mycases} & \arrc{M_x\\\Mxcases} \\
\hhline{|=#=|=|=|}
  \arrc{\kk\\\kcases} &
    \twoshi{\kk\oplus\kk(\x+\y-\w)\\{}\oplus\kk(\x+\z-\w)\\ {}\oplus\kk(\y+\z-\w)}{\kk(\x)\oplus\kk(\y)\oplus\kk(\z) \\{}\oplus\kk(\x+\y+\z-\w)} &
    \twoshi{\kk(\x+\y-\w)\oplus{}\\\kk(\y+\z-\w)}{\kk(\y)\oplus{}\\\kk(\x+\y+\z-\w)} &
    \twoshi{\kk(\x+\y-\w)\oplus{}\\\kk(\x+\z-\w)}{\kk(\x)\oplus{}\\\kk(\x+\y+\z-\w)} \\
\hline
  \arrc{M_y\\\Mycases} &
    \twoshifts{\kk\oplus \kk(\x+\z-\w)}{\kk(\x)\oplus\kk(\z)} &
    \twoshifts{\kk[y]/(y^q)}{\kk[y]/(y^q)(\z)
    } &
    \twoshifts{\kk(\x+\z-\w)}{\kk(\x)}
    \\
\hline
  \arrc{M_x\\\Mxcases} &
    \twoshifts{\kk\oplus\kk(\y+\z-\w)}{\kk(\y)\oplus\kk(\z)} &
    \twoshifts{\kk(\y+\z-\w)}{\kk(\y)} &
    \twoshifts{\kk[x]/(x^p)}{\kk[x]/(x^p)(\varepsilon)\\[-0.5ex]\!\!\Bigl[\!\!\arrc{\scri\text{where $\varepsilon=\z$ for \textit{(c)},}\\[-1ex]\scri\text{and $\varepsilon=\y$ for \textit{(e)}}}\!\!\Bigr]\!\!} \\
\hline
  \arrc{M_{xy}\\\Mxycases} &
    \twoshifts{\dectwo{\kk\oplus\kk(\x+\y+\z-\w)}}
    {\dectwo{\kk(\z)\oplus\kk(\x+\y)}} &
    \twoshifts{\kk[y]/(y^{q-1})}{\kk[y]/(y^{q-1})(\z)} &
    \twoshifts{\kk[x]/(x^{p-1})}{\kk[x]/(x^{p-1})(\z)} \\
\hline
  \arrc{M_z\\\Mzcases} &
    \twoshifts{\kk\oplus \kk(\x+\y-\w)}{\kk(\x)\oplus\kk(\y)} &
    \twoshifts{\kk(\x+\y-\w)}{\kk(\y)} &
    \twoshifts{\kk(\x+\y-\w)}{\kk(\x)} \\
\hline
  \arrc{M_{xyz}\\\Mxyzcases} &
    \twoshi{\kk\op{}\\\kk(\x+\y+\z-\w)^{\op2}}{\kk(\x+\y)\op\kk(\x+\z)\\{}\op\kk(\y+\z)} &
    \twoshi{\kk[y]/(y^{q-1})\op{}\\\kk(\x+\y+\z-\w)}{\kk[y]/(y^{q})(\y+\z)} &
    \twoshi{\kk[x]/(x^{p-1})\op{}\\\kk(\x+\y+\z-\w)}{\kk[x]/(x^{p})(\x+\y)} \\
\hline
\end{array}
$}
\\[2em]
\scalebox{0.85}{
$
\begin{array}{|c||c|c|c|}
\hline
  \twoshialign{l}{\Hom(\downarrow,\rightarrow)}
  {\Hom(\downarrow,\rightarrow[1])}
  & \arrc{M_{xy}\\\Mxycases} & \arrc{M_z\\\Mzcases} & \arrc{M_{xyz}\\\Mxyzcases} \\
\hhline{|=#=|=|=|}
  \arrc{\kk\\\kcases} &
    \twoshi{\kk(\z-\w)\op{}\\\kk(\x+\y-\w)}{\kk\op{}\\\kk(\x+\y+\z-\w)} &
    \twoshi{\kk(\x+\z-\w)\oplus{}\\\kk(\y+\z-\w)}{\kk(\z)\oplus{}\\\kk(\x+\y+\z-\w)} &
    \twoshi{\kk(\x-\w)\op\kk(\y-\w)\\{}\op\kk(\z-\w)}{\kk^{\op2}\op{}\\\kk(\x+\y+\z-\w)} \\
\hline
  \arrc{M_y\\\Mycases} &
    \twoshifts{\dectwo{\kk[y]/(y^{q-1})(-\y)}}{\dectwo{\kk[y]/(y^{q-1})(\z-\y)}} &
    \twoshifts{\kk(\x+\z-\w)}{\kk(\z)} &
    \simpledec{\kk[y]/(y^q)(-\y)}{\kk\op\kk[y]/(y^{q-1})(\z-\y)} \\
\hline
  \arrc{M_x\\\Mxcases} &
    \simpledec{\kk[x]/(x^{p-1})(-\x)}{\kk[x]/(x^{p-1})(\z-\x)} &
    \twoshifts{\kk(\y+\z-\w)}{\kk(\z)} &
    \simpledec{\kk[x]/(x^p)(-\x)}{\kk\op\kk[x]/(x^{p-1})(\y-\x)}
     \\
\hline
  \arrc{M_{xy}\\\Mxycases} &
    \twoshi{\kk[x](x^{p-1})\op{}\\\kk[y]/(y^{q-1})(-\y)}{\kk[x](x^{q-1})(\z)\op{}\\\kk[y]/(y^{q-1})(\z-\y)} & & \\
\hline
  \arrc{M_z\\\Mzcases} &
     &
    \twoshifts{\kk[z]/(z^r)}{\kk[z]/(z^r)(\x)} &
    \simpledec{\kk[z]/(z^r)(-\z)}{\kk\op\kk[z]/(z^{r-1})(\x-\z)} \\
\hline
  \arrc{M_{xyz}\\\Mxyzcases} &
    &
    \twoshi{\kk[z]/(z^{r-1})\op{}\\\kk(\x+\y+\z-\w)}{\kk[z]/(z^{r})(\x+\z)} &
    \twoshi{\kk[x]/(x^p)\op{}\\\kk[y]/(y^{q-1})(-\y)\op{}\\\kk[z]/(z^{r-1})(-\z)}{\kk[x]/(x^p)(\y)\op{}\\\kk[y]/(y^{q-1})(\z)\op{}\\\kk[z]/(z^{r-1})(\x)}  \\
\hline
\end{array}
$}
\caption{Morphisms between the objects in $\dsgof{}{\kk[x,y,z]/(w)}$ for the five types of invertible polynomials $w$ listed in the beginning of \S\ref{sec:n=3}}
\label{table:all-morphisms:n=3}
\end{table}

\endgroup

\begin{figure}[p]
\centering
\CubeThreeDimSplit
\caption{\textbf{(\threea{})} Exceptional collection in $\dsgof{L_w}{\kk[x,y,z]/(x^p+y^q+z^r)}$}
\label{fig:n=3:split-split}
\end{figure}

\begin{figure}[p]
\centering
\CubeThreeDimSplitChain
\\[2ex]
There are no arrows in directions $(1,0,*)$ inside the front/green face, where $*=0,1$.
\caption{\textbf{(\threeb)} Exceptional collection in $\dsgof{L_w}{\kk[x,y,z]/(x^p+xy^q+z^r)}$}
\label{fig:n=3:chain-split}
\end{figure}

\begin{figure}[p]
\centering
\CubeThreeDimSplitLoop
\\[2ex]
There are no arrows in directions $(0,1,*)$ inside the left/yellow face and no arrows in directions $(1,0,*)$ inside the front/green face, where $*=0,1$.
\caption{\textbf{(\threec)} Exceptional collection in $\dsgof{L_w}{\kk[x,y,z]/(yx^p+xy^q+z^r)}$}
\label{fig:n=3:loop-split}
\end{figure}

\begin{figure}[p]
\centering
\scalebox{0.95}{
\CubeThreeDimChainNotStrong
}
\\[2ex]
There are no arrows in directions $(1,0,*)$ inside the front/green face and no arrows in directions $(*,1,0)$ inside the bottom/blue face, where $*=0,1$.
\caption{\textbf{(3-chain, \textit{not strong})} Exceptional collection in $\dsgof{L_w}{\kk[x,y,z]/(w)}$ for $w=x^p+xy^q+yz^r$, which is not strong for $p>2$ due to an extra-arrow from $\Mz(-(p-2)\x+\y)[p-3]$ to $\My(\x)[-1]$ of degree~$p-2$}
\label{fig:n=3:3-chain-not-strong}
\end{figure}

\begin{figure}[p]
\centering
\CubeThreeDimChain
\\[2ex]
There are no arrows in directions $(1,0,*)$ inside the front/green face and no arrows in directions $(*,1,0)$ inside the bottom/blue face, where $*=0,1$.
\caption{\textbf{(3-chain)} Exceptional collection in $\dsgof{L_w}{\kk[x,y,z]/(x^p+xy^q+yz^r)}$}
\label{fig:n=3:3-chain}
\end{figure}

\begin{figure}[p]
\centering
\scalebox{0.95}{
\CubeThreeDimLoop
}
\\[2ex]
There are no arrows in directions $(1,0,*)$ inside the front/green face, no arrows in directions $(*,1,0)$ inside the bottom/blue face and no arrows in directions $(0,*,1)$ inside the left/yellow face, where $*=0,1$. 
\caption{\textbf{(3-loop)} Exceptional collection in $\dsgof{L_w}{\kk[x,y,z]/(zx^p+xy^q+yz^r)}$}
\label{fig:n=3:3-loop}
\end{figure}

\subsubsection{Verifying exceptionality}

Similarly to the case of $n=2$, in order to prove Proposition~\ref{prop:main:n=3} we firstly need to check that the listed collections are indeed exceptional and strong and have the morphisms exactly as described on the corresponding figures. We do this in the following lemma by using information about the morphisms from Table~\ref{table:all-morphisms:n=3} and about the group~$\lwbar$ from Table~\ref{table:max-grading-groups:n=3}.

\begin{lemma}
\label{lemma:n=3:check}
  The collections described in Proposition~\ref{prop:main:n=3} are strongly exceptional and the morphisms are as depicted on the corresponding figures.
\end{lemma}
\begin{proof}
  The argument is the same as in the proof of Lemma~\ref{lemma:main:n=2:checking-collections}. For each type of $w$ we need to see which modules we are using in our collection among $\kk$, $M_z$, $M_y$, $M_x$, $M_{xy}$ and $M_{xyz}$ and we should consider all different pairs of such modules. When we fixed the first and the second module, we should use information about the morphisms between them from Table~\ref{table:all-morphisms:n=3} and properties of the grading group from Table~\ref{table:max-grading-groups:n=3} in order to conclude that the morphisms between the shifts of the first module to the shifts of the second module are only the ones depicted on the corresponding figures.
  
  In the \threea{} case, only the module~$\kk$ is used, so we should consider only one pair~$(\kk,\kk)$. From the structure of the group~$\lwbar$ and of the morphisms from~$\kk$ to the shifts of~$\kk$, it follows that the only morphisms that may occur on Figure~\ref{fig:n=3:split-split} are the morphisms to the adjacent objects in directions $(\eps_x,\eps_y,\eps_z)\in\{0,1\}^3$. Similarly, in the \threeb{} case, we have two modules~$\kk$ and~$M_y$, hence four possible pairs. In the \threec{} case, we have four objects~$\kk$, $M_y$, $M_x$ and $M_{xy}$, hence 16 pairs. However, due to the $x$-$y$-symmetry, it is enough to consider only 10 different pairs in this case. In the 3-chain case, three modules are used, so we have $9$ possible pairs. In the 3-loop case, five modules are used, so we have 25~different pairs. However, due to a cyclic $x$-$y$-$z$-symmetry, we do not need to consider all of them. Indeed, there are 4~pairs made out of objects~$\kk$ and~$M_{xyz}$ which are symmetrical to themselves. The rest 21~pairs decompose into 7~orbits of length~$3$ under the cyclic permutation of~$x$, $y$ and~$z$. This gives us only $4+7=11$ pairs that we need to consider in this case.
  
  The computations in each case for each pair of objects are straightforward and are done exactly as in the proof of Lemma~\ref{lemma:main:n=2:checking-collections}, so we omit them here.
\end{proof}

\begin{remark}
  Similarly, one can check that the non-strong collection for the 3-chain potential depicted on Figure~\ref{fig:n=3:3-chain-not-strong} is also exceptional and has the morphisms as described.
\end{remark}


\begin{remark}
  Due to Remark~\ref{remark:thom-sebastiani:2->3}, the computations in the three split cases, when $w=v+z^r$, will mostly repeat the ones performed for the corresponding polynomials~$v$ in the case~$n=2$. Indeed, in these cases, the group $\lwbar$ decomposes as $\lvbar\oplus \ZZ/r\ZZ$, and one may notice that the morphisms between the modules for~$w$ from Table~\ref{table:all-morphisms:n=3} are related to the morphisms between the modules for~$v$ from Table~\ref{table:all-morphisms:n=2} as follows:
  \vspace{-1ex}
  \begin{gather*}
  \Hom_{\dsgof{}{w}}(X,Y) = \Hom_{\dsgof{}{v}}(X,Y) \oplus \Hom_{\dsgof{}{v}}(X,Y[1])(\z-\w),\\
  \Hom_{\dsgof{}{w}}(X,Y[1]) = \Hom_{\dsgof{}{v}}(X,Y[1]) \oplus \Hom_{\dsgof{}{v}}(X,Y)(\z),
  \end{gather*}
  where $\w=r\z$, $\dsgof{}{w}$ and $\dsgof{}{v}$ denote the ungraded categories of singularities $\dsgof{}{\kk[x,y,z]/(w)}$ and $\dsgof{}{\kk[x,y]/(v)}$ respectively, and $X$ and $Y$ are any of the modules $\kk$, $M_x$, $M_y$ and $M_{xy}$ considered as modules over both $\kk[x,y]/(v)$ and $\kk[x,y,z]/(w)$. In particular, this implies that for any two objects~$X$ and~$Y$ of the collection for $\dsgof{L_v}{\kk[x,y]/(v)}$, the morphisms in $\dsgof{L_w}{\kk[x,y,z]/(w)}$ from $X(-k_1\z)$ to $Y(-k_2\z)[\bullet]$ (from the $k_1$-th to the $k_2$-th slice) exist only when $k_2 = k_1$ or $k_2 = k_1+1$ and when the morphisms from $X$ to $Y[\bullet]$ existed in $\dsgof{L_v}{\kk[x,y]/(v)}$.
\end{remark}

\subsubsection{Generating shifts of \texorpdfstring{$\kk$}{k}}
\label{sec:n=3:generating-shifts-of-k}

Here, we generate shifts of $\kk$ from the objects of the given collections similarly to the case of $n=2$. We do this in Lemmas~\ref{lemma:prop:main:n=3:generates:a,b,c}, \ref{lemma:prop:main:n=3:generates:d} and \ref{lemma:prop:main:n=3:generates:e} for the three split cases, the chain case and the loop case respectively.

\begin{lemma}
\label{lemma:prop:main:n=3:generates:a,b,c}
  The collections described in Proposition~\ref{prop:main:n=3}(a,b,c) generate the modules $\kk(-i\x-j\y-k\z)$ for all $0\le i\le p-1$, $0\le j\le q-1$ and $0\le k\le r-1$.
\end{lemma}
\begin{proof}
  By Remark~\ref{remark:thom-sebastiani:2->3}, the horizontal slices of the 3-dimensional collections coincide with the corresponding 2-dimensional collections shifted by $(-k\z)$ for the $k$-th slice. Thus we can repeat the proof of Lemma~\ref{lemma:main:n=2:checking-fullness} verbatim for each of the slices. In this way, we show that the objects of the $k$-th slice generate $\kk(-i\x-j\y-k\z)$ for all $0\le i\le p-1$ and $0\le j\le q-1$, where $0\le k\le r-2$.
  
  After that, similarly to the proof of Lemma~\ref{lemma:main:n=2:checking-fullness}, we notice that in each of our three split cases, the object $M_{z,r} = \kk[z]/(z^r)$ is zero in $\dsgof{L_w}{\kk[x,y,z]/(w)}$ by Lemma~\ref{lemma-perfect}, hence we have a generation rule analogous to \eqref{eq:exc-coll:n=2-gen-rule-k-simple-x} and \eqref{eq:exc-coll:n=2-gen-rule-k-simple-y} for the $z$-direction (for each $l\in L_w$):
  \vspace{-1ex}
  \begin{equation}
    \label{eq:exc-coll:n=2-gen-rule-k-simple-z}
    \kk(-(r-1)\z+l) \in \Bigl\langle \kk(l), \dots, \kk(-(r-2)\z+l) \Bigr\rangle.
  \end{equation}
  \\[-3ex]
  Applying this rule for $l = -i\x-j\y$, where $0\le i\le p-1$ and $0\le j\le q-1$, we generate the modules $\kk(-i\x-j\y-(r-1)\z)$ for all such $i$ and $j$. Summarizing, we see that we already got all the shifts of type $\kk(-i\x-j\y-k\z)$ for all $0\le i\le p-1$, $0\le j\le q-1$ and $0\le k\le r-1$ as needed.
\end{proof}

\begin{lemma}
\label{lemma:prop:main:n=3:generates:d}
  The collection described in Proposition~\ref{prop:main:n=3}(d) generates the modules $\kk(-i\x-j\y-k\z)$ for all $0\le i\le p-1$, $0\le j\le q-1$ and $0\le k\le r-1$.
\end{lemma}
\begin{proof}
  Similarly to the proofs of Lemma~\ref{lemma:main:n=2:checking-fullness} and Lemma~\ref{lemma:prop:main:n=3:generates:a,b,c}, we will use the given objects to add more shifts of $\kk$ in 3 steps:
  \begin{enumerate}[label={\textit{\arabic*)}}, itemsep=0ex, parsep=0ex, topsep=0.5ex]
      \item expand in $x$-direction by adding $\kk(-(p-1)\x-j\y-k\z)$ for different $j$, $k$;
      \item expand in $y$-direction by adding $\kk(-i\x-(q-1)\y-k\z)$ for different $i$, $k$;
      \item expand in $z$-direction by adding $\kk(-i\x-j\y-(r-1)\z)$ for different $i$, $j$.
  \end{enumerate}

  To perform $x$-expansion for $\kk$, we just notice that $M_{x,p}=\kk[x]/(x^p)$ is perfect by Lemma~\ref{lemma-perfect}, hence we can use relation \eqref{eq:exc-coll:n=2-gen-rule-k-simple-x} for $l=-j\y-k\z$ to generate all $\kk(-(p-1)\x-j\y-k\z)$ for $0\le j\le q-2$ and $0\le k \le r-2$.
  
  To prepare for the later $y$- and $z$-expansions for $\kk$, we also need to perform $x$-expansion for the auxiliary modules~$M_y$ and~$M_z$. Namely, we consider the modules $M_{z,(x,i)} = \kk[x,z]/(x^i)$, for $1\le i\le p$, where $M_{z,(x,1)} = M_z$ and $M_{z,(x,p)}$ is perfect. Similarly to \eqref{eq:exc-coll:n=2-ses-M_xi}, we consider the following short exact sequences for $1\le i\le p-1$:
  \vspace{-1ex}
  \begin{equation}
    \label{eq:exc-coll:n=3-ses-M_z_xi}
    0\to M_z(-i\x) \xrightarrow{x^i} M_{z,(x,i+1)} \xrightarrow{1} M_{z,(x,i)} \to 0.
  \end{equation}
  Similarly to \eqref{eq:exc-coll:n=2-gen-rule-k-simple-x}, these sequences give rise to the relation
  $$ M_z(-(p-1)\x+l)\in \Bigl< M_z(l),\dots, M_z(-(p-2)\x+l) \Bigr>. $$
  Using the already provided shifts of $M_z$ and this relation, we generate $M_z(-(p-1)\x-j\y)$ for $0\le j\le q-2$ by setting $l=-j\y$, and generate $M_z(-(p-2)\x+\y)$ by setting $l=\y+\x$. Finally, let us note that we already have $M_z(-(p-1)\x+\y)=M_z(\x+\y-\w)=M_z(\x+\y)[-2]$ and $M_y(-(p-1)\x-k\z) = M_y(\x-\w-k\z) = M_y(\x-k\z)[-2]$ for $0\le k\le r-2$.
  
  Now, in order to perform $y$-expansion for~$\kk$, we just use the relations \eqref{eq:exc-coll:n=2-gen-rule-y} for $l=-i\x-k\z$ and the already obtained shifts of $\kk$ and $M_y$ to get $\kk(-i\x-(q-1)\y-k\z)$ for all $0\le i\le p-1$ and $0\le k\le r-2$. (Note that we were able to increase the bound for $i$ from $p-2$ to $p-1$ due to the previous $x$-expansion for~$\kk$ and~$M_y$.) Also, we note that we already have $M_z(-i\x-(q-1)\y) =M_z(-i\x-\w+\x+\y)=M_z(-(i-1)\x+\y)[-2]$ for all $0\le i\le p-1$.
  
  Finally, to perform $z$-expansion for $\kk$, we notice that similarly to \eqref{eq:exc-coll:n=2-gen-rule:k(-ix)+Mx+Mx(y)=>k+} and \eqref{eq:exc-coll:n=2-gen-rule:k(-jy)+My+My(x)=>k+}, we get the following relation in the 3-chain case (since $\y+r\z=\w$):
  \begin{equation}
  \label{eq:exc-coll:n=2-gen-rule:k(-kz)+Mz+Mz(y)=>k+}
    \kk(-(r-1)\z+l) \in \Bigl\langle \kk(l), \dots, \kk(-(r-2)\z+l), M_z(l), M_z(\y+l) \Bigr\rangle.
  \end{equation}
  Using this relation for $l=-i\x-j\y$ and the previously obtained shifts of~$\kk$ and~$M_z$, we get $\kk(-i\x-j\y-(r-1)\z)$ for all $0\le i\le p-1$ and $0\le j\le q-1$.
  
  Summarizing, we got $\kk(-i\x-j\y-k\z)$ for all $0\le i\le p-1$, $0\le j\le q-1$ and $0\le k\le r-1$. Hence the proof.
\end{proof}
  
\begin{lemma}
\label{lemma:prop:main:n=3:generates:e}
  The collection described in Proposition~\ref{prop:main:n=3}(e) generates the modules $\kk(-i\x-j\y-k\z)$ for all $0\le i\le p-1$, $0\le j\le q-1$ and $0\le k\le r-1$, and the module $\kk(\x+\y+\z)$.
\end{lemma}
\begin{proof}
\def\Mxyz{M_{xyz}}
  Let us follow the same steps as in the proof of Lemma~\ref{lemma:prop:main:n=3:generates:d}. For the $x$-expansion, similarly to \eqref{eq:exc-coll:n=2-gen-rule:k(-ix)+Mx+Mx(y)=>k+}, we have the following relation: 
  $$ \kk(-(p-1)\x+l) \in \Bigl\langle \kk(l), \dots, \kk(-(p-2)\x+l), M_x(l), M_x(\z+l) \Bigr\rangle.$$
  (The difference from \eqref{eq:exc-coll:n=2-gen-rule:k(-ix)+Mx+Mx(y)=>k+} is only in the last shift, since we have $\z+p\x=\w$ in this case, instead of $\y+p\x=\w$.) We apply this relation for $l=-j\y-k\z$ and obtain $\kk(-(p-1)\x-j\y-k\z)$ for all $0\le j\le q-2$ and $0\le k\le r-2$.
  
  In order to prepare for the $y$- and $z$-expansions for the shifts of~$\kk$, we also need to perform $x$-expansion for the shifts of $M_z$, similarly to how we did it in the proof of Lemma~\ref{lemma:prop:main:n=3:generates:d}. Namely, we consider the modules $M_{z,(x,i)} = \kk[x,z]/(x^i)$, for $1\le i\le p$, where $M_{z,(x,1)}=M_z$, and also an extra-module $M_{z,(\overline{x,p})} = \kk[x,z]/(zx^p)$ which is perfect by Lemma~\ref{lemma-perfect} (instead of $M_{z,(x,p)}$). In addition to \eqref{eq:exc-coll:n=3-ses-M_z_xi}, we will use the following short exact sequences of modules:\vspace{-1ex}
  \begin{gather*}
      0 \to M_{xz}(-(p-1)\x) \xrightarrow{x^{p-1}} M_{z,(\overline{x,p})} \xrightarrow{1} M_{z,(x,p-1)} \to 0,\\
      0 \to M_x(-p\x) \xrightarrow{x^p} M_{z,(\overline{x,p})} \xrightarrow{1} M_{z,(x,p)} \to 0.
  \end{gather*}
  These sequences and \eqref{eq:exc-coll:n=3-ses-M_z_xi} for $1 \le i \le p-1$ give rise to the following relations:\vspace{-1ex}
  \begin{gather*}
   M_{xz}(\x+\z+l)\in \Bigl< M_z(l),\dots, M_z(-(p-2)\x+l)\Bigr>, \\
   M_z(-(p-1)\x+l)\in \Bigl< M_z(l),\dots, M_z(-(p-2)\x+l),\; M_x(\z+l) \Bigr>,
  \end{gather*}
  \par\vspace{-1ex}\noindent
  where we use that $-(p-1)\x = \x+\z-\w$ and $-p\x = \z-\w$. Using the second relation for $l=-j\y$ and the available shifts of $M_z$ and $M_x$, we generate $M_z(-(p-1)\x-j\y)$ for all $0\le j\le q-2$. Using the first relation for $l=\y$, we generate $M_{xz}(\x+\y+\z)$. After that, we use the module $M_{xyz}(\x+\y+\z)$ from our collection and the sequence $$ 0\to M_y(\x+\z)\xrightarrow{y} M_{xyz}(\x+\y+\z) \xrightarrow{1} M_{xz}(\x+\y+\z) \to 0 $$ in order to generate $M_y(\x+\z)$ as well. Repeating the argument for a cyclic permutation of variables $x$, $y$, $z$ and $p$, $q$, $r$, we also get $M_x(\y+\z)$ and $M_z(\x+\y)$.
  
  Now, we can perform $y$-expansion for $\kk$ in a straightforward way. Firstly, we note that we have got all the modules $M_y(-(p-1)\x-k\z)= M_y(x-(k-1)\z-\w)\simeq M_y(\x-(k-1)\z)[-2]$ for $0\le k\le r-1$ (for $k>0$ they were all present in the initial collection, and for $k=0$ we have just obtained $M_y(\x+\z)$ in the previous step). We also already had $M_y(-i\x-k\z)$ for $0\le i\le p-2$ and $0\le k\le r-2$. Thus, we can use \eqref{eq:exc-coll:n=2-gen-rule:k(-jy)+My+My(x)=>k+} for $l=-i\x-k\z$ in order to generate $\kk(-i\x-(q-1)\y-k\z)$ for all $0\le i\le p-1$ and $0\le k\le r-2$.
  
  Now, we can perform $z$-expansion for $\kk$ by using~\eqref{eq:exc-coll:n=2-gen-rule:k(-kz)+Mz+Mz(y)=>k+}. For that, we have all the required shifts of $M_z$ except for $M_z(\y-(p-1)\x)=M_z(\x+\y+\z-\w)\simeq M_z(\x+\y+\z)[-2]$. The latter shift can be obtained from the following exact sequence:
  \vspace{-1ex}
  $$ 0\to M_x(\y+\z)\oplus M_y(\x+\z) \xrightarrow{(x,y)} \Mxyz(\x+\y+\z) \xrightarrow{1} M_z(\x+\y+\z)\to 0, $$
  since we already have $\Mxyz(\x+\y+\z)$, $M_x(\y+\z)$ and $M_y(\x+\z)$. After this, we use \eqref{eq:exc-coll:n=2-gen-rule:k(-kz)+Mz+Mz(y)=>k+} for $l=-i\x-j\y$ and get the shifts $\kk(-i\x-j\y-(r-1)\z)$ for all $0\le i\le p-1$ and $0\le j\le q-1$.
  
  Summarizing, we have obtained $\kk(-i\x-j\y-k\z)$ for all $0\le i\le p-1$, $0\le j\le q-1$ and $0\le k\le r-1$. Now in order to finish the proof, we also need to show how to generate $\kk(\x+\y+\z)$. For that, we use the $(\x+\y+\z)$-shift of the short exact sequence
  \vspace{-1ex}
  $$ 0\to M_x(-\x)\oplus M_y(-\y)\oplus M_z(-\z) \xrightarrow{(x,y,z)} \Mxyz \xrightarrow{1} \kk\to 0, $$
  and the fact that we have already obtained $\Mxyz(\x+\y+\z)$, $M_x(\y+\z)$, $M_y(\x+\z)$ and $M_z(\x+\y)$. This concludes the proof of the lemma.
\end{proof}

\subsubsection{Proof of Proposition~\ref{prop:main:n=3} and concluding remarks}

Let us conclude the proof.

\begin{proof}[Proof of Proposition~\ref{prop:main:n=3}]
  In Lemma~\ref{lemma:n=3:check}, we checked that all our collections are indeed exceptional and strong. In Lemmas~\ref{lemma:prop:main:n=3:generates:a,b,c}, \ref{lemma:prop:main:n=3:generates:d} and \ref{lemma:prop:main:n=3:generates:e}, we showed that these collections generate the shifts $\kk(-i\x-j\y-k\z)$ for all $0\le i\le p-1$, $0\le j\le q-1$ and $0\le k\le r-1$ in all cases, and also the shift $\kk(\x+\y+\z)$ in the 3-loop case. Thus, by Lemma~\ref{lem:lwbar-elements:n=3}, the generation criterion of Proposition~\ref{prop:gen-criterion-invertibles-collections} and Remark~\ref{remark:generation-up-to-quasi-periodicity}, we see that the collections are full. This concludes the proof.
\end{proof}

\begin{remark}
Similarly to Remark~\ref{remark:quiver-compositions:n=2}, it is possible to explicitly describe the quivers defined by our collections on Figures~\ref{fig:n=3:split-split}--\ref{fig:n=3:3-loop} by making the generators out of standard short exact sequences of the given modules.

For example, for the morphisms between the shifts of $\kk$, we can use the same generators $e_x$, $e_y$ in the $x$- and $y$-directions as in Remark~\ref{remark:quiver-compositions:n=2}, and an analogous generator~$e_z$ for the morphisms in $z$-direction. The relations between these can be described as:
\vspace{-1ex}
\begin{gather*}
e_x^2 = e_y^2 = e_z^2 = 0,\qquad  e_y e_x = -e_x e_y \ne 0,\qquad e_z e_x = -e_x e_z \ne 0,\\ e_z e_y = -e_y e_z \ne 0,\qquad e_x e_y e_z \ne 0.
\end{gather*}

Let us also see what happens with the morphisms in the bottom-left-front corner of the collection for the 3-loop case on Figure~\ref{fig:n=3:3-loop}. There we have arrows $f_x\:M_x\to\kk$, $f_y\:M_y\to\kk$ and $f_z\:M_z\to\kk$ defined as in Remark~\ref{remark:quiver-compositions:n=2}. We also have an arrow $g_x\:M_x(\z)\to M_z[1]$ defined by the short exact sequence $(0\to M_z\xrightarrow{z}M_{xz}(\z)\xrightarrow{1}M_x(\z)\to 0)$ and analogous arrows $g_y\:M_y(\x)\to M_x[1]$ and $g_z\:M_z(\y)\to M_y[1]$. The arrow $h_x$ from $M_{xyz}(\x+\y+\z)$ to $M_z(\y)[1]$ can be defined by the sequence $(0\to M_z(\y)\xrightarrow{xz} N(\x+\y+\z)\xrightarrow{1}M_{xyz}(\x+\y+\z)\to 0)$, where $N=\kk[x,y,z]/(xy,yz,zx^2)$. Similarly, we define $h_y\:M_{xyz}(\x+\y+\z)\to M_x(\z)[1]$ and $h_z\:M_{xyz}(\x+\y+\z)\to M_y(\x)[1]$. Then we can form non-zero compositions $f_y g_z h_x$, $f_x g_y h_z$ and $f_z g_x h_y$ from $M_{xyz}(\x+\y+\z)$ to $\kk[2]$. These three morphisms generate the 2-dimensional space of morphisms from $M_{xyz}(\x+\y+\z)$ to $\kk[2]$ and are related by a single linear relation:
\vspace{-1ex}
$$ f_y g_z h_x + f_x g_y h_z + f_z g_x h_y = 0. $$
\end{remark}

\begin{remark}
\label{remark:geometric-3}
  Similarly to Remark~\ref{remark:geometric-2}, there are straightforward counterparts for our modules in the categories of equivariant sheaves over $\Aff3$. Namely, the module $\kk$ again corresponds to the structure sheaf of the origin, $M_x$, $M_y$ and $M_z$ to the structure sheaves of $x$-, $y$- and $z$-axes respectively, $M_{xy}$ to the structure sheaf of the union of $x$- and $y$-axes, and $M_{xyz}$ to the structure sheaf of the union of all three axes. Also, the modules are well-defined over $\kk[x,y,z]/(w)$ if and only if the corresponding supports lie on the hypersurface $\{w=0\}$ in $\Aff3$.
\end{remark}

\subsection{General pattern and higher dimensions}
\label{sec3:higher-n}

Here, we summarize observations on our explicit collections 
for $n\le 3$ constructed above and formulate a conjecture on how such collections should look like for arbitrary~$n$.

We see that the collections built for $n\le 3$ follow a certain pattern. First of all, in geometric language of Remarks~\ref{remark:geometric-2} and~\ref{remark:geometric-3}, all the modules used in our collections correspond to the structure sheaves of the origin (like $\kk$), of the coordinate axes (like $M_x$, $M_y$, $M_z$), or their unions (like $M_{xy}$ and $M_{xyz}$). In order for these objects to be well-defined in the corresponding categories, the corresponding coordinate axes must lie in the given hypersurface $\{w=0\}\subseteq \Affn$.

In the split cases, our collections can be described as Thom-Sebastiani products of the collections for the constituent polynomials (see Remark~\ref{remark:thom-sebastiani:2->3}). In general, such products can be defined as follows.

\begin{definition}
  Consider a \textit{split} invertible polynomial $w=w_1 + \ldots + w_m\in R$, where each $w_i$ is an invertible polynomial in the ring with its own variables $R_i=\kk[x_{i,1},\dots,x_{i,n_i}]$, $1\le i\le m$, and $ R = \otimes_i R_i =\kk[x_{i,j}]_{1\le i\le m, 1\le j\le n_i}$. Let the categories $\dsgof{L_{w_i}}{R_i/(w_i)}$ have exceptional collections $(E_{i,j})_j$ consisting of the shifts of modules over $R_i/(w_i)$. Then we define the \textit{Thom-Sebastiani product} of these collections as the collection in $\dsgof{L_w}{R/(w)}$, consisting of the tensor products $\bigotimes_i E_{i,j_i}$ of the corresponding modules over~$\kk$ (shifted accordingly).
\end{definition}

We also see that the non-strong collection for the 3-chain polynomial $w=x^p+xy^q+yz^r$ on Figure~\ref{fig:n=3:3-chain-not-strong} can be obtained as the union of the collections for the two split-chain polynomials $x^p+(y^q+yz^r)$ and $(x^p+xy^q)+z^r$. The first collection provides the shifts of $\kk$ and $M_z$, while the second one provides the same shifts of $\kk$ and the shifts of $M_y$. In order to achieve strongness, we also need to adjust the shift for one of the objects, namely we should replace $M_z(-(p-2)\x+\y)$ by $M_z(\y+\x)$. Thus we obtain the strong collection depicted on Figure~\ref{fig:n=3:3-chain}.

Also, the collections for the loop polynomials can be obtained from the collections for the corresponding chain polynomials with an addition of one extra-object. For example, for the 2-loop polynomial $w=yx^p+xy^q$, the collection on Figure~\ref{fig:2loop} can be thought as the union of the collection from Figure~\ref{fig:2chain} for the corresponding chain polynomial $w=x^p+xy^q$ and of the symmetrical collection for $w=yx^p+y^q$. The first collection provides the shifts of $\kk$ and $M_y$, while the second one provides the same shifts of $\kk$ and the shifts of $M_x$. In the bottom-left corner, where two shifts of $M_x$ and $M_y$ collide, we need to replace them by a single object which is a shift of $M_{xy}$.
Similarly, for the 3-loop polynomial $w=zx^p+xy^q+yz^r$, the collection on Figure~\ref{fig:n=3:3-loop} can be obtained as the union of the collections for the corresponding chain polynomials $\underline{x^p}+xy^q+yz^r$, $zx^p+\underline{y^q}+yz^r$ and $zx^p+xy^q+\underline{z^r}$ with an addition of one extra-object which is a shift of $M_{xyz}$.

The above observations suggest that we can build exceptional collections for arbitrary invertible polynomials either by using the collections coming from the lower dimensions (in the split cases), or by using the collections for the reduced versions of the polynomial (in the chain and loop cases). Let us summarize these observations as the following conjecture.

\begin{conjecture}
\label{conj:collection-general-n}
    If $w\in\kk[x_1,\dots,x_n]$ is an invertible polynomial, then the category of singularities $\dsgof{L_w}{\kk[x_1,\dots,x_n]/(w)}$ admits a full strongly exceptional collection whose objects are the shifts of modules corresponding to the structure sheaves of coordinate subspaces in $\Affn$ (and their unions) lying in the hypersurface $\{w=0\}$ as follows:
    \begin{enumerate}[label=(\alph*)]
        \item In the \textbf{split} cases, the exceptional collection can be obtained from the Thom-Sebastiani products of the constituent collections. It will be strong, if the constituent collections were strong.
        \item In the \textbf{chain} case $w=x_1^{p_1}+\ldots+x_{n-2}x_{n-1}^{p_{n-1}}+x_{n-1} x_n^{p_n}$, where $n\ge 3$, the exceptional collection can be obtained as the union of the collections for the two split-chain polynomials $(x_1^{p_1}+\ldots+x_{n-2}x_{n-1}^{p_{n-1}})+ x_n^{p_n}$ and $(x_1^{p_1}+\ldots+x_{n-3}x_{n-2}^{p_{n-2}})+(x_{n-1}^{p_{n-1}}+x_{n-1} x_n^{p_n})$. In order to make it strong, one also needs to additionally shift some of the objects via $L_w$-grading.
        \item In the \textbf{loop} case $w=x_n x_1^{p_1}+\ldots+x_{n-2}x_{n-1}^{p_{n-1}}+x_{n-1} x_n^{p_n}$, where $n\ge 2$, the strong exceptional collection can be obtained from the union of the collections for the corresponding chain polynomials obtained by replacing one of $n$ monomials $x_{i-1} x_i^{p_i}$ by $x_i^{p_i}$. If $n$ is odd, then we need to add one extra-object corresponding to the structure sheaf of the union of the maximal coordinate subspaces of dimension $\frac{n-1}2$ lying in the hypersurface $\{w=0\}$. If $n$ is even, then we need to remove two objects corresponding to the structure sheaves of the coordinate subspaces $x_1 x_3 \dots x_{n-1}$ and $x_2 x_4\dots x_n$ of dimensions $\frac n2$ and replace them by the structure sheaf of the union of these two subspaces.
    \end{enumerate}
\end{conjecture}

\begin{remark}
  It is expected that the total number of objects in the full exceptional collection for $\dsgof{L_w}{R/(w)}$ should be equal to the Milnor number of the dual singularity (for example, see Conjecture~1.4 in \cite{hirano-ouchi}). This fact can be used to quickly assess whether a given exceptional collection is a valid candidate for being full or not.
\end{remark}

\begin{example}
\label{ex:4-chain}
  In the case of 4-chain polynomial $w=x^p+xy^q+yz^r+zt^s$, the strong exceptional collection should consist of:
  \begin{itemize}[itemsep=0ex, topsep=0.5ex, parsep=0ex]
      \item $(p-1)(q-1)(r-1)(s-1)$ shifts of $\kk$,
      \item $p(r-1)(s-1)$ shifts of $M_y=\kk[y]$,
      \item $(p-1)q(s-1)$ shifts of $M_z=\kk[z]$,
      \item $(p-1)(q-1)r$ shifts of $M_t=\kk[t]$,
      \item and $pr$ shifts of $M_{[yt]}=\kk[y,t]$,
  \end{itemize}
  \vspace{0.5ex}
for the total of $(pqrs-qrs+rs-s+1)$ objects (which agrees with the Milnor number of the dual polynomial~$x^py+y^qz+z^rt+t^s$). Here, the first corresponding split-chain polynomial~$(x^p+xy^q+yz^r)+t^s$ provides the shifts of~$\kk$, $M_y$ and~$M_z$. The second polynomial~$(x^p+xy^q)+(z^r+zt^s)$ provides the same shifts of~$\kk$ and~$M_y$, and the shifts of $M_t$ and $M_{[yt]}$, where the object $M_{[yt]}$ is obtained as the tensor product of the modules $M_y$ and $M_t$ for the constituent 2-chain polynomials $x^p+xy^q$ and $z^r+zt^s$. The object $M_{[yt]}$ also corresponds to the structure sheaf of the coordinate $yt$-plane, which is lying in $\{w=0\}$ because $w$ belongs to the ideal $(x,z)$ defining this plane.
\end{example}

\begin{example}
\label{ex:4-loop}
  In the case of 4-loop polynomial $w=tx^p+xy^q+yz^r+zt^s$, the strong exceptional collection should consist of:
  \begin{itemize}[itemsep=0ex, topsep=0.5ex, parsep=0ex]
      \item $(p-1)(q-1)(r-1)(s-1)$ shifts of $\kk$,
      \item $(q-1)(r-1)s$ shifts of $M_x=\kk[x]$,
      \item $p(r-1)(s-1)$ shifts of $M_y=\kk[y]$,
      \item $(p-1)q(s-1)$ shifts of $M_z=\kk[z]$,
      \item $(p-1)(q-1)r$ shifts of $M_t=\kk[t]$,
      \item $qs-1$ shifts of $M_{[xz]}=\kk[x,z]$,
      \item $pr-1$ shifts of $M_{[yt]}=\kk[y,t]$,
      \item and one additional shift of $M_{[xz][yt]}=\kk[x,y,z,t]/(xy,xt,yz,yt)$,
  \end{itemize}
  for the total of $pqrs$ objects (which agrees with the Milnor number of the dual polynomial~$x^py+y^qz+z^rt+t^sx$), where the object $M_{[xz][yt]}$ corresponds to the structure sheaf of the union of the coordinate $xz$- and $yt$-planes. This collection can be obtained from the union of the collections for the corresponding chain polynomials $\underline{x^p}+xy^q+yz^r+zt^s$, $tx^p+\underline{y^q}+yz^r+zt^s$, $tx^p+xy^q+\underline{z^r}+zt^s$ and $tx^p+xy^q+yz^r+\underline{t^s}$ by replacing two shifts of~$M_{[xz]}$ and~$M_{[yt]}$ with a shift of~$M_{[xz][yt]}$.
\end{example}

\begin{example}
  Similarly to the previous examples, in the case of $5$-chain polynomial $w = x^p+xy^q+yz^r+zt^s+tu^e$, the collection should consist of the shifts of the modules $\kk$, $M_y$, $M_z$, $M_t$, $M_u$, $M_{[yt]}$, $M_{[yu]}$ and $M_{[zu]}$. In the case of $5$-loop polynomials one should also add the shifts of $M_x$, $M_{[xz]}$, $M_{[xt]}$ and one extra-object corresponding to the structure sheaf of the union of the coordinate planes $xz$, $xt$, $yt$, $yu$ and $zu$.
\end{example}

\section{Block decompositions}
\label{sec:blocks}

Let us clarify the definitions used in \hyperref[conj:orlov]{Orlov's Conjecture}.

\begin{definition}
\label{def:block}
  We say that an exceptional collection $(E_1,\dots,E_n)$ is a \textit{block}, if all its objects are pairwise orthogonal, that is if $\bull\Ext(E_i,E_j)=0$ for any $i\ne j$.
\end{definition}

\begin{definition}
\label{def:n-blocks}
  We say that an exceptional collection \textit{consists of $n$ blocks}, if it has the form $ (\BB_1,\dots,\BB_n) $, where each subcollection $\BB_i$ is a block.
\end{definition}

We prove Theorem~\ref{thm2} by using the classification of invertible polynomials from Example~\ref{ex:invertible-classification-n<=3} and the explicit exceptional collections constructed in Section~3 for the corresponding 9 types of invertible polynomials for $n\le 3$.
More precisely, we prove the following.

\begin{theorem}[{=Theorem~\ref{thm2} for general~$\kk$}]
\label{theorem:orlov's-conjecture-for-n<=3}
  For the 9 different types of invertible polynomials $w$ from Example~\ref{ex:invertible-classification-n<=3}, the category $\dsgof{L_w}{\kk[x_1,\dots,x_n]/(w)}$ admits a full strongly exceptional collection consisting of at most $n+1$ blocks.
\end{theorem}

The idea of the proof is as follows. We notice that our explicit collections from Section~3 in all 9~cases already have a nice block structure. Indeed, due to them being arranged into an $n$-dimensional lattices and due to arrows going only in nonnegative directions, the blocks can be formed by uniting the objects lying in the layers/hyperplanes orthogonal to the vector $(1,1,\dots,1)$ (see the left picture in Example~\ref{example:block-mutations}). However, the number of blocks in this decomposition depends on the exponents $p_i$ of the polynomial~$w$ and can be arbitrarily large for the given~$n$.

In order to reduce the number of blocks in the above decomposition, we apply the well-chosen \textit{mutations} to our collections (see~\S\ref{sec:mutations} for the preliminaries on mutations). We apply them in a way that preserves both the strongness of the collections and the overall structure of the corresponding graph/quiver. Essentially, each mutation will just invert some arrows in the graph. After doing so, we will be able to rearrange the objects into the smaller number of blocks by uniting some of the older blocks together. (See Example~\ref{example:block-mutations}.)

In \S\ref{sec:strong-mutations}, we investigate how mutations act on the strong collections. In \S\ref{sec4:inverting-arrows}, we formulate the procedure of inverting the arrows. Then we proceed with the proof of Theorem~\ref{theorem:orlov's-conjecture-for-n<=3} in \S\ref{sec4:proof}.

\subsection{Mutations of strong collections}
\label{sec:strong-mutations}

\begin{lemma}
\label{lemma:mutation-of-strong-is-strong}
\begin{enumerate}[itemsep=0em, label={(\alph*)}, labelwidth=\widthof{ (a)}, leftmargin=!, topsep=1ex]
    \item If in the exceptional collection $(E,F_1,\dots,F_n)$, the subcollection $(F_1,\dots,F_n)$ is strong, then the mutated collection $(\LL_E F_1,\dots,\LL_E F_n)$ is strong as well.
    \item If the subcollection $(E_1,\dots,E_n)$ in the exceptional collection $(E_1,\dots,E_n,F)$ is strong, so is the mutated collection $(\RR_F E_1,\dots,\RR_F E_n)$.
\end{enumerate}
\end{lemma}
\begin{proof}
  Follows immediately from the functoriality property of Lemma~\ref{lemma:mutations-are-functors} and from the definition of strong collections.
\end{proof}

\begin{lemma}
\label{lemma:mutated-morphisms-basic}
  If $(E,F)$ is an $\Ext$-finite exceptional pair, then the morphisms in the mutated pairs $(\LL_E F, E)$ and $(F,\RR_F E)$ are as follows:
  $$ \bull\Ext(\LL_E F, E) = 
  \bull\Ext(F,\RR_F E) = (\bull\Ext(E,F))^*. $$
\end{lemma}
\begin{proof}
  Applying the functor $\Hom(-,E)$ to the distinguished triangle defining $\LL_E F$ and using the fact that $\bull\Ext(F,E) = 0$, we get the following equalities:
  \begin{multline*}
    \bull\Ext(\LL_E F, E) = \bull\Ext(\bull\Ext(E,F)\otimes E, E) = \text{[by Lemma~\ref{lemma:Ext-computation-with-tensor-products}]} = \\ = (\bull\Ext(E,F))^* \otimes \bull\Ext(E,E) = (\bull\Ext(E,F))^* \otimes \kk = (\bull\Ext(E,F))^*.
  \end{multline*}
  Similarly, one can prove that $\bull\Ext(F,\RR_F E) = (\bull\Ext(E,F))^*$ by applying the functor $\Hom(F,-)$ to the triangle defining $\RR_F E$.
\end{proof}

\begin{lemma}
  Let $(E,F,G)$ be an $\Ext$-finite exceptional collection such that $E$ is orthogonal to $G$. Then in the mutated collection $(E,G,\RR_G F)$, the morphisms are as follows:
  $$ \bull\Ext(E, \RR_G F) = \bull\Ext(E,F)[1]. $$
\end{lemma}
\begin{proof}
  The same as in Lemma~\ref{lemma:mutated-morphisms-basic}.
\end{proof}

\begin{lemma}
\label{lemma:strong-mutation:basic-1}
  If $(E,F,G)$ is a strong exceptional collection such that $E$ and $G$ are orthogonal, then the collection $(E,G[-1],\RR_G F[-1])$ is strong as well.
\end{lemma}
\begin{proof}
  Follows from the previous lemma and from Lemma~\ref{lemma:mutated-morphisms-basic}.
\end{proof}

\begin{lemma}
  Let $(F,G,H)$ be an $\Ext$-finite exceptional collection such that $G$ is orthogonal to $H$. Then in the mutated collection $(G,\RR_G F, H)$ the morphisms are as follows:
  $$ \bull\Ext(\RR_G F, H) = \bull\Ext(F,H)[-1]. $$
\end{lemma}
\begin{proof}
  The same as in Lemma~\ref{lemma:mutated-morphisms-basic}.
\end{proof}

\begin{lemma}
\label{lemma:strong-mutation:basic-2}
  If $(F,G,H)$ is a strong exceptional collection such that $G$ and $H$ are orthogonal, then the collection $(G[-1],\RR_G F[-1],H)$ is strong as well.
\end{lemma}
\begin{proof}
  Follows from the previous lemma and from Lemma~\ref{lemma:mutated-morphisms-basic}.
\end{proof}

\begin{remark}
\label{remark:shift-of-collection-preserves-properties}
  The shift functors $[i]$ applied to the collections preserve their properties of being exceptional, full and strong.
\end{remark}

\subsection{Inverting sink arrows}
\label{sec4:inverting-arrows}

Here we prove two technical facts about the strong exceptional collections. If we consider strong collections as the graphs (or quivers) in which the vertices are the objects of the collection and the arrows stand for the non-trivial morphisms between the corresponding objects, then our statements mean that in a strong collection we can simultaneously invert all the arrows going into \textit{sinks}, that is into the vertices with no outgoing arrows.

Firstly, we show that we can simultaneously invert all arrows going into a single sink.

\begin{lemma}
\label{lemma:mutations:main-lemma-auxiliary}
  Let $(\AA,\BB,E,\DD)$ be a strong exceptional collection consisting of three subcollections $\AA$, $\BB$, $\DD$ and one object $E$, such that $\AA$ is orthogonal to $\DD$ and $E$ is orthogonal to both $\AA$ and $\DD$. Then the mutated collection $(\AA, E[-1], \RR_E \BB [-1], \DD)$ is also strong.
\end{lemma}
\begin{proof}
  Follows from Lemmas~\ref{lemma:mutation-of-strong-is-strong}, \ref{lemma:strong-mutation:basic-1}, \ref{lemma:strong-mutation:basic-2} and Remark~\ref{remark:shift-of-collection-preserves-properties}.
\end{proof}

\begin{remark}
  Using graphs, we can illustrate this lemma by the following picture, where the objects of each subcollection are grouped together:
  $$ \begin{tikzcd}
   \AA \ar[d] & E  \\
   \BB \ar[ru] \ar[r] & \DD
  \end{tikzcd}
  \quad\;\;{\raisebox{-1.5ex}{\text{\Huge$\rightsquigarrow$}}}\;\;\quad
  \begin{tikzcd}
  \AA \ar[d] & E[-1] \ar[ld] \\
   \RR_E \BB[-1] \ar[r] & \DD
  \end{tikzcd}
  $$
 Here, $E$ is a sink object and the given mutation inverts all the arrows going from the objects of $\BB$ into $E$ while preserving all other morphisms/arrows both between the groups~$\AA$, $\BB$ and~$\DD$ and inside them.
\end{remark}

Now, we just apply the previous fact to simultaneously invert all the arrows going into a block consisting of sinks.

\begin{lemma}
\label{lemma:mutations:main-lemma}
  Let $(\AA,\BB,\CC)$ be a strong $\Ext$-finite exceptional collection consisting of three subcollections $\AA$, $\BB$ and $\CC$. Assume that $\AA$ is orthogonal to $\CC$ and that the objects in $\CC$ are pairwise orthogonal. Then, the mutated exceptional collection $(\AA,\CC[-1],\RR_\CC \BB [-m])$ is strong as well, where $m$ is the number of objects in $\CC$.
\end{lemma}
\begin{proof}
  Let us prove this statement by induction on the number of objects in the collection $\CC$. If $\CC$ is empty, the statement is trivial. If $m\ge 1$, let us denote the objects as $\CC = (E_1,\dots,E_m)$. The right mutation over $\CC$ then can be done in two steps by firstly mutating over $E_1$ and then over the tail $\DD :=(E_2,\dots,E_m)$. Applying Lemma~\ref{lemma:mutations:main-lemma-auxiliary} for $\AA$, $\BB$, $E=E_1$ and $\DD$, we see that the exceptional collection $(\AA,E_1[-1],\RR_{E_1}\BB[-1],\DD)$ obtained after the first step is strong. Noting that $\AA':=(\AA,E_1[-1])$, $\BB':=\RR_{E_1}\BB[-1]$ and $\CC':=\DD$ satisfy the initial conditions of our lemma and that $\DD$ has $m-1$ elements, we can apply induction and conclude the proof.
\end{proof}

\begin{remark}
  This statement can be illustrated by the following picture:
  $$
  \begin{tikzcd}
   \AA \ar[d] & \CC  \\
   \BB \ar[ru] & 
  \end{tikzcd}
  \quad\;\;{\raisebox{-1.5ex}{\text{\Huge$\rightsquigarrow$}}}\;\;\quad
  \begin{tikzcd}
  \AA \ar[d] & \CC[-1] \ar[ld] \\
   \RR_\CC \BB [-m] &
  \end{tikzcd}
  $$
  Here, $\CC$ consists of the sink objects (therefore, does not have any arrows inside the group) and the mutation inverts all the incoming arrows for $\CC$, while preserving all other morphisms/arrows inside and between the groups~$\AA$ and $\BB$.
\end{remark}

\subsection{The proof}
\label{sec4:proof}

Let us firstly prove the following technical lemma.

\begin{lemma}
\label{lemma:decrease-number-of-blocks}
  Let an $\Ext$-finite strong exceptional collection consist of $N$ blocks $(\BB_1,\dots,\BB_N)$ and satisfy an additional condition that the blocks $\BB_i$ and $\BB_j$ are orthogonal as soon as $\left|i-j\right|>n$. Then, there is a mutation of this collection making it into a strong exceptional collection consisting of at most $n+1$ blocks.
\end{lemma}
\begin{proof}
  If $N \le n+1$, then this holds automatically. Otherwise, let us denote by $\AA:=(\BB_1,\dots, \BB_{N-n-1})$, $\BB:=(\BB_{N-n},\dots,\BB_{N-1})$ and $\CC:=\BB_N$, the subcollections consisting of the first $(N-n-1)$ blocks, the next $n$ blocks, and the last block respectively. Then our condition implies that $\AA$ is orthogonal to $\CC$. This means that we can apply Lemma~\ref{lemma:mutations:main-lemma} to get another strong exceptional collection, namely $(\AA,\CC[-1],\RR_\CC \BB[-m])$, where $m$ is the number of objects of $\CC=\BB_N$. The latter collection also consists of $N$ blocks as follows:
  $$ (\BB_1,\dots,\BB_{N-n-1},\quad\BB_{N}[-1],\quad\RR_{\BB_N} \BB_{N-n}[-m],\dots,\RR_{\BB_N} \BB_{N-1}[-m]). $$
  However, we notice that after this mutation, the orthogonal blocks $\BB_{N-n-1}$ and $\BB_N$ became adjacent, so we can unite them into one block:
  $$ (\BB_1,\dots,\BB_{N-n-2},\quad\BB_{N-n-1}\cup\BB_{N}[-1],\quad\RR_{\BB_N} \BB_{N-n}[-m],\dots,\RR_{\BB_N} \BB_{N-1}[-m]). $$
  The resulting collection consists of $N-1$ blocks. Moreover, we can check that it still satisfies the original orthogonality condition. This follows from the fact that if $\BB_N$ and $\BB_{N-i}$ were both orthogonal to some block $\BB_j$ for $i\le n$ and $j\le N-n-1$, then $\RR_{\BB_N} \BB_{N-i}$ will be also orthogonal to $\BB_j$. Thus, we can conclude the proof by induction on $N$.
\end{proof}

Now we are ready to prove our main statement.

\begin{proof}[Proof of Theorem~\ref{theorem:orlov's-conjecture-for-n<=3}]
\def\eps{\varepsilon}
  In Example~\ref{ex:main:n=1} and Propositions~\ref{prop:main:n=2} and~\ref{prop:main:n=3}, we have explicitly described full strongly exceptional collections in $\dsgof{L_w}{R/(w)}$ for the given types of $w$. These collections are arranged into $n$-dimensional lattices, where we can think of the objects as sitting at the points with integer coordinates $(a_1,\dots,a_n)\in\ZZ^n$ and where the non-zero morphisms from the object at $(a_1,\dots,a_n)$ may only go to the objects at $(a_1+\eps_1,\dots,a_n+\eps_n)$ for $\eps_i\in\{0,1\}$.
  Then, for each $b\in\ZZ$, let us consider the objects lying on the hyperplane $\{(a_1,\dots,a_n)\in\ZZ^n\mid \sum_i a_i = b\}$. These objects are pairwise orthogonal, so we can arrange them into a block which we denote by $\BB_b$.
  
  After this, we see that our exceptional collection consists of the blocks $(\BB_c,\BB_{c+1},\dots, \BB_{d-1},\BB_d)$, where the integers $c$ and $d$ are chosen in such a way that all blocks $\BB_i$ are empty for $i$ outside of the interval $[c,d]$. Moreover, we notice that the non-zero morphisms between the blocks may exist only from the objects of $\BB_b$ to the objects of $\BB_{b+i}$ for $0\le i\le n$. Thus, we can apply Lemma~\ref{lemma:decrease-number-of-blocks} to mutate this collection into the collection consisting of $n+1$ blocks only. The resulting strong exceptional collection will be full as well, since the original collection was such. Hence the proof.
\end{proof}

\begin{example}
\label{example:block-mutations}
  Let us demonstrate how the above procedure works for the 2-dimensional loop-type polynomial $w = yx^4 + xy^4\in \kk[x,y]$. Firstly, we decompose the strong exceptional collection from Figure~\ref{fig:2loop} into $7$ blocks as on the left picture below. Then at each step, we apply Lemma~\ref{lemma:mutations:main-lemma} for inverting the arrows going into the last block and after that we unite it with the block whose index is less by 3. So, firstly we invert arrows going into the 7-th block and unite it with the 4-th block, then we invert arrows going into the 6-th block and unite it with the 3-th block and so on, until we get only 3 blocks left. The final configuration is described on the right picture below (the arrows which ended up being inverted are dashed).
  $$ 
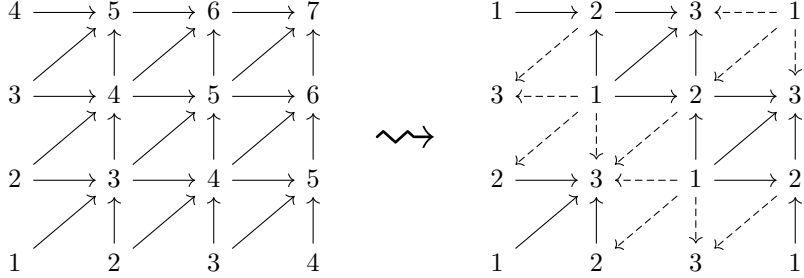
\begin{tikzcd}
   4 \ar[r] & 5 \ar[r] & 6 \ar[r] & 7 \\
   3 \ar[r] \ar[ru] & 4 \ar[r] \ar[u] \ar[ru] & 5 \ar[r] \ar[u] \ar[ru] & 6  \ar[u] \\
   2 \ar[r] \ar[ru] & 3 \ar[r] \ar[u] \ar[ru] & 4 \ar[r] \ar[u] \ar[ru] & 5 \ar[u] \\
   1 \ar[ru] & 2 \ar[u] \ar[ru] & 3 \ar[u] \ar[ru] & 4 \ar[u]
  \end{tikzcd}
  \quad\;\;{\raisebox{-1.5ex}{\text{\Huge$\rightsquigarrow$}}}\;\;\quad
  \begin{tikzcd}
   1 \ar[r] & \ar[ld, dashed] 2 \ar[r] & 3 & \ar[l, dashed] \ar[ld, dashed] \ar[d, dashed] 1 \\
   3 & \ar[l, dashed] \ar[ld, dashed] \ar[d, dashed] 1 \ar[r] \ar[u] \ar[ru] & \ar[ld, dashed] 2 \ar[r] \ar[u] & 3 \\
   2 \ar[r] & 3 & \ar[l, dashed] \ar[ld, dashed] \ar[d, dashed] 1 \ar[r] \ar[u] \ar[ru] & \ar[ld, dashed] 2 \ar[u] \\
   1 \ar[ru] & 2 \ar[u] & 3 & 1 \ar[u]
  \end{tikzcd}
  $$
\end{example}

\appendix

\section{Exceptional collections}

In this appendix, we recall the necessary definitions and facts about exceptional collections used throughout the paper. We are working in a triangulated category $\TT$ defined over a field $\kk$. All subcategories are assumed to be full and triangulated.

\subsection{Definitions and notations}

\begin{definition}
  An object $E$ in $\TT$ is called \textit{exceptional} if $\Hom_\TT(E,E) = \kk$ and $\Hom_\TT(E,E[i]) = 0$ for $i \ne 0$. Similarly, a collection of objects $(E_1,\dots,E_n)$ in $\TT$ is called \textit{exceptional}, if all $E_i$ are exceptional and $\Hom_\TT(E_i,E_j[k]) = 0$ for $i>j$ and $k \in \ZZ$.
\end{definition}

\begin{definition}
\label{def:full/strong-collection}
  An exceptional collection $(E_1,\dots,E_n)$ in $\TT$ is called:
  \begin{enumerate}[label=(\alph*), itemsep=0ex]
      \item \textit{full}, if it generates the whole category $\TT$, that is $ \TT = \< E_1, \dots, E_n \> $;
    \item \textit{strong}, if $\Hom_\TT(E_i,E_j[k])=0$ for $k \ne 0$ and all $i,j$.
  \end{enumerate}
\end{definition}

\begin{definition}
  Let $(X_\alpha)_{\alpha\in\mathcal{I}}$ be some collection of objects (or subcategories) of $\TT$. We say that an object $Y\in \TT$ is \textit{generated} by this collection, if $Y$ belongs to the minimal triangulated subcategory of $\TT$ containing all of $X_\alpha$. The latter subcategory is also denoted as $ \<X_\alpha\>_{\alpha \in \mathcal{I}} $.
\end{definition}

\begin{remark}
  That any given object $Y$ of $\TT$ is generated by the given collection $(X_\alpha)_{\alpha\in\mathcal{I}}$ can be checked explicitly. For that we just need to sequentially enlarge the collection by adding the shifts of its objects and the cones of its morphisms until we obtain $Y$.
\end{remark}

\begin{definition}
  Two objects $X$ and $Y$ are called \textit{orthogonal}, if $\Hom(X,Y[k]) = 0$ and $\Hom(Y,X[k])=0$ for all $k\in\ZZ$.
\end{definition}


\begin{definition}
  By a \textit{splitting of an idempotent} $e\:X\to X$, $e^2=e$, $X\in\TT$, we mean an isomorphism $X\simeq Y\oplus Z$ in $\TT$, under which $e$ is presented as the matrix $\left(\begin{smallmatrix} \id_Y & 0 \\ 0 & 0\end{smallmatrix}\right)\:Y\oplus Z \to Y\oplus Z$.
\end{definition}

\begin{definition}
  A category $\TT$ is called \textit{idempotent complete} (or \textit{Karoubian}) if every idempotent in $\TT$ splits. For any triangulated category $\TT$, we denote by $\overline{\TT}$ its \textit{idempotent completion}, that is a ``minimal'' idempotent complete triangulated category containing $\TT$. (See \cite{idemp-completion} for the details.)
\end{definition}

\begin{definition}
  We will say that a triangulated subcategory $\AA$ of $\TT$ is \textit{thick}, if for any $X$ and $Y$ in $\TT$, $X\oplus Y\in\AA$ implies that $X, Y\in\AA$.
\end{definition}


\begin{definition}
  We will denote by $\bull\Ext_\TT(X,Y) := \Hom_\TT(X,Y[\bullet])$ a $\ZZ$-graded vector space with the components defined as\\[-2ex]
  $$\Ext_\TT^k(X,Y) := \Hom_\TT(X,Y[k]).$$
\end{definition}

\begin{definition}
  A category $\TT$ is called \textit{$\Ext$-finite}, if $\bigoplus_{k\in\ZZ}\Hom_\TT(X,Y[k])$ is a finite-dimensional vector space for all $X$ and $Y$ in $\TT$. Similarly, a collection $(X_1,\dots,X_n)$ of objects in $\TT$ is called \textit{$\Ext$-finite}, if $\bigoplus_{k\in\ZZ}\Hom_\TT(X_i,X_j[k])$ is a finite-dimensional vector space for all $i$ and $j$.
\end{definition}

\subsection{Important statements}

\newcommand{\rorthog}[1]{{#1}^\bot}
\newcommand{\lorthog}[1]{{}^\bot {#1}}

\begin{lemma}\label{lemma:except-collect=>idemp-complete}
  Let $\TT$ be an $\Ext$-finite triangulated category admitting a full exceptional collection. Then $\TT$ is idempotent complete.
\end{lemma}
\begin{proof}
  Let $(E_1,\dots,E_n)$ denote the given exceptional collection. Consider the idempotent completion $\overline{\TT} \supseteq \TT$. Then, $(E_1,\dots,E_n)$ forms an exceptional collection in $\overline{\TT}$ as well, though \textit{a priori} it is not full in $\overline{\TT}$. By \cite[Theorem~3.2(a)]{bondal}, the subcategory $\TT = \<E_1,\dots,E_n\>$ is admissible in $\overline{\TT}$. Thus we have a decomposition $ \overline{\TT} = \<\rorthog{\TT}, \TT \>$. However, it follows from definitions that $\rorthog\TT=0$ in $\overline{\TT}$. Hence, $\overline{\TT}=\TT$, that is $\TT$ is idempotent complete.
\end{proof}

\begin{lemma}
\label{lemma:exc-collection:summand-thick=>full}
  Let $(E_1,\dots,E_n)$ be an $\Ext$-finite exceptional collection in a triangulated category $\TT$. Assume that every object in $\TT$ is a direct summand of an object generated by $(E_1,\dots,E_n)$. Then the collection is full:\\[-1ex]
  $$ \TT = \la E_1, \dots, E_n \ra. $$
\end{lemma}
\begin{proof}
  Consider the subcategory $\AA:=\<E_1,\dots,E_n\>$ in $\TT$. That the collection $(E_1,\dots,E_n)$ is $\Ext$-finite implies that $\AA$ is $\Ext$-finite as well. By definition, $(E_1,\dots,E_n)$ forms a full exceptional collection in $\AA$, hence $\AA$ is idempotent complete by Lemma~\ref{lemma:except-collect=>idemp-complete}. This also automatically implies that $\AA$ is a thick subcategory in $\TT$. Then our initial assumption means that $\AA = \TT$, that is that our collection is full in $\TT$.
\end{proof}

\subsubsection{Reconstruction from strong collections}

Starting from as far as Bondal in \cite[Theorem~6.2]{bondal}, people were trying to describe triangulated categories in terms of their exceptional collections. Different variations of such reconstruction statements exist in the literature. For example, the argument in \cite[Corollary~1.9]{orlov-reconstruction} shows that as soon as a triangulated category admits a \dg enhancement
, it can be uniquely recovered from the data of its full strongly exceptional collection. There is a stronger version of this statement stating that not only the triangulated category but also its \dg enhancement can be uniquely recovered from such data. Here, we give an explicit proof to this well-known statement.

\begin{prop}
\label{prop:strong-coll=>unique-dg-enhanc}
  Let $\TT$ be an idempotent complete triangulated category admitting a full strongly exceptional collection. Then any \dg enhancement of $\TT$ is unique and can be determined from the morphisms of the collection and their compositions.
\end{prop}
\begin{proof}
  Let $(E_1,\dots,E_n)$ denote the given full strongly exceptional collection in~$\TT$ and let $\TT^\dgonly$ be some \dg enhancement of $\TT$. Also, let us denote by~$\AA$ and~$\AA^\dgonly$ the full subcategories of $\TT$ and $\TT^\dgonly$ respectively, consisting of the objects $E_i$ only. We will also treat $\AA$ as a \dg category whose morphisms are situated in degree~$0$ and have trivial differentials.
  
  By the minimality theorem of Kadeishvili (see Theorem~1 in \cite{kadeishvili}), there is a minimal \ainf model $\BB$ on $\Hb(\AA^\dgonly)$ which is quasi-equivalent to $\AA^\dgonly$. 
  We also have $\Hb(\AA^\dgonly)=\HH^0(\AA^\dgonly)$ due to the strongness of the collection. Thus, all higher \ainf products of $\BB$ are trivial for the degree reasons. In other words, $\BB=\AA$ considered as \dg categories, so $\AA$ is quasi-equivalent to $\AA^\dgonly$.
  
  Now, since $\TT$ is generated by $\AA$, we can use results and terminology of \cite{bfk13} to show that $\TT^\dgonly$ can be recovered from $\AA^\dgonly$, hence from~$\AA$. Indeed, $\TT^\dgonly$ is a pre-thick \dg category, because $\TT$ is idempotent complete by assumption, hence it is automatically thick in $\D[(\TT^\dgonly)^{\textrm{op}}\textrm{-Mod}]$. Then by Lemma~2.29 of \cite{bfk13} applied for the subcategory $\AA^\dgonly$ in $\TT^\dgonly$, we have the quasi-equivalence $\TT^\dgonly\simeq \widehat{(\AA^\dgonly)}_{\textrm{pe}}$, where the latter category is the category of perfect \dg modules over $\AA^\dgonly$. However, the operation $\widehat{(-)}_{\textrm{pe}}$ is a homotopical invariant of \dg categories (see, for example, Proposition~4.4.2 in \cite{toen-dg-lectures} or Proposition~2.6 in \cite{orlov-dg}). Hence, $\AA^\dgonly\simeq\AA$ implies that $\widehat{(\AA^\dgonly)}_{\textrm{pe}} \simeq \widehat{\AA}_{\textrm{pe}}$. Therefore, the \dg enhancement $\TT^\dgonly\simeq \widehat{\AA}_{\textrm{pe}}$ is determined uniquely up to quasi-equivalence by the data of $\AA$ alone, the latter consisting of the morphisms of the collection $(E_1,\dots,E_n)$ and their compositions. Hence the proof.
\end{proof}

\begin{remark}
  The condition of being idempotent complete holds automatically for the $\Ext$-finite collections/categories by Lemma~\ref{lemma:except-collect=>idemp-complete}. In particular, it holds for the maximally graded categories of singularities of invertible polynomials studied in this paper, since the corresponding explicit collections are all $\Ext$-finite. However, it seems that this condition can be removed altogether by replacing $\widehat{(-)}_{\textrm{pe}}$ with another invariant of \dg categories, the \textit{pretriangulated hull} (see definition in \cite{orlov-dg}), and by adjusting the argument in \cite{bfk13} accordingly.
\end{remark}

\subsection{Mutations}
\label{sec:mutations}

In this section, we recall the notion of mutations of exceptional collections following \cite{bondal}. We also provide few extra-details necessary for our computations in Section~4. The basic facts here are Lemma~\ref{lemma:mutated-collections} stating that the class of exceptional collections is closed under mutations, and Lemma~\ref{lemma:mutations-are-functors} stating that the mutations are functorial. We continue working in a $\kk$-linear triangulated category~$\TT$.

\begin{definition}
  Let $\bull V$ be a finite-dimensional $\ZZ$-graded vector space over $\kk$ and $X$ an object of $\TT$. We define the tensor product $\bull V\otimes X$ as the following object in $\TT$:
  \vspace{-1ex}
   $$ \bull V \otimes X := \bigoplus_{i\in\ZZ} X^{\oplus \dim V^i} [-i]. $$
\end{definition}

\begin{remark}
\label{remark:def-of-otimes-by-vector-spaces}
  In this definition, it is crucial to assume that $\bull V$ is finite-dimensional. Indeed, the infinite direct sums might not exist in an arbitrary triangulated category, so the object $\bull V\otimes X$ might not be well-defined without this assumption.
\end{remark}

\begin{definition}
  For a finite-dimensional $\ZZ$-graded vector space $\bull V$, the dual $\ZZ$-graded vector space $(\bull V)^*$ is defined by setting its $i$-th component to be
  $$ \Bigl((\bull V)^*\Bigr)^i := (V^{-i})^* = \Hom_\kk(V^{-i},\kk). $$
\end{definition}

\begin{lemma}
\label{lemma:Ext-computation-with-tensor-products}
For any finite-dimensional $\ZZ$-graded vector space $\bull V$ and any two objects $X$ and $Y$ in $\TT$, we have the following isomorphisms of the graded vector spaces:
\vspace{-2ex}
\begin{align*}
  \bull\Ext_\TT(X, \bull V \otimes Y) & = \bull V \otimes \bull\Ext_\TT(X,Y),\\
  \bull\Ext_\TT(\bull V \otimes X, Y) & = (\bull V)^* \otimes \bull\Ext_\TT(X,Y),
\end{align*}
where the tensor product of the graded vector spaces is defined as $(\bull V \otimes \bull W)^n := \bigoplus_{i+j=n} V^i\otimes W^j$.
\end{lemma}


\begin{definition}
  For an $\Ext$-finite pair of objects $(X,Y)$ in $\TT$, let us define the \textit{evaluation map}
  $$\ev\: \bull\Ext(X,Y)\otimes X \to Y$$
  as follows. Let us fix a homogeneous $\kk$-basis $(e_j)$ in a graded vector space $\bull V := \bull\Ext(X,Y)$. Let us identify the basis elements of degree $i$ in $\bull V$ with the summands of type $X[-i]$ in the definition of $\bull V \otimes X$. Thus, a basis element $e_j\in V^i=\Hom(X,Y[i])$ corresponds to a summand $X[-i]$. We define the restriction of $\ev$ on the latter summand to be $e_j[-i]\:X[-i]\to Y$.
\end{definition}

\begin{definition}
  Similarly to the previous definition, for an $\Ext$-finite pair of objects $(X,Y)$ in $\TT$, we can define a \textit{coevaluation map}
  $$\coev\:X \to (\bull\Ext(X,Y))^* \otimes Y$$
  as follows. For a basis $(e_j)$ defined for $\bull V:=\bull\Ext(X,Y)$ as in the previous definition, we consider a dual basis $(e_j^*)$ in $(\bull V)^*$. For $e_j\in V^i = \Hom(X,Y[i])$, the dual basis element $e_j^*$ has degree $-i$ and we identify it with a summand $Y[i]$ in $(\bull V)^* \otimes Y$. We define the projection of $\coev$ onto this summand to be exactly $e_j\:X\to Y[i]$.
\end{definition}


\begin{definition}
\label{definition:mutation}
  Let $(E,F)$ be an $\Ext$-finite exceptional pair in $\TT$. We define the \textit{left mutation} $\LL_E F$ of $F$ over $E$ and the \textit{right mutation} $\RR_F E$ of $E$ over $F$ as the objects of $\TT$ inscribed in the following distinguished triangles:
  $$ \LL_E F \to \bull\Ext(E,F) \otimes E \xrightarrow{\ev} F \to \LL_E F[1],$$
  $$ \RR_F E[-1] \to E \xrightarrow{\coev} (\bull\Ext(E,F))^* \otimes F \to \RR_F E. $$
\end{definition}

\begin{remark}
  In the original definition in \cite{bondal}, the $\Ext$-finiteness condition is not mentioned, but is used implicitly (see Remark~\ref{remark:def-of-otimes-by-vector-spaces}).
\end{remark}

\begin{example}
  Let $(E,F)$ be an $\Ext$-finite exceptional pair. If $E$ and $F$ are orthogonal, then $\LL_E F \simeq F[-1]$ and $\RR_F E \simeq E[1]$.
\end{example}

Now let us formulate the basic facts about mutations.

\begin{lemma}[see \S2 in \cite{bondal}]
\label{lemma:mutated-collections}
  For an $\Ext$-finite exceptional pair $(E,F)$ in $\TT$, the mutated pairs $(F,\RR_F E)$ and $(\LL_E F, E)$ are also exceptional.
  More generally, for an $\Ext$-finite exceptional collection $(E_1,\dots,E_i,E_{i+1},\dots,E_n)$ in $\TT$, the following mutated collections are also exceptional and generate the same subcategory in $\TT$:
  \vspace{-1ex}
  $$ (E_1,\dots,E_{i-1},\quad E_{i+1},\;(\RR_{E_{i+1}} E_i),\quad E_{i+2},\dots, E_n),$$
  $$ (E_1,\dots,E_{i-1},\quad (\LL_{E_i} E_{i+1}),\; E_i,\quad E_{i+2},  \dots, E_n).$$
\end{lemma}


\begin{lemma}[Functoriality]
\label{lemma:mutations-are-functors}
  (a) If $(E,F_1,\dots,F_n)$ is an $\Ext$-finite exceptional collection, then the left mutation $\LL_E$ extends to an equivalence functor of the following subcategories:
  \vspace{-1ex}
  $$ \LL_E\: \<F_1,\dots,F_n\> \to \<\LL_E F_1,\dots,\LL_E F_n\>. $$
  In particular, $\bull\Ext(\LL_E F_i, \LL_E F_j) = \bull\Ext(F_i,F_j)$ for all $1\le i,j\le n$.
  
  (b) If $(E_1,\dots,E_n,F)$ is an $\Ext$-finite exceptional collection, then the right mutation $\RR_F$ extends to an equivalence functor
  \vspace{-1ex}
  $$ \RR_F\: \<E_1,\dots,E_n\> \to \<\RR_F E_1,\dots, \RR_F E_n\>, $$
  and $\bull\Ext(\RR_F E_i, \RR_F E_j) = \bull\Ext(E_i,E_j)$ for all $i,j$.
\end{lemma}
\begin{proof}
  To prove (a), we can consider the subcategory $\AA = \<E,F_1,\dots,F_n\>$ generated by the given exceptional collection. This will be an $\Ext$-finite category and we can use Definition~\ref{definition:mutation} to define $\LL_E X$ for any object $X$ generated by $(F_1,\dots,F_n)$ in $\AA$. The argument in the proof of Theorem~3.2(a) in \cite{bondal} shows that the so-defined $\LL_E$ becomes a functor of triangulated subcategories. Item (b) is proved similarly.
\end{proof}

\printbibliography[heading=bibintoc]

\bigskip
\begin{tabular}{@{}l@{}}%
    Department of Mathematics, Columbia University\\
    2990 Broadway\\
    New York, NY 10027, USA\\[0.5ex]
    \textit{e-mail}: \href{mailto:okravets@math.columbia.edu}{\nolinkurl{okravets@math.columbia.edu}}
  \end{tabular}

\end{document}